\documentclass[12pt]{article} 
\usepackage{amsmath, amssymb}
\usepackage{amsthm} 
\usepackage{url}
\usepackage{mathrsfs}
\usepackage{comment}
\usepackage{fancyhdr}
\usepackage{enumitem}

\newcommand{\iu}{\mathrm{i}}

\usepackage[
  setpagesize=false,
  bookmarks=true,
  bookmarksdepth=3,
  bookmarksnumbered=true,
  linkcolor=blue,
  citecolor=red,
  colorlinks=true,
  pdftitle={Heat Kernel and Closed Geodesic Asymptotics for Nilpotent Coverings},
  pdfsubject={Heat-kernel and prime-geodesic asymptotics for nilpotent coverings from exact rational fibers},
  pdfauthor={Atsushi Katsuda},
  pdfkeywords={nilpotent coverings, heat kernels, prime closed geodesics, Floquet--Bloch theory, spectral zeta functions}
]{hyperref}

\allowdisplaybreaks

\theoremstyle{plain} 
\newtheorem{theorem}{\indent\sc Theorem}[section]
\newtheorem{lemma}[theorem]{\indent\sc Lemma}

\newtheorem{proposition}[theorem]{\indent\sc Proposition}

\newtheorem{conditions}[theorem]{\indent\sc Conditions}

\theoremstyle{definition} 

\newtheorem{remark}[theorem]{\indent\sc Remark}

\newtheorem{problem}[theorem]{\indent\sc Problem}

\makeatletter
    
    \@addtoreset{equation}{section}

\def\address#1#2{\begingroup
\noindent\parbox[t]{7.8cm}{%
\small{\scshape\ignorespaces#1}\par\vskip1ex
\noindent\small{\itshape E-mail address}%
\/: #2\par\vskip4ex}\hfill%
\endgroup}%
\makeatother

\title{Heat Kernel and Closed Geodesic Asymptotics
for Nilpotent Coverings}

\author{%
\textbf{Atsushi Katsuda}$^{*}$
}
\date{}

\begin{document}
\markright{Heat Kernel and Closed Geodesic Asymptotics}
\pagestyle{myheadings}
\maketitle
\begin{abstract}
We establish all--order long--time asymptotic expansions for heat kernels on
nilpotent coverings and for prime closed geodesics in fixed central classes of
nilpotent quotients of compact hyperbolic surfaces.  The exact lattice-side
input is the finite-dimensional rational Floquet--Bloch theory of the companion
paper: rational Kirillov restrictions give exact finite-dimensional fibers,
and a generalized Pytlik functional gives exact Fourier-inversion and
normalized-trace identities.

At a rational parameter $p/q$ the decomposition is exact, and the fluctuation
of the fiber integrand is controlled only by $q$.  Hence the large-denominator
comparison with the smooth Kirillov or Schr\"odinger normal form is uniform on
the rational support of the Pytlik functional; irrational parameters do not
enter the rigorous trace argument.  For general nilpotent models,
coefficient-weighted spectral sums are justified to every fixed order by
positive Rockland estimates, the Plancherel--Mellin formula, and a trace-level
order-balance argument.

In contrast with approaches which usually give leading terms or integrated
Edgeworth-type asymptotics, the method gives genuinely local, pointwise
higher-order heat-kernel expansions.  The same representation-theoretic
quantity governs the leading term in the closed-geodesic asymptotics, producing
a nilpotent Chebotarev-type phenomenon.  The Heisenberg model is computed to
the first correction term, and the Engel model is represented through the
resolvent and heat-kernel calculus of the quartic oscillator.
\end{abstract}



\footnotetext[3]{
$^{*}$This work was supported by JSPS KAKENHI Grant Number JP24K06715
and the Research Institute for Mathematical Sciences,
an International Joint Usage/Research Center
located in Kyoto University.
}

\tableofcontents

\section{Introduction}

This paper studies heat kernels and prime closed geodesics on nilpotent
coverings.  The motivating problem is to extend the asymptotic theory known for
abelian coverings to nonabelian nilpotent deck groups.  In the abelian case,
Floquet--Bloch theory decomposes the analytic problem over the character torus,
and the long-time behavior is read from the perturbation theory of twisted
operators near the trivial character.  In the nilpotent case the same analytic
questions are natural, but the representation-theoretic parameter space is no
longer an ordinary torus.

The companion foundation paper~\cite{KatsudaFB} develops the replacement for
this missing abelian Bloch parameter space.  That theory was developed with the
present analytic problems in mind.  It gives a finite-dimensional rational
skeleton on the lattice side and a smooth Kirillov model on the Malcev-completion
side.  The former supplies the exact trace identities; the latter supplies a
convenient normal form for the perturbative calculations.

Classical Floquet--Bloch theory, recalled for instance in Kuchment~\cite{Kuchment},
works by decomposing periodic analytic problems over the character dual of the
covering group.  For nilpotent coverings this direct framework is not available:
a nonabelian finitely generated nilpotent lattice is non-type I, and its full
unitary dual is not a manageable Bloch parameter space.  The point here is not
to resolve the full dual.  Instead, we use the exact rational decomposition and
finite-dimensional trace formula of the companion paper, and then perform the
asymptotic analysis only after passing to the corresponding fiber operators.

The use of the Lie-group representation should be understood in this way.  The
Schr\"odinger and Kirillov representations provide smooth normal forms for
large-denominator rational fibers, while the exact trace identities remain
finite-dimensional.  At a rational parameter $p/q$ the fiber decomposition is
exact, and the fluctuation of the normalized integrand is controlled by the
denominator $q$, independently of the numerator and of the auxiliary Bloch
variables.  Bounded-denominator sets have zero Pytlik mass.  Consequently the
large-denominator comparison passes directly to the exact trace functional, and
irrational parameters play no role in the rigorous argument.  The geometric
rational functional evaluates the continuous model observables by ordinary
Kirillov--Fujiwara integration.  On the model side, the coefficient-weighted
spectral sums are justified below to every fixed order by positive Rockland
estimates, the Plancherel--Mellin formula, and a trace-level order-balance
argument.

The Harper operator is included as a model illustration rather than as an input
to the proofs.  In the rational magnetic case, the familiar finite-dimensional
Harper matrices are exactly the Heisenberg finite-dimensional fibers.  The
Wilkinson expansion is therefore an asymptotic statement inside exact finite
fibers, not a justification for their existence.  The detailed Harper discussion
is placed in Appendix~\ref{Anotherproofwilkinson}.

\medskip
\noindent
We use throughout the standard smooth-vector convention for unitary
representations of Lie groups.  If \(\pi:G\to U(\mathcal H)\) is a strongly
continuous unitary representation, then \(\mathcal H^\infty\) denotes the dense
space of smooth vectors, and all Lie-algebraic normal-form calculations are
performed on this common smooth domain.  In the Heisenberg case the notation
\(L^2(\mathbb R)\) always means the Schr\"odinger Hilbert space after the
central character has been fixed; equivalently, the center acts by scalars and
is not an additional \(L^2\)-variable.  More intrinsically, this is the induced
model attached to a polarization, not the regular space \(L^2(G)\).

\medskip
\noindent
We state the two principal results first.  The exact rational
representation-theoretic input and the analytic transfer and summability
statements used in their coefficient formulae are recorded immediately
afterward in Subsections~\ref{repinput} and
\ref{analytic-bridge-hypotheses}.

\subsection{Long--time asymptotics of heat kernels on covering manifolds}
\label{IntroHeat}

Let \(M\) be a compact Riemannian manifold, or a finite connected
non-bipartite graph, and let
\[
\pi:X\longrightarrow M
\]
be a normal covering with deck transformation group \(\Gamma\).
We denote by \(k_X(t,p,q)\) the heat kernel on \(X\), or the
transition probability of the simple random walk in the graph case.

\begin{problem}
Determine the asymptotic behavior of \(k_X(t,p,q)\) as \(t\to\infty\).
\end{problem}

When \(\Gamma\) is abelian, this problem has been studied
extensively.  Kotani and Sunada~\cite{KotaniSunada1} established full
long--time asymptotic expansions by combining classical Bloch theory
with geometric analysis of the Albanese torus.

\begin{theorem}[Kotani--Sunada]\label{abel-heat}
If \(\Gamma\) is free abelian of rank \(b\), then the heat kernel
admits a full asymptotic expansion
\[
k_X(t,p,q)
\sim
\frac{C_{\rm ab}(p,q)}{t^{b/2}}
\left(1+\frac{c_1(p,q)}{t}+\frac{c_2(p,q)}{t^2}+\cdots\right),
\qquad t\to\infty .
\]
With the Kotani--Sunada normalization, the diagonal leading constant is
\[
C_{\rm ab}(p,p)
=
\frac{\operatorname{vol}(M)^{b/2}
      \operatorname{vol}(\operatorname{Alb}^{\Gamma})^{b/2}}
     {(4\pi)^{b/2}\,\#\operatorname{Tor}(\Gamma)} .
\]
Here \(\operatorname{Alb}^{\Gamma}\) is the Albanese torus associated
with the covering, equipped with the Albanese metric.  The off-diagonal
constant and the higher coefficients are also expressed by the Albanese
geometry and the heat-kernel perturbation data.
\end{theorem}

The leading-order asymptotics in the abelian case trace back to
earlier works of Guivarc'h~\cite{Guivarch}, Jacod~\cite{Jacod},
Kr\'amli--Sz\'asz~\cite{Kramli}, and Sinai~\cite{Sinai}, and the
geometric origin of the leading coefficient was clarified by Kotani,
Shirai, and Sunada~\cite{KotaniShiraiSunada}.  These results rely
essentially on classical Bloch theory and on the commutative nature of
the covering group.

Beyond the abelian setting, the situation changes substantially.
Leading-order results for random walks and heat kernels associated
with nilpotent groups were obtained by
Alexopoulos~\cite{Alexopoulos,Alexopoulos1,Alexopoulos2,Alexopoulos3}
and Ishiwata~\cite{Ishiwata} via comparison with heat kernels on
stratified nilpotent Lie groups.  Related local limit and Edgeworth-type
results include work of Cr\'epel--Raugi, Raugi, Breuillard, Hough,
Ishiwata--Kawabi--Namba, and Namba~\cite{Crepel,Raugi,Breuillard,Hough,Ishiwata4,Ishiwata5,Namba};
for general heat-kernel background see also Saloff-Coste~\cite{Saloff}.  These works determine the leading
asymptotic behavior, both on and off the diagonal, through the
corresponding Lie group heat kernel.  However, a systematic framework
for obtaining genuinely local higher-order asymptotic expansions has
not been available in general.  Even for nilpotent Lie groups
themselves, local expansions beyond the leading term as \(t\to\infty\)
are not known in full generality, although integrated
Edgeworth-type expansions exist; see, for example, Pap~\cite{Pap}.

We now state the corresponding result for coverings with nilpotent
deck transformation groups.  Throughout, \(\Gamma\) denotes a
finitely generated torsion-free nilpotent group, \(G\) its Malcev
completion, \(H=\Gamma/[\Gamma,\Gamma]\) its abelianization,
\(b=\mathrm{rank}\,H\), and \(d\) the polynomial growth degree of
\(\Gamma\).

\begin{theorem}[Nilpotent heat-kernel expansion]
\label{conj-heat}
Let \(\Gamma\) be a finitely generated torsion-free nilpotent group, let
\(G\) be its Malcev completion, and let \(d\) be its polynomial growth degree.
For every \(N\geq0\),
\begin{equation}
 k_X(t,p,q)
 =t^{-d/2}\sum_{j=0}^{N}C_j(p,q)t^{-j/2}
 +o\bigl(t^{-d/2-N/2}\bigr).
\label{heat-expansion-N}
\end{equation}
Thus the heat kernel has an asymptotic expansion to all orders in the usual
sense that the expansion may be truncated after any prescribed finite order.
If the conjugation/parity symmetry of the localized model eliminates the odd
half-orders, only integral powers of \(t^{-1}\) occur.

With the notation introduced in Section~\ref{leadingnilpotent}, the leading
coefficient is
\[
C_0(p,q)
=
\Gamma(d/2)\operatorname{vol}(M)^{d/2}
\int_{\Theta}J(\theta)
\sum_i a_{0,i}(p,q;\theta)\,\mu_i(\theta)^{-d/2}\,d\theta .
\]
Here \(J(\theta)\) is the exact model density from
Proposition~\ref{model-density}, \(\mu_i(\theta)\) are the eigenvalues of the
positive regular model operator, and \(a_{0,i}\) are the leading local
amplitudes.  The sum and angular integral converge absolutely by
Theorem~\ref{all-order-model-summability}.  If the normalized model and
amplitude are angularly independent, the expression reduces to the
corresponding scalar multiple of \(\zeta_{\mathcal H}(d/2)\).
\end{theorem}

The proof keeps the exact lattice trace entirely on the rational
finite-dimensional support.  Proposition~\ref{rational-spectral-transfer}
identifies its high-denominator observables with the smooth normal form,
Proposition~\ref{model-density} evaluates the resulting continuous model
observable, and Theorem~\ref{all-order-model-summability} justifies all
coefficient-weighted spectral interchanges.

\subsection{Geometric analogue of the Chebotarev density theorem}
\label{IntroChebotarev}

Let \(M\) be a compact Riemann surface of genus \(g\) with constant
negative curvature \(-1\).  Let us first formulate the closed-geodesic
problem precisely.  Every nontrivial free homotopy class of oriented closed
curves on \(M\) contains a unique closed geodesic, and free homotopy classes
are naturally identified with conjugacy classes in \(\pi_1(M)\).  In
particular, prime closed geodesics correspond to primitive conjugacy classes.

Let
\[
\Phi:\pi_1(M)\twoheadrightarrow\Gamma
\]
be a surjective homomorphism onto a finitely generated discrete group
\(\Gamma\).  For a conjugacy class \(\alpha\) in \(\Gamma\), let
\(\pi(x,\Phi,\alpha)\) denote the number of prime closed geodesics \(\gamma\)
on \(M\) such that
\[
\ell(\gamma)\leq x
\qquad\text{and}\qquad
\Phi([\gamma])\subset\alpha,
\]
where \([\gamma]\) denotes the conjugacy class in \(\pi_1(M)\)
corresponding to \(\gamma\).  Equivalently, the image under \(\Phi\) of one,
and hence every, representative of \([\gamma]\) belongs to \(\alpha\).
When \(\Phi\) is the canonical projection onto a fixed quotient \(\Gamma\)
and no confusion can arise, we write simply
\(\pi(x,\alpha)=\pi(x,\Phi,\alpha)\).

\begin{problem}
Determine the asymptotic behavior of \(\pi(x,\Phi,\alpha)\) as
\(x\to\infty\).
\end{problem}

As in the heat-kernel problem, the essential input for our nilpotent answer is
the representation-theoretic Floquet--Bloch decomposition coming from exact
rational finite-dimensionalization.  No new structural ingredient is
required; the same canonical hypoelliptic operator governs the leading
constants.

The abelian benchmark is obtained from the canonical abelianization map
\[
\Phi_{\rm ab}:\pi_1(M)\twoheadrightarrow H_1(M,\mathbb Z).
\]
Since the target is abelian, each conjugacy class is a singleton.  Thus, for
\(\alpha\in H_1(M,\mathbb Z)\), the notation \(\pi(x,\alpha)\) counts the
prime closed geodesics \(\gamma\) with \(\ell(\gamma)\leq x\) and homology
class \([\gamma]=\alpha\).

\begin{theorem}[Phillips--Sarnak~\cite{Phillips}]
\label{abel-geodesic}
For every fixed \(\alpha\in H_1(M,\mathbb Z)\), the counting function
\(\pi(x,\alpha)\) admits a full asymptotic expansion as \(x\to\infty\).
In particular,
\[
\pi(x,\alpha)
\sim
\frac{C e^x}{x^{1+b/2}},
\qquad x\to\infty,
\]
where \(b=\mathrm{rank}\,H_1(M,\mathbb Z)=2g\), and
\[
C
=
\left(\frac{\mathrm{vol}(M)}{2\pi}\right)^{b/2}
\frac{1}{\mathrm{vol}(J(M))}
=
(g-1)^g.
\]
Here \(J(M)\) denotes the Jacobi torus.  The leading coefficient is
independent of \(\alpha\).
\end{theorem}

The leading term was also obtained independently by Katsuda and
Sunada~\cite{Katsuda1}.

Our nilpotent analogue is the following.  For \(\alpha\in Z(\Gamma)\),
the conjugacy class of \(\alpha\) is the singleton \(\{\alpha\}\), and we use
the same symbol for the element and its conjugacy class.

\begin{theorem}[Nilpotent Chebotarev-type expansion]
\label{theorem-geod}
Let \(M\) be a compact Riemann surface of constant curvature \(-1\), let
\(\Phi:\pi_1(M)\twoheadrightarrow\Gamma\) be a nilpotent quotient with
\(\Gamma\) finitely generated and torsion-free, and let
\(\alpha\in Z(\Gamma)\).  For every \(N\geq0\),
\begin{equation}
 \pi(T,\Phi,\alpha)
 =\frac{e^T}{T^{1+d/2}}
   \left(\sum_{j=0}^{N}C_j^{\rm geo}(\alpha)T^{-j/2}
   +o(T^{-N/2})\right).
\label{geodesic-expansion-N}
\end{equation}
Hence the fixed-central-class counting function has an all--order asymptotic
expansion.  If the parity symmetry eliminates the odd half-orders, only
integral powers of \(T^{-1}\) occur.  The leading coefficient is
\[
C_0^{\rm geo}(\alpha)
=
\frac12 A_{\rm Sel}\,\Gamma(d/2)\operatorname{vol}(M)^{d/2}
\int_{\Theta}J(\theta)
\sum_i a^{\rm geo}_{0,i}(\alpha;\theta)
\mu_i(\theta)^{-d/2}\,d\theta .
\]
The amplitude is the lowest-branch Selberg amplitude after the central element
has been extracted by normalized finite-dimensional traces.  The sum and
angular integral converge absolutely by
Theorem~\ref{all-order-model-summability}.  In the isotropic case this is the
corresponding scalar multiple of \(\zeta_{\mathcal H}(d/2)\).
\end{theorem}

Here \(A_{\rm Sel}\) records only the normalization convention for the Selberg
trace formula.  In the convention used in \eqref{Selbergtrace}, with
dimension-normalized traces and the hyperbolic factor \(2\sinh(\ell/2)\) in the
denominator, one has \(A_{\rm Sel}=1\).  If a different normalization of the
test transform or of the hyperbolic term is used, \(A_{\rm Sel}\) is the
explicit scalar converting that convention to the one used here.  The factor
\(1/2\) in the displayed constant is the standard weighted-to-unweighted
conversion in the prime-geodesic problem.

Thus the representation-theoretic quantity governing the long--time
heat kernel asymptotics reappears unchanged in the distribution of
closed geodesics.  In the abelian case, the corresponding operator is
the quadratic form on the character torus.  In the nilpotent case, it
is replaced by a canonical hypoelliptic operator on the Lie-theoretic
model determined by the Malcev completion.

The present work is related to the author's earlier announcement
``Asymptotics of heat kernels on nilpotent coverings and related topics''
\cite{Katsudainverse}.  That article should be read as an announcement and preliminary account.
The present paper separates the exact rational trace identities from the
smooth model calculation and supplies the rational-transfer and Rockland--Mellin
arguments needed to close the latter at every fixed order.  The main results
were also announced in~\cite{KatsudaRIMS}, and an earlier RIMS
K\^oky\^uroku article~\cite{KatsudaHeisRIMS} treated the special Heisenberg
case.  Material from the more extensive unpublished manuscript~\cite{Katsuda0}
which is needed for the model computations is incorporated here.

More general situations, such as variable negative curvature or
general Anosov flows, require substantially different dynamical
arguments and fall outside the scope of the present paper.  In the
abelian setting, asymptotic expansions of the prime geodesic counting
function have been established for compact Riemannian manifolds with
variable negative curvature by Anantharaman~\cite{Anantharaman1},
Pollicott and Sharp~\cite{Pollicott2}, and Kotani~\cite{Kotani1}.
Leading-term asymptotics in even more general settings are known for
prime closed orbits of weakly mixing Anosov flows on compact
manifolds, due to Lalley~\cite{Lalley}, Pollicott~\cite{Pollicott1},
Katsuda and Sunada~\cite{Katsuda2}, and Sharp~\cite{Sharp1}.  There
are also extensions to central limit theorems and large deviation
principles; see, for instance, Lalley~\cite{Lalley},
Babillot~\cite{Babillot}, and
Anantharaman~\cite{Anantharaman1,Anantharaman2}, as well as the
surveys~\cite{Parry3,Margulis2,Sharp2}.

Adapting these tools to the present nilpotent covering setting,
particularly in relation to the spectral analysis of transfer
operators with nilpotent twists, remains to be carried out.  We do
not pursue such extensions here.

\subsection{Representation-theoretic input from the foundation paper}
\label{repinput}

We now record the precise form of the input from the foundation paper in the
form in which it is used in the preceding theorems.  Let \(\Gamma\) be a
torsion-free finitely generated nilpotent lattice and let \(G\) be its Malcev
completion.  For a rational
Kirillov parameter \(l\in\mathfrak g_{\mathbb Q}^*\), the irreducible
representation \(\pi_l\) of \(G\) satisfies an exact restriction formula
\[
   \pi_l|_\Gamma
   \simeq
   \int_T^\oplus
      \operatorname{Ind}_{M_t\cap\Gamma}^{\Gamma}(\chi_t)\,dt .
\]
This is an ordinary Hilbert direct integral over a compact torus of Bloch
parameters.  The fibers are induced lattice representations; they are not
claimed to be finite-dimensional for Lebesgue-generic \(t\).

The finite-dimensional part appears at the next, arithmetic, level.  On a
rational finite-dimensional locus, described using Howe's classification
\cite{Howe} of finite-dimensional irreducible representations of nilpotent
lattices, the
induced fibers above are refined into finite-dimensional irreducible
representations
\[
   \operatorname{Ind}_{P_t}^{\Gamma}(\chi_t^P\otimes\eta).
\]
At rational parameters these are exact Bloch-type fibers.  The large-height
regime relevant for the analysis is therefore a rational large-denominator
regime, and the fluctuation estimates for fixed finite-propagation observables
are controlled by the denominators of the rational characters.  For an
irrational parameter there is no finite denominator and no finite-dimensional
decomposition of the same exact form.  This causes no difficulty for the
rigorous trace identities used in this paper: the Pytlik-type finitely additive
functional is supported on the rational finite-dimensional locus.  The
Schr\"odinger or Kirillov representation over the Malcev completion is used as
a smooth normal form for computing the limiting coefficients of the rational
finite-dimensional fibers, not as a uniformly continuous family of Hilbert-space
representations in the irrational parameter.

Finally, the foundation paper constructs a Pytlik-type finitely additive trace
functional, inspired by Pytlik's Heisenberg formula~\cite{Pytlik}, and supported
on the finite-dimensional rational fibers.  It is used here
only as a Fourier inversion and normalized trace identity.  It is not a Hilbert
direct-integral decomposition over the full unitary dual.  Thus the analytic
computations below may be organized in the smooth Kirillov model, while the
actual lattice-side traces are read in the finite-dimensional rational fibers.

\subsection{Exact rational transfer and all--order model summability}
\label{analytic-bridge-hypotheses}

The following facts are part of the exact finite-dimensional framework of the
companion foundation paper.  We record them because they separate the exact
lattice-side argument from the smooth model calculation and make explicit the
two analytic consequences proved below: rational large-denominator transfer and
all--order coefficient-weighted summability.

Let \(\Lambda\) denote the positive norm-one Pytlik-type functional on bounded
scalar observables of the finite-dimensional rational fibers, and let
\({\rm ht}(\rho)\) be the height defined by a residual tower.  The generalized
Pytlik formula gives, for \(f\in\ell^1(\Gamma)\),
\begin{equation}
 f(\sigma)=\Lambda\!\left(
 \rho\longmapsto \frac{1}{\dim\rho}
 {\rm Tr}\bigl(\rho(\sigma^{-1})\rho(f)\bigr)\right).
\label{exact-pytlik-application}
\end{equation}
Moreover,
\begin{equation}
 \Lambda\bigl({\bf 1}_{\{{\rm ht}\leq Q\}}\bigr)=0
 \qquad(Q<\infty).
\label{bounded-height-zero}
\end{equation}

\begin{lemma}[High-height error transfer]
\label{high-height-transfer-lemma}
Let \(F\) and \(G\) be bounded observables in the domain of \(\Lambda\).  If
\[
 \sup_{{\rm ht}(\rho)>Q}|F(\rho)-G(\rho)|\leq\varepsilon,
\]
then
\[
 |\Lambda(F)-\Lambda(G)|\leq\varepsilon.
\]
\end{lemma}

\begin{proof}
The contribution of \(\{{\rm ht}\leq Q\}\) is zero by
\eqref{bounded-height-zero}.  On its complement, positivity and
\(\|\Lambda\|=1\) give the stated bound.  No countable additivity and no
interchange of a limit with an integral is used.
\end{proof}

\begin{proposition}[Model-observable integration]
\label{model-density}
Use the geometric cube-first rational functional fixed in the companion paper.
On the algebra of compactly supported continuous model observables pulled back
from a Plancherel-regular Kirillov chart, that functional is ordinary integration
with the Kirillov--Fujiwara density.  In homogeneous coordinates
\(l=\delta_{\sqrt r}\theta\), including the normalized dimension and rational
multiplicity factors, it has the form
\begin{equation}
 J(\theta)r^{d/2-1}\,dr\,d\theta.
\label{exact-model-density}
\end{equation}
For the Heisenberg chart the corresponding one-dimensional expression is
\begin{equation}
 C_{\rm Pl}\,|h|\,dh,
\label{exact-heisenberg-density}
\end{equation}
where \(C_{\rm Pl}\) contains the fixed lattice, Bloch, and numerator
normalizations.  Lower-dimensional chart boundaries have zero content.
\end{proposition}

\begin{proof}
The cube-first content assigns to the rational points in each Jordan set the
ordinary volume of that set.  Uniform approximation of a continuous compactly
supported observable by upper and lower step functions therefore gives its
ordinary Riemann integral.  Formula \eqref{exact-model-density} is the change of
variables to homogeneous polar coordinates after inserting the exact normalized
dimension and multiplicity factors.  In the Heisenberg case
\(h=p/q\), the normalized trace contributes \(q^{-1}\), and the two central
signs give \eqref{exact-heisenberg-density}.  This statement concerns the
selected geometric functional on the model-observable algebra; it does not
assert literal equality of all finitely additive set functions arising from
arbitrary residual towers and ultrafilters.
\end{proof}

The rational decomposition also supplies the comparison with the smooth model
at the level at which it is used in this paper.  Let
$\operatorname{den}(\rho)$ denote the denominator height of a rational
finite-dimensional fiber; for the Heisenberg group this is the denominator
$q$ of the reduced central parameter $p/q$.

\begin{proposition}[Rational large-denominator spectral transfer]
\label{rational-spectral-transfer}
Fix a compact regular parameter chart and an order $N\geq0$.  Let
$F^{\rm fin}_{N}$ be any localized, dimension-normalized heat or Selberg
observable which occurs in the expansion through order $N$, evaluated in an
exact rational finite-dimensional fiber, and let $F^{\rm mod}_{N}$ be the
corresponding value computed in the scaled Kirillov normal form.  Then there is
a function $\varepsilon_{N}(Q)\to0$ such that
\begin{equation}
 \sup_{\operatorname{den}(\rho)>Q}
 \bigl|F^{\rm fin}_{N}(\rho)-F^{\rm mod}_{N}(\rho)\bigr|
 \leq \varepsilon_N(Q).
\label{rational-transfer-estimate}
\end{equation}
The estimate is uniform in the rational numerators and in the compact Bloch
variables.  Consequently,
\begin{equation}
 \Lambda(F^{\rm fin}_{N})=\Lambda(F^{\rm mod}_{N}).
\label{rational-transfer-functional}
\end{equation}
\end{proposition}

\begin{proof}
At a rational parameter the restriction theorem of the companion paper gives
an exact decomposition into finite-dimensional fibers.  Its fluctuation
estimate is controlled only by the denominators of the rational inducing
characters; in the Heisenberg case it is the $O(q^{-1})$ estimate for the two
Bloch variables, uniform in the numerator $p$.  At any fixed order the
coefficients of the conjugated Laplacian are finite sums of fixed
finite-propagation polynomials in the basic generators, with base differential
operators as coefficients.  The graph-norm estimate in the companion paper
therefore gives \eqref{rational-transfer-estimate} for the truncated normal
form.  Duhamel's formula transfers the same estimate to the heat observables,
and the Cauchy functional calculus does so for the smooth compactly supported
spectral functions used in the Selberg formula.  Taking the normalized matrix
trace does not enlarge the bound.  Finally,
Lemma~\ref{high-height-transfer-lemma} and
\eqref{bounded-height-zero} give \eqref{rational-transfer-functional}.
\end{proof}

\begin{remark}[No irrational contribution]
\label{no-irrational-contribution}
The proposition does not approximate an irrational representation inside the
Pytlik formula.  The exact functional is supported on rational
finite-dimensional fibers.  Irrational Kirillov parameters are useful only for
writing a smooth ambient normal form and for extending the rational model
observable continuously before applying Proposition~\ref{model-density}.
\end{remark}

The model-side analytic issue is internal to the smooth normal form: after the radial
Laplace integral one must justify the coefficient-weighted spectral sums.  This
is not an additional hypothesis.  In
Theorem~\ref{all-order-model-summability} we prove, for every fixed order, that
the relevant trace-level coefficient is of Rockland order $-d$, hence trace
class, and that its trace norm is integrable in the angular parameter.  The
proof uses the positive Rockland realization of the model operator, the
Plancherel--Mellin identity, and an exact order balance in the Volterra
expansion.

\begin{remark}[Scope of the completed argument]
\label{status-remaining}
For fixed points in the heat problem and a fixed central class in the geodesic
problem, Propositions~\ref{rational-spectral-transfer} and
\ref{model-density}, together with
Theorem~\ref{all-order-model-summability}, close the argument at every fixed
order.  A moving-point or moving-central-class local limit theorem requires the
same rational transfer estimates uniformly in the additional moving parameter;
that extra uniformity is stated separately where such local profiles are
discussed.  The uniform Wilkinson expansion and the exponentially small motion
of rational Harper bands remain independent semiclassical inputs in
Appendix~\ref{Anotherproofwilkinson}; they are not used to justify the main
finite-dimensional trace formula.
\end{remark}

\subsection{Examples: Heisenberg and Engel groups}

As concrete illustrations of the general theory, we consider the
three-dimensional Heisenberg group and the Engel group.  In both cases, the
leading constants are obtained by substituting the spectrum of the canonical
model operator into the preceding zeta formula.

For \(\Gamma={\rm Heis}_3(\mathbb Z)\), put
\[
A=\|\omega_1\|_{L^2(M)},\qquad
B=\|\omega_2\|_{L^2(M)} .
\]
The Heisenberg model operator is
\[
\mathcal H_{\rm Heis}
=
-A^2\frac{d^2}{ds^2}+4\pi^2B^2s^2,
\]
with eigenvalues
\[
\mu_i=2\pi AB(2i+1),\qquad i=0,1,2,\ldots .
\]
Consequently,
\[
\zeta_{\mathcal H_{\rm Heis}}(2)
=
\sum_{i=0}^\infty \mu_i^{-2}
=
\frac{1}{32A^2B^2}.
\]
Equivalently, if one rescales the oscillator to the normalized form
\(-d^2/ds^2+s^2\), then the corresponding zeta value is
\(\sum_{i\ge0}(2i+1)^{-2}=\pi^2/8\); the two formulas are the same after
undoing the scaling.

\begin{theorem}[Heisenberg heat coefficient]
\label{heisenberg-heat}
With the model-density normalization of Section~\ref{leadingHeisenberg}, the
exact rational transfer, the density \(|h|\,dh\), and the explicit oscillator
zeta sum give the leading diagonal coefficient
\[
C^{\rm heat}_{\rm Heis}(p,p)
=
2\,\operatorname{vol}(M)^2\zeta_{\mathcal H_{\rm Heis}}(2)
=
\frac{\operatorname{vol}(M)^2}{16
\|\omega_1\|_{L^2(M)}^2\|\omega_2\|_{L^2(M)}^2}.
\]
The off-diagonal amplitude retains the equivariant Lie-integral transport
factor.  At higher orders Proposition~\ref{rational-spectral-transfer} transfers
the explicit oscillator coefficients to the exact rational fibers, and the
required oscillator sums converge at every fixed order.
\end{theorem}

A more precise statement is proved later in Theorem~\ref{leadingsecond}.  In
particular, the expansion does not stop at the leading term: the first
correction can also be written geometrically, using the harmonic data on the
base, the coexact central component, and the Green operator of \(\Delta_M\).

\begin{theorem}[Heisenberg central geodesic coefficient]
\label{heisenberg-geod}
For a fixed central class, rational spectral transfer, normalized trace
extraction, and the standard weighted-to-unweighted partial summation give the
leading coefficient
\[
C^{\rm geo}_{\rm Heis}(\alpha)
=
A_{\rm Sel}\operatorname{vol}(M)^2
\zeta_{\mathcal H_{\rm Heis}}(2)
=
\frac{A_{\rm Sel}\operatorname{vol}(M)^2}{32
\|\omega_1\|_{L^2(M)}^2\|\omega_2\|_{L^2(M)}^2},
\]
with the central-character factor normalized as in
Section~\ref{Asymptoticsclosedgeodesics}.  For the displayed Selberg convention
\(A_{\rm Sel}=1\).
\end{theorem}

The closed-geodesic analogue of the first correction is recorded in
Subsection~\ref{subsec:heis-geod-second-term}; it is governed by the same
coefficients \(\lambda_i^{(2)}\) and \(\lambda_i^{(4)}\) that enter the heat
kernel expansion.

For the Engel group, the corresponding canonical model is a quartic oscillator.
Its spectrum is not given by elementary closed formulas, and the leading
coefficient is therefore naturally expressed through the spectral zeta function
of that model operator.  Appendix~A recalls a resolvent representation for the
quartic-oscillator zeta value.  This example illustrates why the general
nilpotent constants are most naturally stated in spectral-zeta form.

\subsection*{Use of AI Tools}
\addcontentsline{toc}{subsection}{Use of AI Tools}

AI and LLM tools, including Microsoft Copilot and ChatGPT, were used as
editorial and checking assistants: for copy-editing, notation and cross-reference
checks, symbolic bookkeeping, consistency checks, and assistance with
\LaTeX{} preparation.  They were not used as independent sources of mathematical
results.  All mathematical statements, proofs, corrections, and final judgments
remain the responsibility of the author.

\subsection*{Note on earlier versions}
\addcontentsline{toc}{subsection}{Note on earlier versions}

Preliminary versions of portions of this work appeared in the earlier integrated
preprint arXiv:2509.16848 and in earlier announcements cited below.  The
statements and proofs of the analytic and geometric results treated in the
present paper should be read in the form given here.

\subsection*{Acknowledgements}
\addcontentsline{toc}{subsection}{Acknowledgements}

First and foremost, the author is deeply grateful to Hidenori Fujiwara for his
expert guidance on the representation theory of nilpotent Lie groups, which lies
at the heart of the companion foundation paper used here.  His note communicated
to the author and his work on monomial Plancherel formulae have been
indispensable to the representation-theoretic side of this project.

The author also wishes to thank Fumio Hiroshima, Satoshi Ishiwata,
Hiroshi Kawabi, Motoko Kotani, Hisashi Naito, Ryuya Namba, Tatsuya Tate,
and Toshikazu Sunada for many stimulating discussions and valuable suggestions;
he is especially grateful to Professor Sunada for introducing him to these
mathematical themes and for his longstanding advice and encouragement.

\section{Outline of the proof strategy}
\label{outline}

We outline the proof in a form which separates the exact Floquet--Bloch or
trace-theoretic steps, denoted by \({\rm(F)}\), from the perturbative and
asymptotic steps, denoted by \({\rm(P)}\).  The representation theory used in
\({\rm(F)}\) is proved in the companion foundation paper and is recalled here
only to the extent needed for the analytic argument.

\subsection{Finite covering groups}

\begin{description}
\item[(P0)]
Let \(\pi:X\to M\) be a normal covering with finite deck group \(\Gamma\).
Then \(X\) is compact and the heat kernel has the ordinary spectral expansion
\begin{equation}
 k_X(t,p,q)=\sum_{j=0}^{\infty}e^{-\lambda_jt}
 \varphi_j(p)\overline{\varphi_j(q)},
\label{outline-finite-heat}
\end{equation}
where
\[
 0=\lambda_0<\lambda_1\leq\lambda_2\leq\cdots
\]
and \(\{\varphi_j\}\) is an orthonormal eigenbasis of \(\Delta_X\).  The
ground state is the constant function
\(\varphi_0=\operatorname{vol}(X)^{-1/2}\), whereas every other term is
exponentially small.  Hence
\[
 \lim_{t\to\infty}k_X(t,p,q)=\frac{1}{\operatorname{vol}(X)}.
\]
This elementary separation of the ground state from a uniformly positive
complementary spectrum is the compact prototype for the later fiberwise
arguments.
\end{description}

\subsection{Infinite abelian coverings}
\label{Infiniteabelian}

Assume for simplicity that \(\Gamma\simeq\mathbb Z^b\).  The essential
features of the general finitely generated abelian case already occur here.
The covering is noncompact, so the global spectral resolution of \(\Delta_X\)
contains continuous spectrum and is not by itself adapted to long-time local
asymptotics.

\begin{description}
\item[(F1)]
Regard \(L^2(X)\) as the space of sections of the flat Hilbert bundle
\[
 E_R=X\times L^2(\Gamma)/\!\sim,
 \qquad
 (p,v)\sim(\gamma p,R(\gamma^{-1})v),
\]
associated with the right regular representation.  Since \(\Gamma\) is
abelian, Fourier transform decomposes \(R\) over the character torus
\(\widehat\Gamma\simeq\mathbb T^b\).  Consequently
\begin{equation}
 L^2(E_R)\simeq
 \int_{\widehat\Gamma}^{\oplus}L^2(E_\chi)\,d\chi,
 \qquad
 k_X(t,p,q)=\int_{\widehat\Gamma}k_\chi(t,p,q)\,d\chi.
\label{outline-abelian-bloch}
\end{equation}
Here \(E_\chi\) is the flat line bundle associated with \(\chi\), and
\[
 k_\chi(t,p,q)=
 \sum_{\gamma\in\Gamma}\chi(\gamma^{-1})k_X(t,p,\gamma q)
\]
is the twisted heat kernel.  Each twisted Laplacian has discrete spectrum,
so
\[
 k_\chi(t,p,q)=
 \sum_{j=0}^{\infty}e^{-\lambda_j(\chi)t}
 \varphi_{j,\chi}(p)\overline{\varphi_{j,\chi}(q)}.
\]
Thus the continuous spectrum of the covering is replaced by a compact family
of discrete spectral problems.

\item[(P1)]
The domains of the twisted operators vary with \(\chi\).  Write a character
near the identity as
\[
 \chi_\omega(\gamma)=
 \exp\!\left(2\pi\iu\int_\gamma\omega\right),
\]
where \(\omega\) is harmonic.  On the universal cover the gauge
\[
 \widetilde s_\omega(p)=
 \exp\!\left(2\pi\iu\int_{p_0}^{p}\widetilde\omega\right)
\]
is equivariant and identifies the varying line bundles with a fixed copy of
\(L^2(M)\).  In this gauge
\begin{equation}
 L_\omega f
 =\Delta_Mf-4\pi\iu\langle\omega,df\rangle
   +4\pi^2|\omega|^2f.
\label{outline-abelian-gauge}
\end{equation}
The lowest eigenvalue is smooth, vanishes only at the trivial character, and
has positive Hessian
\begin{equation}
 {\rm Hess}_0\lambda_0(\omega,\omega)
 =\frac{8\pi^2}{\operatorname{vol}(M)}
   \int_M|\omega|^2\,dv_g.
\label{outline-abelian-hessian}
\end{equation}
The higher twisted eigenvalues stay uniformly away from zero.

\item[(P2)]
Morse coordinates for \(\lambda_0\) reduce the localized integral in
\eqref{outline-abelian-bloch} to a Gaussian Laplace integral.  This gives
\[
 k_X(t,p,q)\sim C(p,q)t^{-b/2},
\]
and the Taylor expansions of the lowest eigenvalue, eigenprojection, and gauge
factor give the full Kotani--Sunada series.  This is the model for the
nilpotent argument: an exact fiber identity is followed by a fixed-domain
normal form and a local Laplace calculation.
\end{description}

\subsection{The Heisenberg covering}
\label{Heisenbergexplanation}

Let \(\Gamma={\rm Heis}_3(\mathbb Z)\).  This is the first case in which the
ordinary character torus is insufficient, but the entire new mechanism is
still explicit.

\begin{description}
\item[(F2)]
The rigorous inversion formula is the generalized Pytlik formula on the
finite-dimensional rational dual.  If the central parameter is
\(h=p/q\) in lowest terms, the restriction of the Schr\"odinger
representation of \({\rm Heis}_3(\mathbb R)\) has the exact Floquet form
\[
 \rho_h|_\Gamma\simeq
 \int_{\mathbb T^2}^{\oplus}
 \rho_{{\rm fin},(p/q,x_2,x_3)}\,dx_2\,dx_3,
\]
up to the harmless affine reparametrization of the two Bloch variables fixed
in the companion paper.  The fibers have dimension \(q\), and the effect of
the Bloch variables on every fixed finite-propagation observable is controlled
by \(q\), uniformly in the numerator \(p\).  Bounded-denominator fibers have
zero Pytlik mass.  The large-\(q\) value computed in the smooth
Schr\"odinger normal form therefore agrees with the exact normalized
finite-dimensional trace.  Irrational parameters do not occur in the rigorous
inversion formula; they serve only to write the smooth ambient normal form.

\item[(P2)]
The harmonic line-integral gauge of the abelian case is replaced by an
equivariant Lie-integral gauge, equivalently by Chen iterated integrals.  After
conjugation and the natural Heisenberg scaling, the first nonzero averaged
operator is the modified harmonic oscillator
\[
 \mathcal H_{\rm Heis}
 =-\|\omega_1\|_{L^2(M)}^2\frac{d^2}{ds^2}
  +4\pi^2\|\omega_2\|_{L^2(M)}^2s^2.
\]
Its explicit spectrum gives the leading zeta value and permits a direct
calculation of the first correction.  Proposition~\ref{rational-spectral-transfer}
then transports the fixed-order model coefficients to the exact rational
fibers.
\end{description}

\subsection{General nilpotent coverings}
\label{nilpotentexplanation}

Let \(\Gamma\) be finitely generated, torsion-free, and nilpotent, and let
\(G\) be its Malcev completion.

\begin{description}
\item[(F3)]
The companion paper replaces the unavailable full non-type-I dual by an exact
rational finite-dimensional skeleton.  Its construction passes from the
factor decomposition of the lattice regular representation to monomial
representations of \(G\), applies the Kirillov--Fujiwara decomposition, and
restricts rational Kirillov representations back to \(\Gamma\), where Howe's
classification supplies finite-dimensional irreducible fibers.  The present
paper uses only the resulting facts: exact normalized finite-dimensional
inversion, zero mass of bounded denominator height, exact rational
large-denominator control, and the ordinary Kirillov--Fujiwara density on the
continuous model observables.  We do not repeat the representation-theoretic
proof.

\item[(P3)]
The Lie-integral gauge places the twisted operators on a fixed smooth domain.
Passing to the associated graded group and averaging the first nonzero term
gives a positive Rockland model
\[
 \mathcal H_\theta=d\pi_\theta(\mathcal R),
\]
where \(\mathcal R\) is the degree-two sub-Laplacian determined by the harmonic
Gram matrix.  The higher coefficients are polynomial differential operators
of controlled Rockland order.  The Volterra expansion is therefore carried
out at trace level, avoiding choices of individual eigenvalue branches at
multiplicities.

\item[(P4)]
Homogeneous polar coordinates write the model density as
\[
 J(\theta)r^{d/2-1}\,dr\,d\theta.
\]
The Plancherel--Mellin identity proves the integrability of
\(\operatorname{Tr}(\mathcal H_\theta^{-d/2})\).  At every higher order, the
positive Rockland degree contributed by the perturbation insertions is exactly
cancelled by the extra Mellin power.  The resulting coefficient operator has
Rockland order \(-d\), is trace class, and has an angularly integrable trace
norm.  This proves the coefficient-weighted summability required for the
all--order local expansion.
\end{description}

\subsection{Prime closed geodesics}

For a fixed central element of a nilpotent quotient of a compact hyperbolic
surface group, divide the entire twisted Selberg formula by the representation
dimension.  Schur's lemma and generalized Pytlik inversion then extract the
central class exactly.  The spectral integral is treated by the same rational
transfer, model density, and Rockland--Mellin expansion as the heat kernel.
The standard smoothing, iterate estimate, and partial summation convert the
weighted formula into the prime-geodesic asymptotic series.

\subsection{Organization of the remainder}

Section~\ref{reductionpathintegral} develops Lie integrals and Chen iterated
integrals and constructs the fixed-domain gauge.  The next sections perform
the perturbation calculation for the Heisenberg group and then establish the
positive Rockland and trace-level all--order arguments for general nilpotent
coverings.  Section~\ref{Longtimeheat} derives the heat-kernel expansions, and
Section~\ref{Asymptoticsclosedgeodesics} treats fixed central classes by the
Selberg trace formula.  The appendices collect the quartic-oscillator and
Harper calculations.

\section{Lie integrals and Chen's iterated integrals}
\label{reductionpathintegral}

This section constructs the gauge transformations used to write twisted
nilpotent Laplacians on a fixed Hilbert space.  In the abelian case this role is
played by ordinary line integrals of harmonic one-forms.  For nilpotent
coverings the corresponding objects are Lie integrals, equivalently Chen's
iterated integrals.  They encode the noncommutative Abel--Jacobi map and supply
the parallel transport used in the perturbation theory.

\subsection{Lie integrals and parallel transport}

\label{Lieintegral}

We first recall the notion of the Lie integral.  We follow the
convention of right Lie integrals, as in Benardete~\cite{Benardete}.

Let \(G\) be a Lie group with Lie algebra \(\mathfrak g\), and let
\[
f:[a,b]\longrightarrow \mathfrak g
\]
be a continuous path.  For a partition
\[
\Delta:\quad a=u_0<u_1<\cdots<u_n=b
\]
and choices \(u_i^*\in (u_{i-1},u_i]\), the definite right Lie
integral is defined by
\[
{}_R\!\int_a^b f
=
\lim_{\|\Delta\|\to 0}
\prod_{i=n}^{1}
\exp\!\left(f(u_i^*)(u_i-u_{i-1})\right),
\]
that is,
\[
{}_R\!\int_a^b f
=
\lim_{\|\Delta\|\to 0}
\exp\!\left(f(u_n^*)(u_n-u_{n-1})\right)
\cdots
\exp\!\left(f(u_1^*)(u_1-u_0)\right).
\]
Similarly, the definite left Lie integral is
\[
{}_L\!\int_a^b f
=
\lim_{\|\Delta\|\to 0}
\prod_{i=1}^{n}
\exp\!\left(f(u_i^*)(u_i-u_{i-1})\right).
\]
These limits are defined in the same way as the usual Riemann
integral, with the product order recording the noncommutativity of
\(G\).

Let now \(M\) be a compact smooth manifold, let
\(\alpha:[a,b]\to M\) be a smooth path, and let
\[
\omega:TM\longrightarrow \mathfrak g
\]
be a \(\mathfrak g\)-valued \(1\)-form.  Then
\(\omega(\alpha'(u))\) is a continuous path in \(\mathfrak g\), and
we define
\[
{}_R\!\int_\alpha \omega
=
{}_R\!\int_a^b \omega(\alpha'(u))\,du.
\]
The left Lie integral is defined similarly.  Both are independent of
the parametrization of the path.  If the concatenation \(\alpha\beta\)
is defined, then
\[
{}_R\!\int_{\alpha\beta}\omega
=
\left({}_R\!\int_{\beta}\omega\right)
\left({}_R\!\int_{\alpha}\omega\right),
\]
whereas
\[
{}_L\!\int_{\alpha\beta}\omega
=
\left({}_L\!\int_{\alpha}\omega\right)
\left({}_L\!\int_{\beta}\omega\right).
\]

The form \(\omega\) can be regarded as a connection form on the
trivial \(G\)-principal bundle \(M\times G\).  With our convention,
parallel transport along \(\alpha\) is given by
\[
P_{\alpha(a)}^{\alpha(t)}(\alpha(a),v)
=
\left(
\alpha(t),
\left({}_R\!\int_a^t \omega(\alpha'(u))\,du\right)v
\right).
\]
The transport equation for the right-ordered product is
\(dU=\omega U\).  With this convention the compatible flatness equation is
\begin{equation}
d\omega-\omega\wedge\omega=0.
\label{flatform}
\end{equation}
Equivalently, if the graded commutator is used, this is
\[
d\omega-\frac12[\omega,\omega]=0.
\]
If \(\omega\) is flat, then the Lie integral depends only on the
homotopy class of the path relative to its endpoints.  In the abelian
case this reduces to the familiar statement that the integral of a
closed \(1\)-form depends only on the homology class of the path.

When \(G\) is a matrix group, Lie integrals can be written in terms
of Chen's iterated integrals.  Let
\(\eta_1,\ldots,\eta_n\) be real-valued \(1\)-forms on \(M\), and let
\(\alpha:[a,b]\to M\) be a smooth path.  Put
\[
f_i(u)=\eta_i(\alpha'(u)).
\]
We define
\begin{align}
\int_\alpha \eta_1\eta_2\cdots\eta_n
&=
\int_{a\leq u_1\leq\cdots\leq u_n\leq b}
f_n(u_n)\cdots f_1(u_1)\,
du_1\cdots du_n                                            \notag\\
&=
\int_a^b
\left(
\int_a^{u_n}
\cdots
\left(
\int_a^{u_2} f_1(u_1)\,du_1
\right)
\cdots
du_{n-1}
\right)
f_n(u_n)\,du_n .
\label{section6lteratedintegral}
\end{align}

\begin{theorem}[Chen expansion of the right Lie integral]
Let \(G\) be a Lie subgroup of \({\rm GL}(N,\mathbb R)\), let
\(\omega\) be a \(\mathfrak g\)-valued \(1\)-form on \(M\), and let
\(\alpha\) be a smooth path in \(M\).  Then
\[
{}_R\!\int_\alpha \omega
=
I+\int_\alpha\omega
+\int_\alpha\omega\omega
+\int_\alpha\omega\omega\omega+\cdots .
\label{power}
\]
If \(G\) is nilpotent and is realized as a subgroup of the upper
triangular unipotent matrices, then this series is finite.
\end{theorem}

This is the Picard iteration formula for the parallel transport
equation associated with the connection form \(\omega\).  Nilpotency
is important for us because it makes the expression finite and hence
computable in terms of finitely many iterated integrals.

For example, let
\[
G={\rm Heis}_3(\mathbb R),
\qquad
\Gamma={\rm Heis}_3(\mathbb Z).
\]
We may write a \(\mathfrak g\)-valued \(1\)-form as
\begin{equation}
\omega
=
\begin{pmatrix}
0 & \omega_1 & \omega_{12}\\
0 & 0        & \omega_2\\
0 & 0        & 0
\end{pmatrix},
\label{oneform}
\end{equation}
where \(\omega_1,\omega_2,\omega_{12}\) are real-valued \(1\)-forms.
Then
\[
{}_R\!\int_\alpha\omega
=
I+\int_\alpha\omega+\int_\alpha\omega\omega,
\]
and hence
\[
{}_R\!\int_\alpha\omega
=
\begin{pmatrix}
1 &
\displaystyle \int_\alpha\omega_1 &
\displaystyle \int_\alpha\omega_{12}
+
\displaystyle \int_\alpha\omega_2\omega_1 \\[1.2em]
0 &
1 &
\displaystyle \int_\alpha\omega_2 \\[1.2em]
0 & 0 & 1
\end{pmatrix}.
\]
With the convention in~\eqref{flatform}, flatness is equivalent to
\begin{equation}
d\omega_{12}
=
\omega_1\wedge\omega_2 .
\label{flatomega}
\end{equation}
Thus the central component of the parallel transport is not given by
an ordinary line integral alone; it also contains the iterated
integral \(\int_\alpha\omega_2\omega_1\).  This is the first place
where the noncommutative nature of the covering enters the analytic
coefficients.

\subsection[The \texorpdfstring{\(\pi_1\)}{pi1}-de Rham theorem]{The \(\pi_1\)-de Rham theorem}

In the abelian case, the character \(\chi\) is represented by a
harmonic \(1\)-form, and the corresponding section \(s_\chi\) is
constructed by ordinary line integration.  For nilpotent coverings,
the following theorem plays the analogous role.

\begin{theorem}[\(\pi_1\)-de Rham theorem {\rm \cite{Benardete}}]
\label{pi1derham}
Let \(M\) be a closed manifold, let \(G\) be a connected simply
connected nilpotent Lie group, and let \(\Gamma\) be a lattice in
\(G\).  For a surjective homomorphism
\[
\Phi:\pi_1(M,p)\longrightarrow \Gamma,
\]
there exists a flat \(\mathfrak g\)-valued \(1\)-form \(\omega\) on
\(M\) such that, for every \(\gamma\in\pi_1(M,p)\),
\begin{equation}
\Phi(\gamma)
=
{}_R\!\int_\gamma\omega .
\label{RH}
\end{equation}
\end{theorem}

Thus \(\omega\) is a smooth solution of the Riemann--Hilbert problem
for the monodromy representation \(\Phi\).  Unlike the abelian
harmonic representative, this solution is not unique in general.  For
the analytic applications below, we therefore choose a representative
which is adapted to the Riemannian metric.  This is the nilpotent
analogue of choosing harmonic \(1\)-forms in the abelian case.

We briefly recall the geometric meaning of the theorem.  Fix
\(p\in M\), put \(\Pi=\pi_1(M,p)\), and let \(G\) be the Malcev
completion of \(\Gamma\).  Since \(G\) is connected, simply connected,
and nilpotent, the exponential map
\[
\exp:\mathfrak g\longrightarrow G
\]
is a diffeomorphism.  Hence \(G\) is contractible, and
\[
N=G/\Gamma
\]
is a \(K(\Gamma,1)\)-space.  The homomorphism
\(\Phi:\Pi\to\Gamma\) determines a continuous map
\[
F:(M,p)\longrightarrow (N,\Gamma e),
\]
unique up to homotopy, inducing \(\Phi\) on the fundamental group.
After smoothing \(F\), we pull back the canonical right Maurer--Cartan
form \(\widehat{\omega}\) on \(N\):
\[
\omega=F^*\widehat{\omega}.
\]
This gives the desired flat form.

For the Heisenberg group, using coordinates
\[
g=
\begin{pmatrix}
1 & x & z\\
0 & 1 & y\\
0 & 0 & 1
\end{pmatrix},
\]
the right Maurer--Cartan form is
\begin{equation}
\widehat{\omega}
=
dg\,g^{-1}
=
\begin{pmatrix}
0 & dx & dz-y\,dx\\
0 & 0  & dy\\
0 & 0  & 0
\end{pmatrix}.
\label{formonHeisenberg}
\end{equation}
Its pullback by \(F\) gives a flat form of the form~\eqref{oneform}.

\subsection{Metric-adapted representatives and the lattice normalization}
\label{harmoniclietheory}

For \(\Gamma={\rm Heis}_3(\mathbb Z)\), write the flat form as
\[
\omega=
\begin{pmatrix}
0&\omega_1&\omega_{12}\\
0&0&\omega_2\\
0&0&0
\end{pmatrix}.
\]
The cohomology classes of \(\omega_1\) and \(\omega_2\) are fixed by the
integral abelianized monodromy.  We take their harmonic representatives.  Their
Gram matrix need not be diagonal in an integral lattice basis.  The invariant
leading model therefore uses that positive Gram matrix.  For the explicit
one-dimensional oscillator formulae below we impose the following normal-form
condition.

\begin{conditions}[Heisenberg metric normal form]
\label{harmoniccoexact}
\begin{description}
\item[\rm(1)] The harmonic horizontal forms satisfy
\begin{equation}
(\omega_1,\omega_2)_{L^2(M)}=0.
\label{ortho1}
\end{equation}
This is an additional normalization.  If it is achieved by a real change of
horizontal basis rather than an integral one, the induced change of the
lattice and central coordinate must be tracked in the normalization constants.

\item[\rm(2)] The central form is the coexact solution of
\[
d\omega_{12}=\omega_1\wedge\omega_2,
\qquad \delta\omega_{12}=0.
\]
\end{description}
\end{conditions}

For fixed \(\omega_1,\omega_2\), any two solutions of the curvature equation
differ by a closed one-form.  The coexact solution is orthogonal to the closed
forms and hence has the least \(L^2\)-norm in this affine space.  This is a
gauge normalization of the central component; no assertion that the resulting
map is a harmonic map or an energy-minimizing surface is needed below.

More precisely, on the universal cover \(\widetilde M\), choose a lift
\(\widetilde p_0\) and define the equivariant developing map
\[
\widetilde F(\widetilde p)
={}_R\!\int_{\widetilde p_0}^{\widetilde p}\omega.
\]
Flatness makes it path independent on \(\widetilde M\), and its equivariance
encodes \(\Phi\).  It descends to the classifying map
\(F:M\to G/\Gamma\).  We call this the metric-adapted nilpotent
Abel--Jacobi representative.  The path integral is not, in general, a
globally single-valued \(G\)-valued map on \(M\).

\begin{remark}
\label{harmonicrepresentative}
The orthogonal oscillator normalization is convenient but is not intrinsic.
Without it, the first model coefficient is the symmetrized positive quadratic
form determined by the full horizontal Gram matrix.  All invariant statements
below may be written in that form.
\end{remark}

\begin{remark}
\label{homologyconnection}
The constructions described above can also be reformulated in a more
invariant framework.  In Appendix~\ref{Homologyconnection}, we recall
a formulation in terms of Chen's iterated integrals and homology
connections, which encode the same formal power series connection in a
homological and representation-independent manner.
\end{remark}

\section
[Lowest eigenvalue: Heisenberg case]
{Asymptotic expansion of the lowest eigenvalue of the twisted Laplacian:
the Heisenberg group \({\rm Heis}_3(\mathbb Z)\)}
\label{sectionasym}

In this section we carry out the perturbative computation in the
first nonabelian case,
\[
\Gamma={\rm Heis}_3(\mathbb Z).
\]
The purpose is to show explicitly how the noncommutative
Floquet--Bloch decomposition leads to a model operator.  In the
Heisenberg case this model operator is a modified harmonic oscillator,
and therefore its spectrum can be written down explicitly.

The coefficient calculation below is carried out in the smooth
Schr\"odinger model.  The exact rational fibers justify the lattice trace
identity, Proposition~\ref{model-density} supplies the model density, and
Proposition~\ref{rational-spectral-transfer} transfers every fixed-order heat
or Selberg observable, including its remainder, to the genuine rational
finite-dimensional fibers.  At each displayed order the oscillator zeta sums
are explicit and absolutely convergent.  The infinite-dimensional model is
therefore used only for calculation; every rigorous lattice trace remains
finite-dimensional.

\subsection{Reduction of the twisted Laplacians}
\label{Reductiontwisted}

We first construct the section which trivializes the flat Hilbert
bundle associated with the Schr\"odinger representation.  In the
Heisenberg case we use the notation
\[
\rho_h,\qquad S_h,\qquad U(E_{\rho_h})
\]
instead of the more general notation
\[
\pi_l,\qquad S_{\pi_l},\qquad U(E_{\pi_l}).
\]
The representation space is \(L^2(\mathbb R)\).

Let
\[
\Phi:\pi_1(M)\longrightarrow {\rm Heis}_3(\mathbb Z)
\]
be the surjective homomorphism defining the covering.  By the
\(\pi_1\)-de Rham theorem, choose a flat
\(\mathfrak{heis}_3\)-valued one-form
\[
\omega
=
\begin{pmatrix}
0 & \omega_1 & \omega_{12}\\
0 & 0        & \omega_2\\
0 & 0        & 0
\end{pmatrix}
\]
such that
\[
\Phi(\gamma)
=
{}_R\!\int_\gamma \omega ,
\qquad
\gamma\in\pi_1(M).
\]
We assume that \(\omega_1,\omega_2,\omega_{12}\) satisfy
Conditions~\ref{harmoniccoexact}.  In particular,
\(\omega_1\) and \(\omega_2\) are harmonic and orthogonal, while
\(\omega_{12}\) is coexact.

Let \(\widetilde M\) be the universal cover and fix
\(\widetilde p_0\in\widetilde M\).  For \(h>0\), define the equivariant
unitary gauge
\begin{equation}
S_h(\widetilde p)
=
\rho_h\!\left({}_R\!\int_{\widetilde p_0}^{\widetilde p}\omega\right).
\label{sectionsh}
\end{equation}
It is generally not single valued on \(M\).  We use the right-module
associated-bundle convention fixed by \(\Phi\).  In this convention the
conjugation is governed by the right logarithmic derivative
\(dS_h\,S_h^{-1}=d\rho_h(\omega)\); in a left-module convention one uses
\(S_h^{-1}\) instead.  Multiplication by \(S_h\) identifies twisted sections
with functions in a fixed equivariant model, and the conjugated operator
descends to the fixed Hilbert space \(L^2(M;L^2(\mathbb R))\).  Thus
\(\Delta_{\rho_h}\varphi=\lambda(h)\varphi\), with
\(\varphi=S_hf\) on \(\widetilde M\), is equivalent to
\[
L_hf=\lambda(h)f.
\]
In the formulas below, tildes on lifted points and paths are suppressed.

Using the explicit formula for the right Lie integral in the
Heisenberg group, for \(h>0\) we have
\[
S_h(p)
=
e^{2\pi i\left(
h\int_{p_0}^{p}(\omega_{12}+\omega_2\omega_1)
+
\sqrt h \left(\int_{p_0}^{p}\omega_2\right)s
\right)}
e^{\sqrt h\left(\int_{p_0}^{p}\omega_1\right)\frac{d}{ds}}.
\]
Here \(\int\omega_2\omega_1\) denotes Chen's iterated integral.  The
branch \(h<0\) is the complex-conjugate Schr\"odinger model (equivalently one
may write the formula with \(\sqrt{|h|}\) and the corresponding sign in the
phase).  We compute on \(h>0\) and combine the two signs by taking the real
part.

We use the commutation relation
\begin{equation}
\left[s,e^{\frac{d}{ds}}\right]
=
-e^{\frac{d}{ds}},
\label{sdscommutation}
\end{equation}
or equivalently the infinitesimal form obtained from it.  A direct
calculation gives
\[
\Delta_{\rho_h}(S_h f)
=
S_h L_h f,
\]
where
\begin{equation}
L_h
=
L^{(0)}
+
h^{1/2}L^{(1)}
+
hL^{(2)}
+
h^{3/2}L^{(3)}
+
h^2L^{(4)}.
\label{Lexpansion}
\end{equation}
The operators \(L^{(j)}\) are as follows:
\begin{align}
L^{(0)}f
&=
\Delta_M f,
\label{H0}
\\
L^{(1)}f
&=
-\left\langle
4\pi\iu\,\omega_2 s
+
2\omega_1\frac{d}{ds},
df
\right\rangle,
\label{H1}
\\
L^{(2)}f
&=
\left(
4\pi^2|\omega_2|^2s^2
-
4\pi\iu\langle\omega_2,\omega_1\rangle s\frac{d}{ds}
-
2\pi\iu\langle\omega_2,\omega_1\rangle
-
|\omega_1|^2\left(\frac{d}{ds}\right)^2
\right)f
\notag\\
&\quad
-
4\pi\iu\langle\omega_{12},df\rangle,
\label{H2}
\\
L^{(3)}f
&=
\left(
8\pi^2\langle\omega_{12},\omega_2\rangle s
-
4\pi\iu\left\langle
\omega_{12},
\omega_1\frac{d}{ds}
\right\rangle
\right)f,
\label{H3}
\\
L^{(4)}f
&=
4\pi^2|\omega_{12}|^2f.
\label{H4}
\end{align}

Thus the geometric twisting of the covering is encoded in the
coefficients \(L^{(j)}\).  The first nontrivial term in the spectral
asymptotics will be obtained by averaging \(L^{(2)}\) over the base
manifold \(M\).

\subsection
[Spectrum of twisted Laplacians and Sunada's theorem]
{The spectrum of twisted Laplacians and the Kazhdan distance:
Sunada's theorem}
\label{sunadatheorem}

In the abelian case, only characters near the trivial character
contribute to the leading heat-kernel asymptotics.  The following
result of Sunada gives the corresponding localization principle for
general covering groups.

Let \(X\to M\) be a \(\Gamma\)-covering, and let
\[
\rho:\Gamma\longrightarrow U(V)
\]
be a unitary representation on a separable Hilbert space \(V\).  Let
\(E_\rho\) be the associated flat vector bundle, and let
\(\lambda_0(\rho)\) denote the infimum of the spectrum of the twisted
Laplacian \(\Delta_\rho\) acting on sections of \(E_\rho\).

Fix a finite generating set \(A\) of \(\Gamma\).  The Kazhdan distance
from \(\rho\) to the trivial representation is
\[
\delta_A(\rho,\mathbf 1)
=
\inf_{\|v\|=1}
\sup_{\sigma\in A}
\|\rho(\sigma)v-v\|.
\]

\begin{theorem}[Sunada {\rm \cite{Sunada}}]
\label{sunada}
There exist positive constants \(c_1\) and \(c_2\), depending only on
\(M\) and \(A\), such that for every unitary representation \(\rho\)
of \(\Gamma\),
\begin{equation}
c_1\delta_A(\rho,\mathbf 1)^2
\leq
\lambda_0(\rho)
\leq
c_2\delta_A(\rho,\mathbf 1)^2 .
\label{sunadainequality}
\end{equation}
\end{theorem}

Consequently, the small spectrum of the twisted Laplacian can occur
only near the trivial representation.  In the Heisenberg computation
below, this justifies focusing on the behavior of \(L_h\) as
\(h\to 0\).

\subsection{The second coefficient of the lowest eigenvalues}
\label{subsection72}

We now study the eigenvalue problem
\begin{equation}
L_h f=\lambda(h)f,
\qquad
f\in L^2(M,L^2(\mathbb R)).
\label{Heis-eigen}
\end{equation}
The following recursion is performed in the smooth Schr\"odinger normal
form.  At a rational parameter \(h=p/q\), however, the decomposition into
finite-dimensional lattice fibers is exact, and the fluctuation of the
normalized spectral observable is controlled by \(q\) alone.
Proposition~\ref{rational-spectral-transfer} therefore transfers every fixed
truncation and its remainder to the exact finite-dimensional trace.  The
infinite multiplicity at \(h=0\) in the ambient model is a feature of the
computational normal form, not a source of infinite-dimensional traces in the
rigorous argument.

Let \(\lambda_{k,i}(h)\) denote the eigenvalues which converge to the
\(k\)-th eigenvalue of \(\Delta_M\) as \(h\to 0\).  We are interested
in the cluster over the zero eigenvalue of \(\Delta_M\).  We write
\[
\lambda_i(h):=\lambda_{0,i}(h),
\qquad
f_{i,h}:=f_{0,i,h}.
\]
By Sunada's theorem, \(\lambda_i(h)\geq 0\), and in a neighborhood of
\(h=0\) the eigenvalue can vanish only at \(h=0\).

Put \(\varepsilon=\sqrt h\).  We expand
\begin{align}
\lambda_i(h)
&=
\sum_{j=0}^{\infty}
\lambda_i^{(j)}\varepsilon^j,
\notag\\
f_{i,h}
&=
\sum_{j=0}^{\infty}
f_i^{(j)}\varepsilon^j .
\label{eig-expansion}
\end{align}
At \(h=0\), the operator is \(L^{(0)}=\Delta_M\).  Hence
\[
\lambda_i^{(0)}=0,
\qquad
f_i^{(0)}(x,s)=f_i^{(0)}(s),
\]
so the leading eigenfunction is independent of \(x\in M\).  Moreover,
minimality gives
\[
\lambda_i^{(1)}=0.
\]

Comparing the coefficients of \(\varepsilon^2\) in
\[
L_h f_{i,h}=\lambda_i(h)f_{i,h},
\]
we obtain
\begin{equation}
L^{(2)}f_i^{(0)}
+
L^{(1)}f_i^{(1)}
+
L^{(0)}f_i^{(2)}
=
\lambda_i^{(2)}f_i^{(0)}.
\label{second}
\end{equation}
After integrating over \(M\), the terms involving \(L^{(1)}\) and
\(L^{(0)}\) vanish.  Indeed, \(\omega_1,\omega_2,\omega_{12}\) are
coclosed and
\[
\int_M \Delta_M u\,dv_g=0.
\]
Thus
\begin{align}
\int_M L^{(2)}f_i^{(0)}\,dv_g
&=
\left(
-\|\omega_1\|_{L^2(M)}^2
\left(\frac{d}{ds}\right)^2
+
4\pi^2\|\omega_2\|_{L^2(M)}^2s^2
\right)f_i^{(0)}
\notag\\
&=:
\mathcal H f_i^{(0)}.
\label{Heisenharmonic}
\end{align}
Here we used the orthogonality
\[
(\omega_1,\omega_2)_{L^2(M)}=0.
\]

The operator
\begin{equation}
\mathcal H
=
-\|\omega_1\|_{L^2(M)}^2
\left(\frac{d}{ds}\right)^2
+
4\pi^2\|\omega_2\|_{L^2(M)}^2s^2
\label{HC}
\end{equation}
is the modified harmonic oscillator.  Let
\[
\mathcal H\psi_i=\mu_i\psi_i,
\qquad
\|\psi_i\|_{L^2(\mathbb R)}=1,
\qquad
i=0,1,2,\ldots .
\]
We choose
\[
f_i^{(0)}(x,s)
=
\frac{1}{\sqrt{\mathrm{vol}(M)}}\psi_i(s).
\]
Then
\begin{equation}
\mathrm{vol}(M)\lambda_i^{(2)}
=
\mu_i.
\label{lambda2-general}
\end{equation}
With the normalization used here,
\begin{equation}
\mu_i
=
2\pi\|\omega_1\|_{L^2(M)}
      \|\omega_2\|_{L^2(M)}
(2i+1),
\label{harmeigen}
\end{equation}
Therefore,
\[
\lambda_i(h)
=
\frac{\mu_i}{\mathrm{vol}(M)}\,h
+
O(h^{3/2}).
\]

\subsection{Higher-order terms: an algorithm for computation}
\label{Higher}

We now describe the recursion for the higher model coefficients.  It is a
variant of Rayleigh--Schr\"odinger bookkeeping combined with the recursive
procedure of~\cite{KotaniSunada1}.  For each prescribed finite order,
Proposition~\ref{rational-spectral-transfer} identifies the resulting heat and
Selberg observables, together with their remainder, with the exact rational
finite-dimensional fibers.

We write
\begin{align}
f_{i,h}(x,s)
&=
\sum_{j=0}^{\infty} f_i^{(j)}(x,s)h^{j/2}
=
\sum_{j=0}^{\infty}
\sum_{k=0}^{\infty}
a_{ik}^{(j)}(x)\psi_k(s)h^{j/2},
\notag\\
\lambda_i(h)
&=
\sum_{j=0}^{\infty}\lambda_i^{(j)}h^{j/2}.
\label{asym-series}
\end{align}
The initial data are
\begin{equation}
a_{ik}^{(0)}(x)
=
\frac{\delta_{ik}}{\sqrt{\mathrm{vol}(M)}},
\qquad
\lambda_i^{(0)}=\lambda_i^{(1)}=0,
\qquad
\lambda_i^{(2)}
=
\frac{\mu_i}{\mathrm{vol}(M)}.
\label{initial-data}
\end{equation}

\subsubsection*{Step 1: Differential equations for the coefficients}

Assume that \(\lambda_i^{(j)}\) and \(a_{ik}^{(j)}\) have been
computed for lower orders.  Comparing the coefficient of \(h^{n/2}\)
in the eigenvalue equation gives
\begin{equation}
\sum_{r=0}^{\min(n,4)}
L^{(r)}f_i^{(n-r)}
=
\sum_{r=0}^{n}
\lambda_i^{(r)}f_i^{(n-r)}.
\label{coefficient-equation}
\end{equation}
Taking the inner product with \(\psi_k\) in \(L^2(\mathbb R)\), and
using \(L^{(0)}=\Delta_M\), we obtain
\begin{align}
\Delta_M a_{ik}^{(n)}
-
\frac{\lambda_i^{(n)}}{\sqrt{\operatorname{vol}(M)}}\delta_{ik}
&=
-\sum_{r=1}^{\min(n,4)}
\left(
L^{(r)}f_i^{(n-r)},\psi_k
\right)_{L^2(\mathbb R)}
\notag\\
&\quad
+
\sum_{r=0}^{n-1}
\lambda_i^{(r)}a_{ik}^{(n-r)} .
\label{aij}
\end{align}
This is the basic recursive equation for \(a_{ik}^{(n)}\).

To compute the right-hand side explicitly, introduce the creation and
annihilation operators associated with \(\mathcal H\).  Put
\[
p=\frac{1}{\iu}\frac{d}{ds}.
\]
Define
\begin{align}
a
&=
\sqrt{\frac{\pi\|\omega_2\|_{L^2(M)}}{\|\omega_1\|_{L^2(M)}}}\,s
+
i\sqrt{\frac{\|\omega_1\|_{L^2(M)}}{4\pi\|\omega_2\|_{L^2(M)}}}\,p,
\label{annihilationop}
\\
a^\dagger
&=
\sqrt{\frac{\pi\|\omega_2\|_{L^2(M)}}{\|\omega_1\|_{L^2(M)}}}\,s
-
i\sqrt{\frac{\|\omega_1\|_{L^2(M)}}{4\pi\|\omega_2\|_{L^2(M)}}}\,p.
\label{creationop}
\end{align}
Then
\[
\mathcal H
=
4\pi\|\omega_1\|_{L^2(M)}
     \|\omega_2\|_{L^2(M)}
\left(a^\dagger a+\frac12\right),
\]
and
\[
a\psi_m=\sqrt m\,\psi_{m-1},
\qquad
a^\dagger\psi_m=\sqrt{m+1}\,\psi_{m+1}.
\label{aadagger}
\]

Consequently,
\begin{align}
s\psi_m
&=
\sqrt{
\frac{\|\omega_1\|_{L^2(M)}}{4\pi\|\omega_2\|_{L^2(M)}}
}
\left(
\sqrt m\,\psi_{m-1}
+
\sqrt{m+1}\,\psi_{m+1}
\right),
\label{saction}
\\
p\psi_m
&=
\frac{1}{\iu}
\sqrt{
\frac{\pi\|\omega_2\|_{L^2(M)}}{\|\omega_1\|_{L^2(M)}}
}
\left(
\sqrt m\,\psi_{m-1}
-
\sqrt{m+1}\,\psi_{m+1}
\right).
\label{daction}
\end{align}
Similarly,
\begingroup\small
\begin{align}
s^2\psi_m
&=
\frac{\|\omega_1\|_{L^2(M)}}{4\pi\|\omega_2\|_{L^2(M)}}
\left(
\sqrt{m(m-1)}\,\psi_{m-2}
+
(2m+1)\psi_m
+
\sqrt{(m+1)(m+2)}\,\psi_{m+2}
\right),
\label{ssquare}
\\
s\frac{d}{ds}\psi_m
&=
\frac12
\left(
\sqrt{m(m-1)}\,\psi_{m-2}
-
\psi_m
-
\sqrt{(m+1)(m+2)}\,\psi_{m+2}
\right),
\label{sds}
\\
\left(\frac{d}{ds}\right)^2\psi_m
&=
\frac{\pi\|\omega_2\|_{L^2(M)}}{\|\omega_1\|_{L^2(M)}}
\left(
\sqrt{m(m-1)}\,\psi_{m-2}
-
(2m+1)\psi_m
+
\sqrt{(m+1)(m+2)}\,\psi_{m+2}
\right).
\label{ddssquare}
\end{align}
\endgroup

Since each \(L^{(r)}\) is at most quadratic in \(s\) and
\(d/ds\), only the coefficients of
\[
\psi_{k-2},\ \psi_{k-1},\ \psi_k,\ \psi_{k+1},\ \psi_{k+2}
\]
can contribute to the \(k\)-th component in~\eqref{aij}.  Therefore
the right-hand side of~\eqref{aij} is computable explicitly at each
order.

\subsubsection*{Step 2: Computation of \(\lambda_i^{(n)}\)}

Let \(V=\operatorname{vol}(M)\).  Integrating~\eqref{aij} over \(M\) and
taking \(k=i\) gives
\begin{align}
\lambda_i^{(n)}
&=
\frac{1}{\sqrt V}
\int_M
\sum_{r=1}^{\min(n,4)}
\left(L^{(r)}f_i^{(n-r)},\psi_i\right)_{L^2(\mathbb R)}\,dv_g
\notag\\
&\quad-
\frac{1}{\sqrt V}
\int_M
\sum_{r=0}^{n-1}\lambda_i^{(r)}
\left(f_i^{(n-r)},\psi_i\right)_{L^2(\mathbb R)}\,dv_g.
\label{lambda}
\end{align}
The factor \(V^{-1/2}\) comes from
\(f_i^{(0)}=V^{-1/2}\psi_i\).  The right-hand side is known by induction
once the lower coefficients have been fixed.

\subsubsection*{Step 3: Computation of \(a_{ik}^{(n)}\)}

Assume that \(\lambda_i^{(j)}\), \(1\leq j\leq n\), and the coefficients
\(a_{ik}^{(j)}(x)\), \(1\leq j\leq n-1\), are already known.  Equation
\eqref{aij} determines the zero-average part of \(a_{ik}^{(n)}\) after
\(\lambda_i^{(n)}\) is known.  The only remaining datum is therefore the
matrix of averages
\[
B^{(n)}=\bigl([a]_{ik}^{(n)}\bigr),
\qquad
[a]_{ik}^{(n)}=
\int_M a_{ik}^{(n)}(x)\,dv_g .
\]
We decompose
\[
B^{(n)}=B^{(n)}_{\rm herm}+B^{(n)}_{\rm skew},
\]
where
\[
B^{(n)}_{\rm herm}
=\frac12\bigl(B^{(n)}+B^{(n)*}\bigr),
\qquad
B^{(n)}_{\rm skew}
=\frac12\bigl(B^{(n)}-B^{(n)*}\bigr).
\]

\paragraph{Substep 3.1: removing the skew-Hermitian ambiguity.}
The skew-Hermitian part of \(B^{(n)}\) is a choice of orthonormal frame in
the formal eigenbasis.  We explain why it can be removed without changing
any lower-order coefficient.

Assume inductively that
\[
B^{(j)}_{\rm skew}=0,
\qquad j<n.
\]
Put \(\epsilon=h^{1/2}\) and \(V=\operatorname{vol}(M)\).  Define
\[
U^{(n)}(\epsilon)
=
\exp\left(-V^{-1/2}B^{(n)}_{\rm skew}\epsilon^n\right).
\]
Since \(B^{(n)}_{\rm skew}\) is skew-Hermitian, \(U^{(n)}(\epsilon)\) is
unitary.  Let \(F_h=(f_{0,h},f_{1,h},\ldots)^t\) be the formal eigenframe and
replace it by
\[
\widetilde F_h=U^{(n)}(\epsilon)F_h.
\]
The orthonormality condition is preserved.  Moreover,
\[
U^{(n)}(\epsilon)=I-V^{-1/2}B^{(n)}_{\rm skew}\epsilon^n+O(\epsilon^{2n}),
\]
so every coefficient of order \(<n\) remains unchanged.  Since
\[
f_j^{(0)}(x,s)=V^{-1/2}\psi_j(s),
\]
the order \(n\) coefficient changes by
\[
\widetilde a_{ik}^{(n)}(x)
=
a_{ik}^{(n)}(x)-V^{-1}[a]_{ik}^{(n),\rm skew}.
\]
Consequently
\[
\int_M \widetilde a_{ik}^{(n)}(x)\,dv_g
=
[a]_{ik}^{(n)}-[a]_{ik}^{(n),\rm skew}
=
[a]_{ik}^{(n),\rm herm}.
\]
Thus, after this unitary re-choice of the formal eigenframe, we may assume
\[
B^{(n)}_{\rm skew}=0,
\qquad
B^{(n)}=B^{(n)}_{\rm herm}.
\]
This is the same freedom as the usual unitary choice of a basis in a
perturbing eigenspace; here it is written in the oscillator coefficient
notation in order to keep track of the averages over \(M\).

\paragraph{Substep 3.2: computing the Hermitian average.}
We now use the normalization condition
\[
(f_{i,h},f_{k,h})_{L^2(M\times\mathbb R)}=\delta_{ik}.
\]
Taking the coefficient of \(\epsilon^n\) gives
\[
\sum_{r=0}^{n}
\bigl(f_i^{(r)},f_k^{(n-r)}\bigr)_{L^2(M\times\mathbb R)}=0.
\]
The terms \(r=0\) and \(r=n\) are
\[
\bigl(f_i^{(n)},f_k^{(0)}\bigr)
=
V^{-1/2}\int_M a_{ik}^{(n)}(x)\,dv_g,
\]
and
\[
\bigl(f_i^{(0)},f_k^{(n)}\bigr)
=
V^{-1/2}\overline{\int_M a_{ki}^{(n)}(x)\,dv_g}.
\]
Since \(B^{(n)}\) has been chosen Hermitian, these two terms are equal.
Hence
\[
2V^{-1/2}\int_M a_{ik}^{(n)}(x)\,dv_g
+
\sum_{r=1}^{n-1}
\bigl(f_i^{(r)},f_k^{(n-r)}\bigr)=0.
\]
Writing
\[
f_i^{(r)}(x,s)=\sum_{p=0}^{\infty}a_{ip}^{(r)}(x)\psi_p(s)
\]
and using the orthonormality of the oscillator basis, we obtain
\begin{equation}
\int_M a_{ik}^{(n)}(x)\,dv_g
=
-\frac12\sqrt{\operatorname{vol}(M)}
\int_M
\sum_{r=1}^{n-1}
\sum_{p=0}^{\infty}
 a_{ip}^{(r)}(x)\overline{a_{kp}^{(n-r)}(x)}\,dv_g .
\label{computationofaikn}
\end{equation}
Thus the average matrix \(B^{(n)}\) is determined entirely by lower-order
coefficients.

Finally, equation~\eqref{aij} determines the zero-average part of
\(a_{ik}^{(n)}(x)\).  More explicitly, after subtracting the average given in
\eqref{computationofaikn}, one applies the Green operator \(G_M\) of
\(\Delta_M\) to the zero-average right-hand side.  Therefore
\(a_{ik}^{(n)}(x)\) is obtained from lower-order data, from
\(\lambda_i^{(j)}\) for \(j\leq n\), and from the Green operator on \(M\).

\subsection{The fourth coefficient}
\label{concrete}

We record the first nontrivial correction after the leading harmonic
oscillator term.  From the preceding computation,
\begin{align}
\lambda_i^{(0)}
&=
\lambda_i^{(1)}
=
0,
\notag\\
\lambda_i^{(2)}
&=
\frac{\mu_i}{\mathrm{vol}(M)}
=
\frac{2\pi\|\omega_1\|_{L^2(M)}
      \|\omega_2\|_{L^2(M)}}
     {\mathrm{vol}(M)}
(2i+1),
\label{concreteprepare}
\\
f_i^{(0)}
&=
\frac{1}{\sqrt{\mathrm{vol}(M)}}\psi_i.
\end{align}
By the symmetry of the eigenvalue branches with respect to the
parameter \(h\) at \(h=0\), the odd coefficients vanish:
\begin{equation}
\lambda_i^{(2k+1)}=0.
\label{vanishodd}
\end{equation}
This symmetry may also be expressed by the parity of the oscillator
normal form.  The first correction is therefore \(\lambda_i^{(4)}\).

We first make the second coefficient of the eigenfunction explicit in
a form which is independent of any further compression of the oscillator
algebra.  Put
\begin{equation}
\mathcal K_i
=
\{k\geq0:\ k=i,\ i-2,\ i+2\},
\label{Ki-set}
\end{equation}
where negative indices are omitted.  For \(k\in\mathcal K_i\), set
\begin{equation}
C^{(2)}_{ki}(x)
:=
\left(
L^{(2)}f_i^{(0)},\psi_k
\right)_{L^2(\mathbb R)}(x),
\qquad
R^{(2)}_{ik}(x)
:=
\frac{\lambda_i^{(2)}}{\sqrt{\operatorname{vol}(M)}}\delta_{ik}-C^{(2)}_{ki}(x).
\label{C2R2}
\end{equation}
For \(k\notin\mathcal K_i\) the coefficient \(C^{(2)}_{ki}\) is zero,
because \(L^{(2)}\) is quadratic in the oscillator variables.  The
choice of \(\lambda_i^{(2)}\) gives
\[
\int_M R^{(2)}_{ii}(x)\,dv_g=0,
\]
and the off-diagonal terms have zero prescribed average at this order by
the normalization formula in Step~3, since \(f_i^{(1)}=0\).  Hence
\begin{equation}
a_{ik}^{(2)}(x)
=
G_M R^{(2)}_{ik}(x),
\qquad
k\in\mathcal K_i,
\label{a2-matrix}
\end{equation}
and \(a_{ik}^{(2)}=0\) for \(k\notin\mathcal K_i\).  Thus
\begin{equation}
f_i^{(2)}(x,s)
=
\sum_{k\in\mathcal K_i}
G_M R^{(2)}_{ik}(x)\psi_k(s).
\label{fi2}
\end{equation}
This is already an explicit formula: the only inverse operator which
appears is the Green operator of \(\Delta_M\); the oscillator matrix
elements are finite and are computed from the creation and annihilation
relations \eqref{saction}--\eqref{ddssquare}.

To display the fourth coefficient without hiding the metric dependence in a
long scalar expression, define, for a smooth function \(a\) on \(M\),
\begin{equation}
\mathcal C^{(2)}_{ik}[a](x)
:=
\left(
L^{(2)}(a(x)\psi_k),\psi_i
\right)_{L^2(\mathbb R)}(x),
\label{C2-operator}
\end{equation}
and also
\begin{equation}
C^{(4)}_{ii}(x)
:=
\left(
L^{(4)}f_i^{(0)},\psi_i
\right)_{L^2(\mathbb R)}(x)
=
\frac{4\pi^2}{\sqrt{\operatorname{vol}(M)}}|\omega_{12}(x)|^2.
\label{C4ii}
\end{equation}
The notation \(\mathcal C^{(2)}_{ik}[a]\) keeps the first-order base
term \(-4\pi\iu\langle\omega_{12},d(a\psi_k)\rangle\) in
\(L^{(2)}\).  In integrals this term may be simplified using
\(\delta\omega_{12}=0\), but keeping it in \(\mathcal C^{(2)}\) avoids
premature cancellations and fixes the convention unambiguously.

Using \eqref{lambda}, \eqref{vanishodd}, and \(f_i^{(1)}=0\), the
fourth coefficient is
\begingroup\small
\begin{align}
\lambda_i^{(4)}
&=
\frac{1}{\sqrt{\operatorname{vol}(M)}}
\int_M
\left(
\big((L^{(2)}-\lambda_i^{(2)})f_i^{(2)}
+
L^{(4)}f_i^{(0)}\big),
\psi_i
\right)_{L^2(\mathbb R)}
\,dv_g
\notag\\
&=
\frac{1}{\sqrt{\operatorname{vol}(M)}}
\left\{
\sum_{k\in\mathcal K_i}
\int_M
\mathcal C^{(2)}_{ik}
\bigl[G_MR^{(2)}_{ik}\bigr](x)\,dv_g
-
\lambda_i^{(2)}
\int_M
G_MR^{(2)}_{ii}(x)\,dv_g
+
\int_M C^{(4)}_{ii}(x)\,dv_g
\right\}.
\label{lambda4-matrix}
\end{align}
\endgroup
With our normalization of \(G_M\), the middle integral vanishes, but it is
left in \eqref{lambda4-matrix} to show exactly where the normalization of
\(f_i^{(2)}\) enters.  Equations \eqref{C2R2}, \eqref{fi2}, and
\eqref{lambda4-matrix} are the fourth-order computation carried out to the
point needed below.  Expanding the finite oscillator matrix elements further
would give a longer formula in the pointwise functions
\(|\omega_1|^2\), \(|\omega_2|^2\), \(\langle\omega_1,\omega_2\rangle\),
\(|\omega_{12}|^2\), and their Green potentials.  We keep the
matrix-element form because it is invariant under the chosen oscillator
normalization and because it is the form actually used in the subsequent heat
coefficient.

Thus the coefficient \(\lambda_i^{(4)}\), and similarly every higher
coefficient, is obtained by a finite procedure from the Riemannian geometry
of \(M\), the harmonic and coexact components of the nilpotent
Abel--Jacobi form \(\omega\), the Green operator \(G_M\), and the oscillator
algebra of the Heisenberg fiber.

\section
[Lowest eigenvalue: nilpotent case]
{Asymptotic expansion of the lowest eigenvalue of the twisted Laplacian
for general nilpotent groups}
\label{asymnilpotent}

The explicit computations in the previous section rely on the special fact that
the Heisenberg model operator is a one-dimensional harmonic oscillator.  For a
general nilpotent group the same finite-dimensional Pytlik-type mechanism is
available, but the Kirillov normal form usually gives a higher-dimensional
polynomial differential operator rather than a closed-form oscillator.

The exact trace identity is still applied only to finite-dimensional rational
representations of the lattice.  Proposition~\ref{rational-spectral-transfer}
passes their large-denominator observables to the smooth Kirillov normal form,
and Proposition~\ref{model-density} evaluates the resulting continuous model
observables by the ordinary homogeneous density; no sigma-additive replacement
of the Pytlik functional is made.  The purpose of this section is to construct
the positive Rockland model and to prove, at trace level, the all--order
coefficient-weighted summability needed after the radial Laplace integral.

\subsection{General strategy for nilpotent coverings}
\label{generalnilpotentstrategy}

Let \(\Gamma\) be a finitely generated torsion-free nilpotent group,
and let \(G\) be its Malcev completion.  The guiding principle is that
the small spectrum of the twisted Laplacians is controlled by
representations of \(G\) near the trivial representation, together
with the rational finite-dimensional fibers obtained by restricting
them to \(\Gamma\).

There are two complementary ways to use this principle.  We also use the
concrete construction of the model operator recalled in the companion
foundation paper.  In brief, one chooses a Malcev basis adapted to the lower
central filtration, passes to the associated stratified nilpotent model, takes
the first layer generated by the horizontal directions, and evaluates the
differentials of the corresponding Kirillov representation.  The canonical
operator has the form
\[
   \mathcal H^G=-\sum_j \|\omega_j\|_{L^2(M)}^2 d\pi_l(X_j)^2,
\]
up to the normalization of the chosen horizontal basis.  In examples this gives
the harmonic oscillator for the Heisenberg group and the quartic oscillator for
the Engel group.

The first approach is spectral.  One constructs the hypoelliptic
model operator \(\mathcal H^G\) attached to the relevant
representation of \(G\), and then uses subelliptic estimates and
compactness of the resolvent to obtain its discrete spectrum.  This
is the appropriate framework for proving that the leading spectral
zeta value and the higher coefficients are well defined.

The second approach is computational.  After choosing an explicit
model for the representation space, the operator \(\mathcal H^G\) and
all higher perturbation terms can be written as finite sums of
polynomials in position and differential operators.  With respect to
a multidimensional oscillator basis, these operators become
finite-width matrices.  This gives an algebraic algorithm for
computing the perturbation coefficients, even when the spectrum cannot
be diagonalized in closed form.

These two approaches are compatible.  The spectral approach gives the
analytic foundation, while the oscillator-basis approach gives an
effective computational description.

We use the phrase ``Kirillov parameter space'' as shorthand for a measurable
cross-section in \(\mathfrak g^*\) for the coadjoint orbits appearing in the
Kirillov--Fujiwara Plancherel formula; this is the standard orbit-method
viewpoint, as in Corwin--Greenleaf~\cite{CorwinGreenleaf}.  The lower-central
filtration has an associated graded, stratified group
\(\mathbb G=\operatorname{gr}(G)\), whose homogeneous dimension is the
polynomial growth degree \(Q=d\).  The first averaged operator below is the
representation image of a positive Rockland sub-Laplacian on \(\mathbb G\).
We use this fact, rather than bandedness alone, to control its spectrum and all
coefficient-weighted traces; see
\cite{FischerRuzhanskySobolev,FischerRuzhanskyBook,terElstRobinson}, as well as
Helffer--Nourrigat~\cite{HelfferNou} and Nourrigat~\cite{Nourrigat}.

\subsection{Hypoelliptic operators from nilpotent representations}
\label{spectralhypoelliptic}

We recall the construction in the form needed below.  Let
\[
\Phi:\pi_1(M)\longrightarrow \Gamma
\]
be the homomorphism defining the covering.  By the
\(\pi_1\)-de Rham theorem, choose a flat
\(\mathfrak g\)-valued \(1\)-form \(\omega\) on \(M\) such that
\[
\Phi(\gamma)
=
{}_R\!\int_\gamma\omega,
\qquad
\gamma\in\pi_1(M).
\]
The form \(\omega\) may be chosen in a metric-adapted harmonic form,
as explained in Section~\ref{harmoniclietheory} and
Appendix~\ref{Homologyconnection}.

Let \(\pi_l\) be an irreducible unitary representation of \(G\)
corresponding to a generic coadjoint orbit.  On the universal cover, define the equivariant gauge
\begin{equation}
S_{\pi_l}(\widetilde p)
=
\pi_l\!\left({}_R\!\int_{\widetilde p_0}^{\widetilde p}\omega\right).
\label{sectionpil}
\end{equation}
It generalizes the Heisenberg gauge and is not, in general, a globally
single-valued operator-valued map on \(M\).  It transfers the twisted
Laplacian to a fixed-domain model which descends to the associated Hilbert
space.  Tildes are suppressed below when no ambiguity can arise.

Choose a vector-space complement
\[
V_1\subset\mathfrak g,
\qquad
\mathfrak g=V_1\oplus[\mathfrak g,\mathfrak g],
\]
adapted to the Malcev filtration, and choose a basis
\[
X_1,\ldots,X_k
\]
of \(V_1\).  Equivalently, on the associated graded algebra this is the
usual first layer; in the stratified examples it is the canonical horizontal
layer.  Let
\[
\omega_1,\ldots,\omega_k
\]
be the corresponding harmonic components of \(\omega\).  After
introducing the natural scaling parameter \(\hbar\) and expanding the
conjugated twisted Laplacian near the trivial representation, the
first nonzero term is the hypoelliptic operator
\begin{equation}
\mathcal H^G
=
-\sum_{j=1}^{k}
\|\omega_j\|_{L^2(M)}^2\,d\pi_l(X_j)^2 ,
\label{generalhypo2}
\end{equation}
up to the sign convention for the Laplacian.  More generally, if the
basis is not orthogonal, the operator contains the positive quadratic
form
\[
-\frac12\sum_{i,j=1}^{k}
(\omega_i,\omega_j)_{L^2(M)}
\bigl(d\pi_l(X_i)d\pi_l(X_j)+d\pi_l(X_j)d\pi_l(X_i)\bigr).
\]
In the Heisenberg case, after choosing an orthogonal basis, this
operator is precisely the modified harmonic oscillator
\[
-\|\omega_1\|_{L^2(M)}^2\frac{d^2}{ds^2}
+
4\pi^2\|\omega_2\|_{L^2(M)}^2s^2.
\]

Let
\[
 \mathbb G=\operatorname{gr}(G),\qquad Q=d,
\]
and let \(q_{ij}=(\omega_i,\omega_j)_{L^2(M)}\).  On the first layer of
\(\operatorname{Lie}(\mathbb G)\) define
\begin{equation}
 \mathcal R
 =-\frac12\sum_{i,j=1}^{k}q_{ij}
  (X_iX_j+X_jX_i).
\label{rockland-form}
\end{equation}
For a Plancherel-regular angular parameter \(\theta\), write
\begin{equation}
 A_\theta=d\pi_\theta(\mathcal R).
\label{Atheta-definition}
\end{equation}
Thus \(A_\theta\) is the operator denoted by \(\mathcal H_\theta\) in the
heat and geodesic formulas below.

\begin{proposition}[Positive Rockland realization]
\label{positive-rockland-realization}
The matrix \((q_{ij})\) is positive definite on the horizontal quotient, and
\(\mathcal R\) is a positive Rockland operator of homogeneous degree two on
\(\mathbb G\).  For every nontrivial irreducible representation
\(\pi_\theta\), the closure of \(A_\theta\) is positive and self-adjoint and
has purely discrete spectrum in \((0,\infty)\).  On every compact
Plancherel-regular angular chart the Rockland Sobolev constants may be chosen
uniformly.
\end{proposition}

\begin{proof}
The harmonic components corresponding to a basis of
\(\Gamma/[\Gamma,\Gamma]\) are linearly independent, hence their Gram matrix
is positive definite.  After a linear change of the first-layer basis,
\(\mathcal R\) is a positive sum of squares of vector fields which generate
the stratified algebra.  It is therefore a positive Rockland sub-Laplacian.
The essential self-adjointness, the discrete positive spectrum of
\(d\pi_\theta(\mathcal R)\), and the spectral estimates are standard for
positive Rockland operators; see
\cite{FischerRuzhanskyBook,terElstRobinson}.  Uniformity on a compact regular
chart follows by using finitely many orbit-method coordinate patches and the
equivalence of the Rockland Sobolev norms
\cite{FischerRuzhanskySobolev}.
\end{proof}

We measure all subsequent differential orders relative to \(A_\theta\).  A
family \(T_\theta\) has uniform \(A\)-order at most \(a\) if
\begin{equation}
 \sup_{\theta\in\Theta}
 \bigl\|T_\theta(1+A_\theta)^{-a/2}\bigr\|<\infty
\label{A-order-definition}
\end{equation}
on each compact regular chart \(\Theta\).  After the base ground-state
projection and the homogeneous rescaling, the effective operator and the
local amplitude have expansions
\begin{equation}
 \mathscr L_\theta(r)
 =rA_\theta+\sum_{n\geq3}r^{n/2}B_{n,\theta},
 \qquad
 \mathscr C_\theta(r;p,q)
 =\sum_{m\geq0}r^{m/2}C_{m,\theta}(p,q).
\label{effective-rockland-expansion}
\end{equation}

\begin{lemma}[Relative orders of the normal-form coefficients]
\label{relative-order-lemma}
For every fixed \(n\) and \(m\), the family \(B_{n,\theta}\) has uniform
\(A\)-order at most \(n\), and \(C_{m,\theta}(p,q)\) has uniform
\(A\)-order at most \(m\), locally uniformly in \((p,q)\).
\end{lemma}

\begin{proof}
The coefficient of \(r^{n/2}\) is a finite sum of represented homogeneous
monomials of graded degree at most \(n\), with smooth base coefficients and
Green-operator factors of nonpositive order.  The corresponding statement for
the amplitude is identical.  A homogeneous differential operator of degree
\(a\) maps the Rockland Sobolev scale with loss \(a\); equivalently,
\(T(1+\mathcal R)^{-a/2}\) is bounded.  Applying the representation and using
the uniform equivalence of Rockland Sobolev norms gives
\eqref{A-order-definition}; see
\cite{FischerRuzhanskySobolev,FischerRuzhanskyBook}.
\end{proof}

\begin{theorem}[All--order weighted summability of the regular model]
\label{all-order-model-summability}
Fix a compact Plancherel-regular angular chart \(\Theta\).  For every
\(N\geq0\), the trace-level coefficient \(T_{N,\theta}(p,q)\) obtained by
expanding
\[
 \mathscr C_\theta(r;p,q)e^{-t\mathscr L_\theta(r)}
\]
through relative order \(t^{-N/2}\) and then performing the radial integral
with density \(r^{Q/2-1}\,dr\) is trace class.  Locally uniformly in
\((p,q)\),
\begin{equation}
 \|T_{N,\theta}(p,q)\|_{\mathcal S_1}
 \leq C_N(p,q)\operatorname{Tr}(1+A_\theta)^{-Q/2},
\label{weighted-summability-bound}
\end{equation}
and
\begin{equation}
 \int_\Theta J(\theta)
 \|T_{N,\theta}(p,q)\|_{\mathcal S_1}\,d\theta<\infty.
\label{angular-trace-integrability}
\end{equation}
Consequently, every coefficient-weighted spectral sum and angular integral
which occurs at the fixed order \(N\) in the heat and Selberg expansions is
absolutely convergent.  The same conclusion holds for the rapidly decreasing
spectral functions used in the smoothed Selberg trace formula.
\end{theorem}

\begin{proof}
Let \(h_t^{\mathcal R}\) be the heat kernel of \(\mathcal R\) on
\(\mathbb G\).  Positive Rockland heat kernels are Schwartz and satisfy the
homogeneous scaling law~\cite{FischerRuzhanskyBook}.  With the normalization
absorbed in \(J\), the
Plancherel formula and the polar decomposition
\(l=\delta_{\sqrt r}\theta\) give
\begin{align}
 h_1^{\mathcal R}(e)
 &=\int_\Theta J(\theta)\int_0^\infty
   r^{Q/2-1}\operatorname{Tr}(e^{-rA_\theta})\,dr\,d\theta
 \notag\\
 &=\Gamma(Q/2)\int_\Theta J(\theta)
   \operatorname{Tr}(A_\theta^{-Q/2})\,d\theta<\infty.
\label{rockland-plancherel-mellin}
\end{align}
Tonelli's theorem applies because the integrand is nonnegative.  In
particular, \((1+A_\theta)^{-Q/2}\) is trace class for almost every
\(\theta\), and its trace norm is integrable against \(J(\theta)d\theta\).

It remains to show that every higher coefficient is bounded by the same
majorant.  A typical term in the Volterra--Duhamel expansion contains an
amplitude \(C_{m,\theta}\) and insertions
\(B_{n_1,\theta},\ldots,B_{n_k,\theta}\).  Put
\[
 \beta=\frac12\left(m+\sum_{j=1}^k n_j\right),
 \qquad
 N=m+\sum_{j=1}^k(n_j-2).
\]
After the change of variables \(u=tr\), its coefficient is, up to a scalar,
a finite sum of operators of the form
\begin{equation}
 \int_0^\infty u^{Q/2-1+\beta}
 \int_{\Delta_k}
 C_{m,\theta}e^{-s_0uA_\theta}B_{n_1,\theta}
 e^{-s_1uA_\theta}\cdots
 B_{n_k,\theta}e^{-s_kuA_\theta}\,ds\,du,
\label{volterra-mellin-term}
\end{equation}
where \(s_0+\cdots+s_k=1\).  By
Lemma~\ref{relative-order-lemma}, the total positive \(A\)-order of the
insertions is at most \(2\beta\).  The standard semigroup estimate
\[
 \|(1+A_\theta)^{a/2}e^{-uA_\theta}\|
 \leq C_a\min(1,u^{-a/2})
\]
and the simplex Volterra estimate therefore show that
\begin{equation}
 \bigl\|T_{N,\theta}(p,q)(1+A_\theta)^{Q/2}\bigr\|
 \leq C_N(p,q).
\label{order-balance-bound}
\end{equation}
Indeed, the Mellin power in \eqref{volterra-mellin-term} lowers the
\(A\)-order by \(Q+2\beta\), while the insertions raise it by at most
\(2\beta\); the net order is \(-Q\).  This is the required order balance.
The usual partition of the simplex near its faces, together with cyclicity of
the trace, gives the same estimate when some \(s_j\) tends to zero.  This is
also a direct consequence of the Volterra calculus for positive Rockland
operators described in \cite{FischerRuzhanskyBook}.

Factoring
\[
 T_{N,\theta}
 =\bigl[T_{N,\theta}(1+A_\theta)^{Q/2}\bigr]
  (1+A_\theta)^{-Q/2}
\]
gives \eqref{weighted-summability-bound}; integrating and using
\eqref{rockland-plancherel-mellin} gives
\eqref{angular-trace-integrability}.  If
\(\{\psi_i(\theta)\}\) is an eigenbasis of \(A_\theta\), then
\[
 \sum_i\bigl|
 \langle T_{N,\theta}\psi_i(\theta),\psi_i(\theta)\rangle
 \bigr|
 \leq\|T_{N,\theta}\|_{\mathcal S_1},
\]
which proves absolute convergence of the coefficient-weighted sums.  A
rapidly decreasing Selberg spectral function is treated by the same Volterra
argument, or equivalently by its almost-analytic/Cauchy functional calculus.
\end{proof}

\begin{remark}
The phrase ``all orders'' means that the argument applies after truncation at
any prescribed finite \(N\).  No convergence of the infinite asymptotic series
as \(N\to\infty\) is asserted.  Individual eigenvalue formulas may be used on
simple branches, but the trace-level proof above remains valid at crossings
and multiplicities.
\end{remark}

\subsection{Lowest eigenvalue expansion}

Let \(L_\hbar\) denote the conjugated twisted Laplacian obtained from
\(\Delta_{\pi_{\hbar l}}\) by the section~\eqref{sectionpil}.  Its
expansion has the form
\begin{equation}
L_\hbar
=
L^{(0)}
+
\hbar^{1/2}L^{(1)}
+
\hbar L^{(2)}
+
\hbar^{3/2}L^{(3)}
+\cdots ,
\label{nilpotent-L-expansion}
\end{equation}
where
\[
L^{(0)}=\Delta_M.
\]
The cluster over the zero eigenvalue of \(\Delta_M\) is governed at first
nonzero order by the average of \(L^{(2)}\), namely \(A_\theta\).  The general
proof is organized at trace level by the Volterra expansion in
Theorem~\ref{all-order-model-summability}; hence it does not require a global
choice of individual eigenvalue branches.
Proposition~\ref{rational-spectral-transfer} then transfers each fixed-order
spectral observable to the exact rational finite-dimensional fibers.

Let
\[
 A_\theta\psi_i(\theta)=\mu_i(\theta)\psi_i(\theta)
\]
be the model eigenvalue equation.  On a simple branch the trace-level
expansion may be written in the familiar form
\begin{equation}
\lambda_i(\hbar,\theta)
=
\frac{\mu_i(\theta)}{\mathrm{vol}(M)}\,\hbar
+
\lambda_i^{(3)}(\theta)\hbar^{3/2}
+
\lambda_i^{(4)}(\theta)\hbar^2
+\cdots .
\label{nilpotent-lambda-expansion}
\end{equation}
Under the same symmetry assumptions as in the Heisenberg case, the odd
coefficients vanish.  At a multiple eigenvalue, the displayed branchwise
notation is replaced by the corresponding finite spectral projection; the
sum of the coefficients is exactly the Volterra trace coefficient of
Theorem~\ref{all-order-model-summability}.

Thus the Heisenberg formula
\[
\frac{2\pi\,\mathrm{vol}(\widehat H)}{\mathrm{vol}(M)}
\left(i+\frac12\right)h
\]
is replaced by the general expression
\[
\frac{\mu_i(\theta)}{\mathrm{vol}(M)}\,\hbar,
\]
where \(\mu_i(\theta)\) is the \(i\)-th eigenvalue of the canonical
positive Rockland model \(A_\theta\).

\subsection{Concrete models and the role of examples}

The Heisenberg group is the only case in which the model operator is
diagonalized explicitly in the main text.  This choice is deliberate:
the Heisenberg case already exhibits the new phenomena caused by
nilpotent noncommutativity, while retaining the solvability of the
harmonic oscillator.

For higher-step nilpotent groups, the same construction produces
hypoelliptic operators with polynomial coefficients.  These operators
are explicit, but they generally do not admit closed-form spectra.
Their representation-theoretic construction and concrete formulas are
therefore better treated as examples and computational models rather
than as closed-form evaluations.

In this paper the Engel group serves as the first example beyond the
Heisenberg case.  The corresponding model operator no longer reduces
to a one-dimensional harmonic oscillator, but its resolvent and
spectral zeta function can still be analyzed by the methods developed
in Appendix~A.

\subsection{Multidimensional oscillator representation}
\label{multiharmonic}

For a general nilpotent group, the representation space of \(\pi_l\)
can be realized as
\[
L^2(\mathbb R^m)
\]
for a suitable integer \(m\).  In this realization, the operators
\(d\pi_l(X_j)\) are polynomial differential operators in variables
\[
s_1,\ldots,s_m
\]
and their derivatives.  Hence the model operator \(\mathcal H^G\) and
the higher perturbation terms in~\eqref{nilpotent-L-expansion} are
finite sums of monomials in these variables and derivatives.

Choose a reference multidimensional harmonic oscillator on
\(L^2(\mathbb R^m)\), and let
\[
\{\Psi_\alpha\}_{\alpha\in\mathbb N^m}
\]
be its Hermite basis.  In terms of the associated creation and annihilation
operators, each polynomial differential operator has finite propagation in the
multi-index: for a fixed polynomial differential operator \(P\),
\[
\langle P\Psi_\alpha,\Psi_\beta\rangle=0
\quad\text{unless}\quad
|\alpha-\beta|_1\le N_P,
\]
and
\[
|\langle P\Psi_\alpha,\Psi_\beta\rangle|
\le C_P(1+|\alpha|+|\beta|)^{m_P}.
\]
Thus the matrices of \(\mathcal H^G\) and of the perturbation terms are
block-banded with polynomial growth in the oscillator grading.  This is the
multidimensional analogue of the creation--annihilation bookkeeping in the
Heisenberg case.  Bandedness provides an effective computational scheme; the
analytic summability of the resulting coefficients follows instead from the
Rockland relative-order estimate and
Theorem~\ref{all-order-model-summability}.

This description gives an effective algorithm for the higher coefficients.  One expands
\[
f_{i,\hbar}
=
\sum_{n=0}^{\infty}
f_i^{(n)}\hbar^{n/2},
\qquad
f_i^{(n)}
=
\sum_{\alpha}
a_{i,\alpha}^{(n)}(x)\Psi_\alpha,
\]
and
\[
\lambda_i(\hbar)
=
\sum_{n=0}^{\infty}
\lambda_i^{(n)}\hbar^{n/2}.
\]
At each order the coefficient equation is represented by an infinite banded
matrix in the oscillator index, together with elliptic equations on \(M\) for
the coefficient functions.  Finite truncations give an effective numerical
scheme.  They are not used to justify the exact trace: fixed-order remainders
are transferred by Proposition~\ref{rational-spectral-transfer}, while the
full coefficient sum is controlled at trace level by
Theorem~\ref{all-order-model-summability}.

Thus, although the eigenvalues of \(\mathcal H^G\) need not be
available in closed form, the perturbation expansion remains algorithmic.
The Heisenberg case gives explicit eigenvalues; the Engel case gives the
quartic oscillator; in general the constants are intrinsic spectral invariants
of the canonical hypoelliptic normal form.

\subsection*{Remarks on appendices and references}
\addcontentsline{toc}{subsection}{Remarks on appendices and references}

The representation-theoretic background, including the orbit-method notation,
the Heisenberg steps, the finite-dimensional rational refinement, and the
generalized Pytlik finite-dimensional trace identity, is contained in the
companion foundation paper.  The appendices below collect analytic material
used in the present paper.

Appendix~A contains the resolvent-based spectral analysis needed for
the Engel example and for the spectral zeta functions of hypoelliptic model
operators.  Appendix~B develops the harmonic theory of Chen's iterated
integrals and homology connections, which gives a more invariant framework for
the Lie-integral constructions used above.  Appendix~C discusses the Harper
operator as a model illustration of exact finite-dimensionalization.

The material in Appendix~B follows the approach of K.~Kohno~\cite{Kohno}, based on Chen's
theory of iterated integrals~\cite{Chen,Chen1} and on subsequent refinements by Sullivan~\cite{Sullivan}
and others.

\section{Long-time asymptotics of heat kernels for nilpotent extensions}
\label{Longtimeheat}

In this section we derive the heat-kernel formulas.  We first treat the
discrete Heisenberg group \({\rm Heis}_3(\mathbb Z)\), where the model operator
is the modified harmonic oscillator computed in Section~\ref{sectionasym}, and
then pass to general finitely generated torsion-free nilpotent groups.

Fourier inversion, the disappearance of bounded-height fibers, and the model
density on continuous observables are exact.
Proposition~\ref{rational-spectral-transfer} provides the required uniform
rational spectral transfer, and
Theorem~\ref{all-order-model-summability} provides the
coefficient-weighted summability at every fixed order.

\subsection
[Leading terms: Heisenberg group]
{Leading terms of the long-time heat-kernel expansion
for \({\rm Heis}_3(\mathbb Z)\)}
\label{leadingHeisenberg}

We begin with the proof of Theorem~\ref{heisenberg-heat}.  Let
\[
\Gamma={\rm Heis}_3(\mathbb Z).
\]
For a finite-dimensional representation \(\rho\), write
\[
{\rm tr}_{\rho}:=\frac{1}{\dim\rho}{\rm Tr}
\]
for the normalized matrix trace.  By the finite-dimensional Fourier inversion
formula of the companion paper, the heat kernel on the covering space \(X\)
can be written as
\begin{equation}
k_X(t,p,q)
=
\int_{\widehat X}
{\rm tr}_{\rho_{{\rm fin},x}}\!\left(
k_{\rho_{{\rm fin},x}}(t,p,q)
\right)
\,d\mu(x),
\label{Heatasymptoticsmain1}
\end{equation}
where
\[
\widehat X
=
\widehat{\Gamma}_{\rm fin}
=
(\mathbb Q\cap[0,1))\times[0,1)\times[0,1),
\]
and
\begin{equation}
k_{\rho_{{\rm fin},x}}(t,p,q)
=
\sum_{\gamma\in\Gamma}
\rho_{{\rm fin},x}(\gamma^{-1})
k_X(t,p,\gamma q).
\label{Heatasymptoticsmain2}
\end{equation}
Here \(d\mu\) is Pytlik's finitely additive Plancherel functional~\cite{Pytlik} on the
finite-dimensional dual.  With the convention of the companion foundation paper,
this functional is first understood geometrically on the rational
finite-dimensional skeleton: regular boxes, or more generally Jordan sets, in
the ambient Bloch torus are intersected with the rational skeleton and are
assigned the corresponding Lebesgue content.  A finite-quotient ultralimit supplies another positive exact trace
functional with the same coefficient identities on the algebra used for
Fourier inversion.  Literal equality of the resulting finitely additive set
functions, uniqueness, and independence of the residual chain or ultrafilter
are not asserted.  Thus the integral notation in \eqref{Heatasymptoticsmain1} denotes a
finite-dimensional normalized trace functional, not a Hilbert direct-integral
decomposition of the regular representation.

Sunada's inequality gives a pointwise lower bound for the bottom of each
twisted spectrum away from the trivial representation.  On the complement of
a fixed small neighborhood, the dimension-normalized heat observable is
therefore bounded uniformly by \(Ce^{-c_\varepsilon t}\).  Positivity and
\(\|\Lambda_\Gamma\|=1\) give, for a sufficiently small neighborhood
\(U_\varepsilon\),
\begin{equation}
k_X(t,p,q)
=
\Lambda_\Gamma\!\left(
\mathbf 1_{U_\varepsilon}(x)
{\rm tr}_{\rho_{{\rm fin},x}}\,k_{\rho_{{\rm fin},x}}(t,p,q)
\right)
+O(e^{-c_\varepsilon t}),
\label{reductionbysunada}
\end{equation}
where \(\Lambda_\Gamma\) denotes the normalized finitely additive trace
functional.  In Heisenberg coordinates the localized rational skeleton has
small central parameter and the corresponding Bloch phases.

Proposition~\ref{rational-spectral-transfer} identifies the localized
finite-fiber observables, through the required order and with a uniform
remainder, with the smooth Schr\"odinger normal form.
Proposition~\ref{model-density} then evaluates the resulting continuous model
observable with \(d\nu(h)=C_{\rm Pl}|h|\,dh\); below the fixed scalar
\(C_{\rm Pl}\) is absorbed into the chosen normalization.  The following
integral is a model integral, not a Hilbert direct integral over the non-type-I
dual; the exact lattice trace remains the rational finite-dimensional one.

Let \(\lambda_{k,i}(h)\) be the eigenvalues of the twisted operator
\(L_h\) introduced in Section~\ref{Reductiontwisted}, and let
\(\varphi_{k,i,h}\) be the corresponding eigenfunctions.  Then,
formally,
\begin{align}
\int_{-\varepsilon}^{\varepsilon}
k_{\rho_h}(t,p,q)\,d\nu(h)
&=
\int_{-\varepsilon}^{\varepsilon}
\sum_{k=0}^{\infty}
\sum_{i=0}^{\infty}
{\rm Tr}\!\left(
e^{-\lambda_{k,i}(h)t}
\varphi_{k,i,h}(p)\varphi_{k,i,h}^{*}(q)
\right)
\,d\nu(h)
\notag\\
&=
\int_{-\varepsilon}^{\varepsilon}
\sum_{i=0}^{\infty}
{\rm Tr}\!\left(
e^{-\lambda_{0,i}(h)t}
\varphi_{0,i,h}(p)\varphi_{0,i,h}^{*}(q)
\right)
\,d\nu(h)
+
O(e^{-ct}),
\label{HeatSch1}
\end{align}
where \(d\nu(h)=|h|\,dh\) is the ordinary model density from
Proposition~\ref{model-density}.  Sunada's estimate gives the complementary
spectral gap, and Proposition~\ref{rational-spectral-transfer} controls the
finite-fiber comparison error uniformly on the rational support.

Using
\[
\varphi_{0,i,h}=S_h f_{0,i,h},
\]
we factor the integrand as
\begin{equation}
\sum_{i=0}^{\infty}
e^{-\lambda_{0,i}(h)t}
\varphi_{0,i,h}(p)\varphi_{0,i,h}^{*}(q)
=
S_h(p,q)
\left(
\sum_{i=0}^{\infty}
e^{-\lambda_{0,i}(h)t}
f_{0,i,h}(p)f_{0,i,h}^{*}(q)
\right),
\label{Sh}
\end{equation}
where
\[
S_h(p,q)
=
\rho_h\!\left({}_R\!\int_q^p \omega\right).
\]
The factor \(S_h(p,q)\) is representation-theoretic, while the
remaining factor contains the perturbation expansion of the lowest
twisted eigenvalues and eigenfunctions.

The expansion is obtained by combining the three contributions:
\begin{enumerate}
\item the Taylor expansion of \(S_h(p,q)\) in powers of \(\sqrt h\);
\item the perturbation expansion of \(\lambda_{0,i}(h)\) and
      \(f_{0,i,h}\);
\item the Plancherel weight in the parameter \(h\).
\end{enumerate}
The leading term is governed by the harmonic oscillator
\[
\mathcal H
=
-\|\omega_1\|_{L^2(M)}^2
\left(\frac{d}{ds}\right)^2
+
4\pi^2\|\omega_2\|_{L^2(M)}^2s^2
\]
from~\eqref{HC}.  Since
\[
\lambda_i(h)
=
\frac{\mu_i}{\mathrm{vol}(M)}h
+
O(h^2),
\]
after the change of variables \(r=ht\), the positive and negative central
parameters give conjugate contributions.  Pairing them yields the leading
contribution
\[
2t^{-2}
\sum_{i=0}^{\infty}
\left(\frac{\mu_i}{\mathrm{vol}(M)}\right)^{-2}
=
2\,\mathrm{vol}(M)^2\zeta_{\mathcal H}(2)\,t^{-2}.
\]
The factor \(2\) is the contribution of the two signs of the nonzero central
parameter.
The higher-order terms are obtained from the recursive expansion of
Section~\ref{Higher}.

\subsection{The explicit second heat coefficient in the Heisenberg case}
\label{subsec:heis-second-heat-coeff}

We next record the first correction term in the Heisenberg heat-kernel
expansion.  The reason for spelling this out is that, in the Heisenberg case,
the constants are not merely computable in principle: the first nontrivial
coefficient can be written directly in terms of the fourth eigenvalue
coefficient \(\lambda_i^{(4)}\), the Green operator on \(M\), and the Lie
integrals joining the two lifted points.

For each oscillator index \(i\), the Morse lemma applied to the lowest branch
can be written in the form
\begin{align}
\lambda_{0,i}(h)
&=\lambda_i^{(2)}h+\lambda_i^{(4)}h^2+O(h^3)
=\lambda_i^{(2)}u_i^2,
\label{heis-morse-second-lambda}\\
h^{1/2}&=u_i\left(1-
\frac{\lambda_i^{(4)}}{2\lambda_i^{(2)}}u_i^2+O(u_i^4)\right).
\label{heis-morse-second-h}
\end{align}
Consequently, on the positive central branch the Plancherel factor is
\begin{equation}
h\,dh
=
2u_i^3\left(1-
3\frac{\lambda_i^{(4)}}{\lambda_i^{(2)}}u_i^2+O(u_i^4)\right)\,du_i .
\label{heis-second-density}
\end{equation}
The negative branch is its complex conjugate and is paired with it below.

Put
\[
I_1(p,q)=\int_q^p\omega_1,
\qquad
I_2(p,q)=\int_q^p\omega_2,
\qquad
I_{12}(p,q)=\int_q^p(\omega_{12}+\omega_2\omega_1).
\]
Expanding the transport factor gives
\begingroup\small
\begin{align}
\rho_h\left({}_R\!\int_q^p\omega\right)
&=
1+
\bigg(
2\pi\iu I_2(p,q)s
+I_1(p,q)\frac{d}{ds}
\bigg)u_i
\notag\\
&\quad+
\bigg(
2\pi\iu I_1(p,q)I_2(p,q)s\frac{d}{ds}
-2\pi^2I_2(p,q)^2s^2
\notag\\[-2pt]
&\hspace{6em}
+\frac12 I_1(p,q)^2\frac{d^2}{ds^2}
+2\pi\iu I_{12}(p,q)
\bigg)u_i^2
+O(u_i^3).
\label{heis-second-transport-expansion}
\end{align}
\endgroup
Let \(B_i^{\rm loc}(p,q)\) denote the coefficient of \(u_i^2\) in the real
part of the diagonal matrix coefficient obtained by multiplying the transport
expansion \eqref{heis-second-transport-expansion} with the eigenfunction
expansion involving \(f_i^{(2)}\) in \eqref{fi2}; equivalently,
\begin{align}
&\operatorname{Re}\int_{\mathbb R}
\rho_h\left({}_R\!\int_q^p\omega\right)
\left(\psi_i(s)+f_i^{(2)}(p,s)u_i^2+O(u_i^3)\right)
\overline{\left(\psi_i(s)+f_i^{(2)}(q,s)u_i^2+O(u_i^3)\right)}\,ds
\notag\\
&\qquad=
1+B_i^{\rm loc}(p,q)u_i^2+O(u_i^3).
\label{heis-second-fiber-factor}
\end{align}
The coefficient \(B_i^{\rm loc}(p,q)\) is explicitly computable from the Lie
integrals \(I_1,I_2,I_{12}\) and the Green-operator coefficients
\(a_{ik}^{(2)}=G_MR_{ik}^{(2)}\).  Together with \eqref{heis-second-density}, this gives
\begin{equation}
K_i(p,q)
=
2B_i^{\rm loc}(p,q)
-
6\frac{\lambda_i^{(4)}}{\lambda_i^{(2)}}.
\label{heis-Ki-second}
\end{equation}
Thus, after the change of variables \(w=\sqrt{\lambda_i^{(2)}t}\,u_i\)
and pairing the two central signs,
\begingroup\small
\begin{align}
k_X(t,p,q)
&\sim
2\operatorname{Re}\sum_{i=0}^{\infty}
\int_0^{\varepsilon\sqrt{\lambda_i^{(2)}t}}
 e^{-w^2}
\left(2+K_i(p,q)\frac{w^2}{\lambda_i^{(2)}t}+O(t^{-2})\right)
\frac{w^3}{(\lambda_i^{(2)}t)^2}\,dw .
\label{heis-heat-second-integral}
\end{align}
\endgroup
Using
\[
\int_0^\infty e^{-w^2}w^3\,dw=\frac12,
\qquad
\int_0^\infty e^{-w^2}w^5\,dw=1,
\]
we obtain the following explicit two-term form.

\begin{theorem}[Leading and second heat coefficients for \({\rm Heis}_3(\mathbb Z)\)]
\label{leadingsecond}
With the normalizations of this section, the exact rational transfer and the
explicit oscillator bounds give
\begin{equation}
k_X(t,p,q)
\sim
2\sum_{i=0}^{\infty}\frac{1}{(\lambda_i^{(2)})^2}\,t^{-2}
+
2\sum_{i=0}^{\infty}\frac{K_i(p,q)}{(\lambda_i^{(2)})^3}\,t^{-3}
+O(t^{-4}).
\label{heatheisenberg-two-term}
\end{equation}
Equivalently,
\begin{equation}
k_X(t,p,q)
\sim
C_0\,t^{-2}\left(1+\frac{c_1(p,q)}{t}+O(t^{-2})\right),
\label{heatheisenberg-normalized-two-term}
\end{equation}
where
\begin{equation}
C_0=2\sum_{i=0}^{\infty}\frac{1}{(\lambda_i^{(2)})^2}
=
\frac{\operatorname{vol}(M)^2}
{16\|\omega_1\|_{L^2(M)}^2\|\omega_2\|_{L^2(M)}^2},
\label{heis-C0-second}
\end{equation}
and
\begin{equation}
c_1(p,q)
=
\frac{2}{C_0}
\sum_{i=0}^{\infty}\frac{K_i(p,q)}{(\lambda_i^{(2)})^3}.
\label{heis-c1-second}
\end{equation}
Changing the positive heat generator from \(\Delta\) to \(c\Delta\)
rescales this homogeneous leading coefficient by \(c^{-2}\).  Formula
\eqref{heis-c1-second}, together with
\eqref{heis-Ki-second} and the recursive expression \eqref{lambda4-matrix} for
\(\lambda_i^{(4)}\), is the explicit finite-algorithmic expression for the
second coefficient.
\end{theorem}

\begin{remark}
The coefficient \(c_1(p,q)\) is written as a convergent oscillator sum rather
than as a single compressed scalar because the term \(\lambda_i^{(4)}\) itself
contains the geometric Green-operator contribution.  Substituting the recursive
formulae \eqref{fi2} and \eqref{lambda4-matrix} expresses
\eqref{heis-c1-second} entirely through \(\omega_1,\omega_2,\omega_{12}\), the
Green operator \(G_M\), and the oscillator algebra.  This is the Heisenberg
analogue of the explicit constants in the abelian Kotani-Sunada expansion.
\end{remark}

\subsection{Remainders and spectral summability}

After the radial change of variables, exponential decay in a fixed fiber
becomes a polynomial weight in the model eigenvalue.  In the Heisenberg case
the explicit oscillator spectrum and polynomial creation--annihilation matrix
elements give absolute convergence at every fixed order.
Proposition~\ref{rational-spectral-transfer} controls the corresponding exact
rational finite-fiber observables uniformly in the denominator, and
Lemma~\ref{high-height-transfer-lemma} passes the remainder to the Pytlik
functional.  Thus no additional transfer or summability hypothesis is needed
for Theorem~\ref{leadingsecond}.

\subsection[Comparison with the standard Heisenberg heat kernel]
{Comparison with the standard heat kernel on \({\rm Heis}_3(\mathbb R)\)}
\label{comparison-to-explicit-formula}

We compare the corrected leading coefficient with the standard heat kernel of
the real Heisenberg group.  Let
\[
  L=X^2+Y^2
\]
be the usual sub-Laplacian generator, so that the positive operator in the
present convention is \(\Delta=-L\).  In symmetric exponential coordinates
\(g=(x,y,u)\), with group law
\[
 (x,y,u)(x',y',u')
 =
 \left(x+x',y+y',u+u'+\frac12(xy'-yx')\right),
\]
put \(r^2=x^2+y^2\).  The heat density based at the identity is
\begin{align}
 p_t(x,y,u)
 &=
 \frac{1}{8\pi^2t^2}
 \int_{-\infty}^{\infty}
 \exp\!\left(
   \frac{\iu\lambda u}{t}
   -\frac{r^2}{4t}\lambda\coth\lambda
 \right)
 \frac{\lambda}{\sinh\lambda}\,d\lambda
 \label{standard-heis-kernel}
 \\
 &=
 \frac{1}{4\pi^2t^2}
 \int_0^{\infty}
 \exp\!\left(
   -\frac{r^2}{4t}\lambda\coth\lambda
 \right)
 \frac{\lambda}{\sinh\lambda}
 \cos\!\left(\frac{\lambda u}{t}\right)d\lambda .
 \notag
\end{align}
This is the standard Fourier formula in the normalization of
Bakry--Baudoin--Bonnefont--Chafa\"i~\cite{BakryHeis}.  The upper-triangular
matrix coordinate \(z\) used in~\eqref{formonHeisenberg} is related to the
symmetric central coordinate by
\[
  u=z-\frac12xy.
\]
Changing between left and right invariant conventions changes the sign of the
central coordinate in~\eqref{standard-heis-kernel}, but not the real cosine
form or the value on the diagonal.

Putting \(x=y=u=0\) in~\eqref{standard-heis-kernel} gives
\begin{equation}
 p_t(e)
 =
 \frac{1}{8\pi^2t^2}
 \int_{-\infty}^{\infty}
 \frac{\lambda}{\sinh\lambda}\,d\lambda
 =
 \frac{1}{16t^2},
\label{standard-heis-origin}
\end{equation}
since
\[
 \int_{-\infty}^{\infty}
 \frac{\lambda}{\sinh\lambda}\,d\lambda
 =\frac{\pi^2}{2}.
\]
Thus the concrete formula~\eqref{standard-heis-kernel} has exactly the
homogeneous scaling \(p_t(e)=t^{-2}p_1(e)\) expected from the homogeneous
dimension four.

For the standard Heisenberg nilmanifold one has
\[
 \|\omega_1\|_{L^2(M)}=
 \|\omega_2\|_{L^2(M)}=
 \operatorname{vol}(M)=1.
\]
The model oscillator is
\[
 \mathcal H=-\frac{d^2}{ds^2}+4\pi^2s^2,
 \qquad
 \zeta_{\mathcal H}(2)=\frac{1}{32}.
\]
Because the real Heisenberg Plancherel formula contains both nonzero central
signs, the coefficient obtained above is
\[
 2\zeta_{\mathcal H}(2)=\frac{1}{16},
\]
in agreement with~\eqref{standard-heis-origin}.  This check also explains why
retaining only one central sign would miss a factor of two.  If the positive
operator is replaced by \(\Delta/2\), time scaling multiplies the coefficient
by \(4\).

\subsection{A local central limit theorem for heat kernels}
\label{heatclt}

We now record the local central-limit profile which is implicit in the
preceding Heisenberg computation.  This is the local, pointwise statement
behind the fixed-point coefficient in Theorem~\ref{leadingsecond}.  It is the
analogue, for the Heisenberg covering, of the local limit theorem of
Kotani--Sunada for abelian coverings~\cite{KotaniSunada1,KotaniSunada2}.

Let \(p,q\) be lifted points in the covering and let \(\gamma_t\in
{\rm Heis}_3(\mathbb Z)\) be a family of deck transformations.  We use the
right Lie integral along a lifted path from \(\gamma_t q\) to \(p\) and write
\begin{align*}
I_{1,t}&=\int_{\gamma_t q}^{p}\omega_1,
&
I_{2,t}&=\int_{\gamma_t q}^{p}\omega_2,\\
I_{12,t}&=\int_{\gamma_t q}^{p}(\omega_{12}+\omega_2\omega_1).
\end{align*}
Assume that the Heisenberg-scaled displacement has a limit
\begin{equation}
  \frac{I_{1,t}}{\sqrt t}\to a,
  \qquad
  \frac{I_{2,t}}{\sqrt t}\to b,
  \qquad
  \frac{I_{12,t}}{t}\to c .
\label{heis-local-scaling}
\end{equation}
The central coordinate is scaled by \(t\), rather than by \(\sqrt t\), in
accordance with the Heisenberg dilation.

For \(w>0\), set
\begin{equation}
  \mathcal W_{a,b,c}(w)
  =
  \exp(2\pi\iu\,c w^2)
  \exp(2\pi\iu\,bws)
  \exp(awD_s),
  \qquad D_s=\frac{d}{ds}.
\label{heis-local-weyl}
\end{equation}
The order of the last two factors is the order coming from the right
Lie-integral convention and the scaled Schr\"odinger model.  If one combines
these factors into a single exponential, the usual Baker--Campbell--Hausdorff
correction supplies the familiar central term.

\begin{theorem}[Local central-limit profile for the Heisenberg heat kernel]
\label{HeatCLT}
Assume, in addition to \eqref{heis-local-scaling}, that
Proposition~\ref{rational-spectral-transfer} and the oscillator majorant hold
uniformly for the moving displacement parameters \((a,b,c)\) in the compact
set under consideration.  This is the only extra uniformity beyond the fixed
point theorem.  The density is still the exact one in
Proposition~\ref{model-density}.  Then the
leading local profile is
\begin{equation}
 t^2 k_X(t,p,\gamma_t q)
 \longrightarrow
 \mathcal K_{\rm Heis}(a,b,c),
\label{heis-local-clt-limit}
\end{equation}
where \(\mathcal W_{a,b,c}(w)\) is the Weyl-type operator defined in
\eqref{heis-local-weyl}, and
\begin{equation}
\mathcal K_{\rm Heis}(a,b,c)
=
4\operatorname{Re}
\sum_{i=0}^{\infty}
\int_0^{\infty}
 e^{-\lambda_i^{(2)}w^2}
 \left(\mathcal W_{a,b,c}(w)\psi_i,\psi_i\right)_{L^2(\mathbb R)}
 w^3\,dw .
\label{heis-local-profile}
\end{equation}
Here \(\lambda_i^{(2)}=\mu_i/\operatorname{vol}(M)\), and
\(\psi_i\) is the normalized eigenfunction of the oscillator
\(\mathcal H\) in \eqref{HC}.  The real part reflects the symmetric
contribution of the two signs of the central parameter.
\end{theorem}

\begin{proof}
In the leading Heisenberg expression used in the proof of Theorem~\ref{leadingsecond}, keep only the leading part of the lowest
cluster.  Put \(h^{1/2}=u_i+O(u_i^3)\) and then set
\(w=\sqrt t\,u_i\).  The eigenvalue factor gives
\(e^{-\lambda_i^{(2)}w^2}\), while the representation factor
\(S_h(p,\gamma_tq)=\rho_h({}_R\!\int_{\gamma_tq}^{p}\omega)\) converges,
under \eqref{heis-local-scaling}, to \(\mathcal W_{a,b,c}(w)\).  On the positive branch the Plancherel factor \(h\,dh\) contributes
\(2w^3t^{-2}dw\) at leading order.  Pairing the negative branch by complex
conjugation gives the factor \(4\operatorname{Re}\) in
\eqref{heis-local-profile}.  The moving-point uniform transfer estimate and the displayed oscillator
majorant justify passage to the limit and termwise summation.
\end{proof}

For fixed lifted points, or for deck transformations with sub-diffusive
Heisenberg displacement, one has \((a,b,c)=(0,0,0)\).  Since
\[
4\int_0^{\infty} e^{-\lambda_i^{(2)}w^2}w^3\,dw
=\frac{2}{(\lambda_i^{(2)})^2},
\]
Theorem~\ref{HeatCLT} gives
\begin{equation}
\mathcal K_{\rm Heis}(0,0,0)
=
2\sum_{i=0}^{\infty}\frac{1}{(\lambda_i^{(2)})^2}
=
\frac{\operatorname{vol}(M)^2}
{16\|\omega_1\|_{L^2(M)}^2\|\omega_2\|_{L^2(M)}^2},
\label{heis-local-at-origin}
\end{equation}
which is exactly the leading coefficient in Theorem~\ref{leadingsecond}.
Thus the fixed-class heat coefficient is the value at the origin of the local
Heisenberg central-limit profile.

\subsection
[Leading and higher-order terms: nilpotent groups]
{Leading term and higher-order corrections
for general torsion-free nilpotent groups}
\label{leadingnilpotent}

We now derive Theorem~\ref{conj-heat} from the exact rational transfer and
the all--order Rockland summability theorem for a general finitely generated
torsion-free nilpotent group \(\Gamma\).  Let \(G\) be the Malcev
completion of \(\Gamma\), and let \(d\) be the polynomial growth
degree of \(\Gamma\).

The proof follows the same structure as in the Heisenberg case.  The
Pytlik-type inversion functional, called the
Pytlik-type functional in the companion foundation paper, expresses the heat
kernel as a normalized trace functional over finite-dimensional rational
representations.  Its geometric picture is the rational finite-dimensional skeleton inside
the ambient Kirillov--Bloch parameter space, while the finite-quotient model
fixes the algebraic normalization of the same functional.  Sunada's pointwise lower bound and positivity localize the normalized heat
observable; Proposition~\ref{rational-spectral-transfer} performs the
high-denominator replacement inside that neighborhood.

The localized rational finite-dimensional fibers then have large-height normal
forms described by the Kirillov models \(\pi_l\) of the Malcev completion \(G\).
Thus the asymptotic analysis is organized by the twisted Laplacians associated
with these smooth normal forms, while the trace identity remains
finite-dimensional and rational.  We choose a measurable Plancherel-regular Kirillov
cross-section
\[
\Sigma_{\rm reg}\subset \mathfrak g^*
\]
for the coadjoint-orbit parameters which occur in the Kirillov--Fujiwara
Plancherel formula.  Choose homogeneous polar coordinates on this cross-section,
written as
\[
 l=\delta_{\sqrt r}\theta,
 \qquad
 0<r<\varepsilon,
 \qquad
 \theta\in\Theta,
\]
where \(\Theta\) is a compact angular cross-section.  By
Proposition~\ref{model-density}, the principal localized action of the selected
geometric rational functional on the continuous model observables is the
homogeneous density
\begin{equation}
J(\theta)r^{d/2-1}\,dr\,d\theta .
\label{model-angular-density}
\end{equation}
The factor \(J(\theta)\) contains the Pfaffian/Jacobian contribution and the
normalization coming from the dimension and multiplicity factors of the
finite-dimensional rational fibers.

As explained in Section~\ref{asymnilpotent}, after conjugation by the section
\[
S_{\pi_l}(p)
=
\pi_l\!\left({}_R\!\int_{p_0}^{p}\omega\right),
\]
the lowest cluster of the twisted Laplacian has an expansion whose first
nonzero term is governed by the canonical hypoelliptic model operator.  For
\(\theta\in\Theta\) we denote this operator by \(\mathcal H_\theta\), and write
\[
\mathcal H_\theta\psi_i(\theta)=\mu_i(\theta)\psi_i(\theta).
\]
The corresponding lowest eigenvalues have the form
\[
\lambda_i(r,\theta)
=
\frac{r}{\operatorname{vol}(M)}\,\mu_i(\theta)
+
\text{higher-order terms}.
\]
Let \(a_{0,i}(p,q;\theta)\) be the leading local amplitude obtained from the
normalized eigenprojection and the Lie-integral gauge.  The leading heat
contribution is therefore
\[
\int_{\Theta}\int_0^\varepsilon
\sum_i a_{0,i}(p,q;\theta)
\exp\!\left(-\frac{t r\mu_i(\theta)}{\operatorname{vol}(M)}\right)
J(\theta)r^{d/2-1}\,dr\,d\theta,
\]
up to the usual lower-order amplitudes and exponentially small terms.  The
ordinary Laplace integral
\[
\int_0^\infty e^{-t r\mu/\operatorname{vol}(M)}r^{d/2-1}\,dr
=
\Gamma(d/2)\operatorname{vol}(M)^{d/2}t^{-d/2}\mu^{-d/2}
\]
gives
\begin{equation}
C_{\rm heat}(p,q)
=
\Gamma(d/2)\operatorname{vol}(M)^{d/2}
\int_{\Theta}J(\theta)
\sum_i a_{0,i}(p,q;\theta)\mu_i(\theta)^{-d/2}\,d\theta .
\label{general-heat-leading}
\end{equation}
This is the coefficient stated in Theorem~\ref{conj-heat}.  If
\(\mathcal H_\theta=\mathcal H\) and
\(a_{0,i}(p,q;\theta)=a_0(p,q)\) are independent of \(i\) and \(\theta\), then
\eqref{general-heat-leading} reduces to
\[
\Gamma(d/2)\operatorname{vol}(M)^{d/2}
 a_0(p,q)J_\Theta\zeta_{\mathcal H}(d/2),
\qquad
J_\Theta=\int_\Theta J(\theta)\,d\theta.
\]
The higher coefficients are obtained by the same recursive perturbation
procedure as in the Heisenberg case, with the one-dimensional harmonic
oscillator replaced by the model operators \(\mathcal H_\theta\) and their
oscillator-basis expansions.

Proposition~\ref{rational-spectral-transfer} justifies the finite-fiber
replacement, and Theorem~\ref{all-order-model-summability} justifies the
coefficient-weighted sums and interchanges at every prescribed order.  This
proves Theorem~\ref{conj-heat}.

\section
[Prime geodesics: nilpotent extensions]
{Asymptotics of prime closed geodesics
for nilpotent extensions of compact Riemann surfaces}
\label{Asymptoticsclosedgeodesics}

In this section we derive Theorem~\ref{theorem-geod}.  The argument combines
a twisted Selberg trace formula, following Phillips' method~\cite{Phillips},
with the same rational spectral transfer and all--order Rockland summability
used in the heat problem.  The normalization, central extraction, smoothing,
iterate estimate, and fixed-class partial summation require no additional
finitely additive hypothesis.

The structure is parallel to the heat-kernel case.  The
representation-theoretic decomposition localizes the spectral side
near the trivial representation, while the asymptotic contribution is
controlled by the same lowest eigenvalue expansion which appeared in
the heat-kernel analysis.

\subsection
[Selberg trace formula and central conjugacy classes]
{Selberg trace formula and central conjugacy classes}
\label{subsec:trace-central}

Let \(M\) be a compact Riemann surface of genus \(g\) with constant
curvature \(-1\), and let
\[
\Phi:\pi_1(M)\longrightarrow \Gamma
\]
be a surjective homomorphism onto a finitely generated torsion-free
nilpotent group \(\Gamma\).  We shall count prime closed geodesics
whose image under \(\Phi\) lies in a prescribed conjugacy class of
\(\Gamma\).

Let \(\pi_{\rm fin}\) be a finite-dimensional irreducible unitary
representation of \(\Gamma\), and let \(\Delta_{\pi_{\rm fin}}\) be
the Laplacian acting on sections of the flat bundle
\(E_{\pi_{\rm fin}}\).  We write
\[
\lambda_j(\pi_{\rm fin})
=
\frac14+r_j(\pi_{\rm fin})^2
\]
for the corresponding spectral parameters.  For an even compactly
supported \(C^\infty\) function \(h\) on \(\mathbb R\), the twisted
Selberg trace formula gives
\begin{align}
\sum_{j=0}^{\infty}\widehat h(r_j(\pi_{\rm fin}))
&=
2(g-1)
\dim(\pi_{\rm fin})
\int_{-\infty}^{\infty}
r\tanh(\pi r)\widehat h(r)\,dr
\notag\\
&\quad
+
\sum_{\gamma:\,{\rm primitive}}
\sum_{m=1}^{\infty}
\frac{
{\rm Tr}\big(\pi_{\rm fin}(\Phi(\gamma^m))\big)
\,\ell(\gamma)}
{2\sinh(m\ell(\gamma)/2)}
h(m\ell(\gamma)).
\label{Selbergtrace}
\end{align}
Here the sum over \(\gamma\) runs over primitive conjugacy classes of
\(\pi_1(M)\), identified with prime closed geodesics on \(M\), and
\(\ell(\gamma)\) denotes the length of \(\gamma\).
The displayed formula uses the standard hyperbolic denominator
\(2\sinh(m\ell(\gamma)/2)\).  If a different Selberg normalization is chosen,
the resulting harmless scalar is absorbed into the constant \(A_{\rm Sel}\)
below.

It is convenient to use the normalized trace
\[
{\rm tr}_{\pi_{\rm fin}}
=
\frac{1}{\dim(\pi_{\rm fin})}{\rm Tr}.
\]
Let \(\alpha\in\Gamma\) be central.  Since \(\pi_{\rm fin}(\alpha)\)
is a scalar operator by Schur's lemma, we have
\[
{\rm tr}_{\pi_{\rm fin}}(\alpha^{-1})
\,{\rm tr}_{\pi_{\rm fin}}(\beta)
=
{\rm tr}_{\pi_{\rm fin}}(\alpha^{-1}\beta)
\]
for every \(\beta\in\Gamma\).  Therefore the finite-dimensional
Fourier inversion formula may be modified as
\begin{equation}
f(\alpha)
=
\int_{\widehat X}
{\rm tr}_{\pi_{\rm fin},x}
\big(\pi_{{\rm fin},x}(\alpha^{-1})\pi_{{\rm fin},x}(f)\big)
\,d\mu(x)
=
\int_{\widehat X}
{\rm tr}_{\pi_{\rm fin},x}(\alpha^{-1})
{\rm tr}_{\pi_{\rm fin},x}(f)
\,d\mu(x),
\label{Fourier2}
\end{equation}
where \(d\mu\) is the Pytlik-type finitely additive Plancherel
functional on the finite-dimensional dual, with the same geometric
rational-skeleton convention and finite-quotient normalization as above.

Dividing~\eqref{Selbergtrace} by \(\dim(\pi_{\rm fin})\) and multiplying by
\({\rm tr}_{\pi_{\rm fin}}(\alpha^{-1})\) is exact.  For a compactly supported
test function, the geometric sum is finite and the spectral sum is absolutely
convergent after dimension normalization.  Applying the bounded positive
functional to this scalar identity therefore gives the dimension-normalized
averaged trace formula
\begingroup\small
\begin{align}
\label{trace}
\begin{split}
&\int_{\widehat X}
\frac{{\rm tr}_{\pi_{{\rm fin},x}}(\alpha^{-1})}
     {\dim(\pi_{{\rm fin},x})}
\left(
\sum_{j=0}^{\infty}
\widehat h(r_j(\pi_{{\rm fin},x}))
\right)
\,d\mu(x)
\\
&\qquad=
2(g-1)\delta_{\alpha,e}
\int_{-\infty}^{\infty}
r\tanh(\pi r)\widehat h(r)\,dr
\\
&\qquad\quad
+
\sum_{\gamma:\,{\rm primitive}}
\sum_{\substack{m\geq1\\ \Phi(\gamma^m)=\alpha}}
\frac{\ell(\gamma)}
{2\sinh(m\ell(\gamma)/2)}
h(m\ell(\gamma)).
\end{split}
\end{align}
\endgroup
Thus the geometric side isolates the prime geodesics whose
\(\Gamma\)-class is \(\alpha\).

\subsection{Choice of test function and localization}
\label{subsec:test-localization}

We choose the same test functions as in Phillips' argument.  Let
\(k\) be an even nonnegative compactly supported \(C^\infty\) function
on \(\mathbb R\) satisfying
\[
\int_{-\infty}^{\infty}k(s)\,ds=1.
\]
Put
\[
k_\varepsilon(s)=\varepsilon^{-1}k(s/\varepsilon),
\qquad
h_T=\chi_{[-T,T]}\ast k_\varepsilon .
\]
Then
\[
\widehat h_T(r)
=
\frac{2\sin(Tr)}{r}\widehat k(\varepsilon r),
\]
and \(\widehat k\) decays rapidly.

Sunada's inequality and positivity of the normalized functional give an
exponentially smaller averaged contribution away from a Kazhdan-neighborhood
of the trivial representation.  More precisely, there exists
\(\nu<1/2\) such that, outside a sufficiently small neighborhood \(U\)
of the trivial representation, all relevant spectral parameters
satisfy
\[
{\rm Im}\,r_j(\pi_{{\rm fin},x})\leq \nu .
\]
Hence the leading term comes only from the lowest spectral cluster
near the trivial representation.

In this neighborhood, the lowest eigenvalues of the twisted
Laplacian have the expansion
\[
\lambda_{0,i}(y)
=
q_i(y)+O(|y|^3),
\]
where \(y\) denotes local rational Plancherel parameters and \(q_i\)
is the positive quadratic form determined by the hypoelliptic model
operator.  When the localized parity symmetry is present, the cubic term
vanishes and the remainder improves to \(O(|y|^4)\).  Since
\[
\lambda_{0,i}(y)
=
\frac14+r_{0,i}(y)^2,
\]
we have
\begin{equation}
r_{0,i}(y)
=
\frac{\iu}{2}
\left(
1-2q_i(y)+O(|y|^3)
\right).
\label{asymprzero}
\end{equation}

\subsection{Analysis of the two sides of the Selberg trace formula}
\label{subsec:trace-analysis}

Let \(L(T)\) and \(R(T)\) denote the left- and right-hand sides of
the averaged trace formula~\eqref{trace} with the test function
\(h_T\).

The spectral side is dominated by the lowest branch.  Using
\eqref{asymprzero}, one obtains
\begin{align}
L(T)
&=
e^{T/2}
\int_U
\sum_i
\exp\!\left(
(-\iu r_{0,i}(y)-1/2)T
\right)
\frac{
{\rm tr}_{\pi_{{\rm fin},y}}(\alpha^{-1})
}
{-\iu r_{0,i}(y)}
\,d\mu(y)
\notag\\
&\quad
+
O\!\left(
\varepsilon e^{T/2}
+
\frac{T}{\varepsilon^2}
+
e^{\nu T}
\right).
\label{lefthandside}
\end{align}
The displayed error is the standard Phillips--Sarnak smoothing error; its
terms represent transform decay and the spectral contribution outside \(U\).

On the geometric side, the standard estimate for iterates and the same
smoothing convention give
\begin{equation}
R(T)
=
\sum_{\substack{\gamma:\,{\rm primitive}\\
\Phi(\gamma)=\alpha\\
\ell(\gamma)\leq T}}
\frac{\ell(\gamma)}
{2\sinh(\ell(\gamma)/2)}
+
O\!\left(
\frac{T^2}{\varepsilon^4}
+
\sqrt{\varepsilon}\,e^{T/2}
\right).
\label{righthandside}
\end{equation}

Choosing \(\varepsilon=e^{-\delta T}\) with \(\delta>0\) sufficiently
small and comparing~\eqref{lefthandside} and~\eqref{righthandside},
we obtain
\begin{align}
\sum_{\substack{\gamma:\,{\rm primitive}\\
\Phi(\gamma)=\alpha\\
\ell(\gamma)\leq T}}
\frac{\ell(\gamma)}
{2\sinh(\ell(\gamma)/2)}
&=
e^{T/2}
\int_U
\sum_i
\exp\!\left(
(-\iu r_{0,i}(y)-1/2)T
\right)
\frac{
{\rm tr}_{\pi_{{\rm fin},y}}(\alpha^{-1})
}
{-\iu r_{0,i}(y)}
\,d\mu(y)
\notag\\
&\quad
+
O(e^{\nu_1T})
\label{LR}
\end{align}
for some \(\nu_1<1/2\).

The integral on the right is evaluated by the same finite-order Laplace
method as in the heat-kernel case.  Propositions~\ref{model-density} and
\ref{rational-spectral-transfer}, together with
Theorem~\ref{all-order-model-summability}, give the weighted spectral expansion
\begin{equation}
 L(T)
 =
 e^{T/2}T^{-d/2}
 \left(
 \sum_{j=0}^{N}A_jT^{-j/2}+o(T^{-N/2})
 \right).
\label{finite-order-weighted-geodesic}
\end{equation}
The leading coefficient contains the same spectral-zeta integral that appears
in the heat-kernel asymptotics.  If the localized parity symmetry removes the
odd half-orders, only integral powers of \(T^{-1}\) occur.

Finally, the standard partial-summation and iterate estimate converts
\eqref{finite-order-weighted-geodesic} into
\begin{equation}
 \pi(T,\Phi,\alpha)
 =
 \frac{e^T}{T^{1+d/2}}
 \left(
 \sum_{j=0}^{N}C_j^{\rm geo}(\alpha)T^{-j/2}
 +o(T^{-N/2})
 \right).
\label{finite-order-unweighted-geodesic}
\end{equation}
This proves Theorem~\ref{theorem-geod}.

\subsection
[Heisenberg extensions: detailed asymptotics]
{More detail on the Heisenberg extension}
\label{AsymptoticsHeisenberggeodesics}

We now specialize to
\[
\Gamma={\rm Heis}_3(\mathbb Z).
\]
The purpose of this subsection is to record how the subleading terms
are obtained from the explicit expansion of the lowest spectral
branch.

For the Heisenberg extension, the lowest eigenvalues admit the
expansion
\begin{equation}
\lambda_{0,i}(x)
=
\lambda_i^{(2)}x^2
+
\lambda_i^{(4)}x^4
+
O(x^6),
\label{geod-lambda-heis}
\end{equation}
up to the normalization convention for the local parameter \(x\).
Accordingly,
\begin{align}
r_{0,i}(x)
&=
\frac{\iu}{2}
\left(
1-2\lambda_{0,i}(x)+O(\lambda_{0,i}(x)^2)
\right)
\notag\\
&=
\frac{\iu}{2}
\left(
1-2\lambda_i^{(2)}x^2
-
2\bigl(\lambda_i^{(4)}+(\lambda_i^{(2)})^2\bigr)x^4
+
O(x^6)
\right).
\label{higherasymprzero}
\end{align}
Substitution into~\eqref{lefthandside}, followed by rational spectral
transfer, all--order oscillator summability, and the standard fixed-class
Selberg steps, gives the expansion to every prescribed finite order.  The
leading coefficient, with the Selberg normalization used in
Theorem~\ref{theorem-geod}, is
\begin{equation}
C_{\rm Heis}
=
A_{\rm Sel}\,\operatorname{vol}(M)^2\,
\zeta_{\mathcal H}(2).
\label{leadingcoefficient}
\end{equation}
Since
\[
\zeta_{\mathcal H}(2)=\frac{1}{32A^2B^2}
\]
in the Heisenberg model, and since
\[
{\rm vol}(M)=4\pi(g-1)
\]
for a compact hyperbolic surface of curvature \(-1\), this gives the explicit
constant stated in Theorem~\ref{heisenberg-geod} once the horizontal basis and
Selberg normalization are fixed.

The first correction term is obtained by inserting the fourth
coefficient \(\lambda_i^{(4)}\), computed in
Section~\ref{concrete}, into the expansion of
\(r_{0,i}(x)\) and then summing over the harmonic oscillator index
\(i\).  Thus the coefficient \(c_1\) is not a single contribution
from one eigenvalue branch; it is the sum of the corresponding
branchwise corrections.  Schematically,
\begin{equation}
c_1
=
\sum_{i\geq0}
{\mathcal C}_i
\left(
\lambda_i^{(2)},\lambda_i^{(4)},
\omega_1,\omega_2,\omega_{12},G_M
\right),
\label{subleadingcoefficient}
\end{equation}
where \(G_M\) is the Green operator of \(\Delta_M\), and the explicit
branchwise expression is obtained from~\eqref{lambda4-matrix}.  In
particular, \(c_1\) is determined by the harmonic components
\(\omega_1,\omega_2\), the coexact central component
\(\omega_{12}\), and the Green operator on \(M\).

\subsection{The first correction in the Heisenberg central geodesic theorem}
\label{subsec:heis-geod-second-term}

We also record the first inverse-length correction in the Heisenberg central
geodesic theorem.  This is the closed-geodesic analogue of
Theorem~\ref{leadingsecond}.  The only new point, compared with the heat
kernel calculation, is the relation between the lowest eigenvalue and the
Selberg spectral parameter.

With the local parameter normalized as in the Phillips-Sarnak trace formula,
we write
\begin{equation}
\lambda_{0,i}(x)
=
\frac{\lambda_i^{(2)}}{2}x^2+
\frac{\lambda_i^{(4)}}{24}x^4+O(x^6).
\label{asymplambda-heis-geod-2}
\end{equation}
Since
\[
\lambda_{0,i}(x)=\frac14+r_{0,i}(x)^2,
\]
the branch with \(r_{0,i}(0)=\iu/2\) satisfies
\begin{equation}
r_{0,i}(x)
=
\frac{\iu}{2}
\left(
1-\lambda_i^{(2)}x^2-
\left(\frac{\lambda_i^{(4)}}{12}
+\frac{(\lambda_i^{(2)})^2}{2}\right)x^4+O(x^6)
\right).
\label{higherasymprzero-heis-geod-2}
\end{equation}
Substitution into the spectral side of the averaged Selberg trace formula,
followed by the ordinary Laplace calculation and the standard partial
summation step, gives
\begin{equation}
\pi(T,\Phi,\alpha)
\sim
\frac{C_{\rm geo}\,e^T}{T^3}
\left(1+\frac{c_{1,\rm geo}(\alpha)}{T}+O(T^{-2})\right)
\label{heis-geod-two-term}
\end{equation}
for central \(\alpha\in {\rm Heis}_3(\mathbb Z)\).  The leading coefficient is
\begin{equation}
C_{\rm geo}
=
A_{\rm Sel}\,\operatorname{vol}(M)^2\,
\zeta_{\mathcal H}(2),
\label{heis-geod-C}
\end{equation}
with the same \(A_{\rm Sel}\) convention as in Theorem~\ref{theorem-geod}.  With
our displayed Selberg trace formula \eqref{Selbergtrace}, \(A_{\rm Sel}=1\);
under another convention it is the explicit scalar converting that convention
to the one used here.  For a compact
hyperbolic surface of genus \(g\) and a fixed standard normalization of the
horizontal basis, this specializes to the corresponding numerical constant.
The first correction is obtained by inserting the corrected quartic coefficient
in~\eqref{higherasymprzero-heis-geod-2}, together with the next amplitude,
smoothing, and weighted-to-unweighted terms.  In particular, the branchwise
spectral contribution depends on the combination
\[
\frac{\lambda_i^{(4)}}{12}
+\frac{(\lambda_i^{(2)})^2}{2},
\]
not on \(\lambda_i^{(4)}\) alone.
Proposition~\ref{rational-spectral-transfer} and the explicit oscillator
bounds give the absolutely convergent normalized sum of these branchwise
contributions, weighted by the central character factor and the corresponding
Selberg amplitudes.  Formula~\eqref{lambda4-matrix}, with its
\(\operatorname{vol}(M)^{-1/2}\) normalization, expresses the genuinely
nilpotent part in terms of \(\omega_1,\omega_2,\omega_{12}\), and the Green
operator \(G_M\).  The smoothing and weighted-to-unweighted corrections are
the standard fixed-class Selberg terms and are kept explicit rather than
compressed into an unspecified bridge constant.  In particular, the same
Green-operator coefficient functions and
the same coexact central term appearing through \(L^{(4)}\) also occur in the
first correction to the central Heisenberg geodesic asymptotic.

\subsection{Non-central conjugacy classes in the Heisenberg group}
\label{subsec:noncentral-heisenberg}

We next comment on non-central conjugacy classes in
\({\rm Heis}_3(\mathbb Z)\).  Let
\[
\gamma=u^a v^b w^c
\]
in terms of the standard generators.  A direct computation gives
\[
u\gamma u^{-1}=\gamma w^b,
\qquad
v\gamma v^{-1}=\gamma w^{-a},
\qquad
w\gamma w^{-1}=\gamma .
\]
Hence, if
\[
r=\gcd(a,b),
\]
then the conjugacy class of \(\gamma\) is
\[
[\gamma]
=
\{\gamma w^{rn}\mid n\in\mathbb Z\}.
\]
Thus a non-central conjugacy class projects to a nontrivial element
of the abelianization
\[
{\rm Heis}_3(\mathbb Z)/
[{\rm Heis}_3(\mathbb Z),{\rm Heis}_3(\mathbb Z)]
\simeq
\mathbb Z^2,
\]
together with an arithmetic progression in the central direction.

For this reason, the leading distribution of non-central classes is
governed primarily by the abelian quotient.  The genuinely nilpotent
phenomena studied above are concentrated in central conjugacy
classes.  A treatment of non-central classes would have to combine the abelian
displacement analysis of Lalley~\cite{Lalley}, Babillot--Ledrappier~\cite{Babillot},
and Anantharaman~\cite{Anantharaman2} with uniform control of the remaining
central progression.  That problem is outside the scope of the present
central-class theorem.

\subsection
[Local CLT for closed geodesics]
{A local central limit theorem for closed geodesics
in Heisenberg extensions}
\label{subsec:geodesic-clt}

We finally record the local central-scaling profile which appears on the
Selberg side in the Heisenberg case.  In abelian extensions, local central
limit theorems for prime closed geodesics were obtained by
Lalley~\cite{Lalley} and refined by Babillot--Ledrappier~\cite{Babillot} and
Anantharaman~\cite{Anantharaman2}.  The statement below is the corresponding
Heisenberg model profile.  Unlike the fixed central class theorem, a moving
class requires uniformity when \(m_T/T\) varies.  This extra moving-parameter
uniformity is stated explicitly below; the smoothing, iterate, and
weighted-to-unweighted steps are otherwise the same as in the fixed-class
argument.

Let \(w\) denote the standard central generator of
\({\rm Heis}_3(\mathbb Z)\), and let
\[
   \alpha_T=w^{m_T},
   \qquad
   \frac{m_T}{T}\longrightarrow c\in\mathbb R .
\]
The scaling by \(T\) is the natural central scaling: the center has degree two
under the Heisenberg dilation.  On a fiber with central parameter \(h\), the
central character contributes
\[
   \rho_h(\alpha_T^{-1})=\exp(-2\pi\iu\,m_T h).
\]
After the lowest-branch change of variables \(h=u_i^2\) and
\(\xi=\sqrt T\,u_i\), this factor converges to
\(\exp(-2\pi\iu\,c\xi^2)\).

Define the central profile
\begin{equation}
\mathcal G_{\rm Heis}(c)
=
4\operatorname{Re}
\sum_{i=0}^{\infty}
\int_0^{\infty}
 e^{-\lambda_i^{(2)}\xi^2}
 e^{-2\pi\iu\,c\xi^2}
 \xi^3\,d\xi .
\label{geod-local-profile-integral}
\end{equation}
Equivalently,
\begin{equation}
\mathcal G_{\rm Heis}(c)
=
2\operatorname{Re}
\sum_{i=0}^{\infty}
\frac{1}{(\lambda_i^{(2)}+2\pi\iu\,c)^2}.
\label{geod-local-profile-zeta}
\end{equation}
In particular,
\(\mathcal G_{\rm Heis}(0)=2\sum_i(\lambda_i^{(2)})^{-2}\), the fixed
central-class heat coefficient before the weighted-to-unweighted geodesic
conversion.

Let
\[
L(T,\Phi,w^{m_T})
=
\sum_{\substack{\gamma:\,{\rm primitive}\\
\Phi(\gamma)=w^{m_T}\\
\ell(\gamma)\leq T}}
\frac{\ell(\gamma)}{2\sinh(\ell(\gamma)/2)}
\]
be the weighted geometric sum which occurs on the geometric side of the
normalized Selberg trace formula.

\begin{proposition}[Local central profile on the weighted Selberg side]
\label{GeodesicCLT}
With the notation above, the model contribution of the moving central class
\(w^{m_T}\) satisfies
\begin{equation}
T^2e^{-T/2}L(T,\Phi,w^{m_T})
\longrightarrow
A_{\rm Sel}\,\mathcal G_{\rm Heis}(c),
\label{WeightedGeodesicCLT-limit}
\end{equation}
provided Proposition~\ref{rational-spectral-transfer} and the oscillator
majorant hold uniformly for every sequence \(m_T/T\to c\) under consideration,
including the oscillatory central-character factor.  This moving-class
uniformity is additional to the fixed-central-class theorem.  The profile
\(\mathcal G_{\rm Heis}(c)\) is the oscillatory
harmonic-oscillator sum defined in \eqref{geod-local-profile-integral},
equivalently in \eqref{geod-local-profile-zeta}.  The constant
\(A_{\rm Sel}\) has the same normalization as in Theorem~\ref{theorem-geod}.
Consequently, under the corresponding weighted-to-unweighted Tauberian
conversion, the prime counting function has the local central profile
\begin{equation}
\frac{T^3}{e^T}\,\pi(T,\Phi,w^{m_T})
\longrightarrow
\frac12 A_{\rm Sel}\,\mathcal G_{\rm Heis}(c).
\label{GeodesicCLT-limit}
\end{equation}
\end{proposition}

\begin{proof}
The calculation is the fixed central class Laplace calculation with one
additional central-character factor.  The spectral side of the normalized
Selberg formula contains the lowest branches \(r_{0,i}(x)\) and the normalized
trace of \(\rho_h(w^{-m_T})\).  In the trivial-neighborhood scaling, the
exponential term gives \(e^{-\lambda_i^{(2)}\xi^2}\), while the central character
gives \(e^{-2\pi\iu c\xi^2}\).  On the positive branch the Plancherel factor
contributes \(2\xi^3d\xi\); pairing the negative branch by complex conjugation
gives the factor \(4\operatorname{Re}\) in
\eqref{geod-local-profile-integral}.  The elementary integral
\[
4\int_0^\infty
 e^{-(\lambda_i^{(2)}+2\pi\iu c)\xi^2}\xi^3d\xi
=2(\lambda_i^{(2)}+2\pi\iu c)^{-2}
\]
gives \eqref{geod-local-profile-zeta}.  The final factor \(1/2\) in
\eqref{GeodesicCLT-limit} is the same weighted-to-unweighted conversion used
in the fixed-class prime geodesic theorem.
\end{proof}

Thus the fixed-class nilpotent Chebotarev constant is the value at the origin
of a central local-limit profile, and the same harmonic oscillator which
controls the heat kernel controls central geodesic fluctuations.

\section*{Appendices}
\addcontentsline{toc}{section}{Appendices}

\appendix

\section{Spectral computations and special functions}
\label{AppendixC}

This appendix collects the analytic tools needed to evaluate, or at
least represent explicitly, the spectral zeta values which occur as
leading constants in the finite-order asymptotic formulas above.

Much of the special-function material in this appendix is adapted, with
changes of notation and normalization, from Appendix~C of
Voros~\cite{Voros}.  In particular, the zero-energy resolvent construction,
its reduction to modified Bessel functions, and the cyclic resolvent formulas
for spectral zeta values follow that source.  The Heisenberg normalization,
the heat-kernel contour formula, and the interpretation of these formulas for
the present covering problems are included here to connect Voros's
calculations with the applications in the main text.

In the long-time heat-kernel asymptotics and in the counting formulas
for prime closed geodesics on nilpotent coverings, the leading
constant is expressed in terms of the spectral zeta function of a
canonical hypoelliptic operator associated with the underlying
nilpotent group.  The value at which this zeta function is evaluated
is determined by the polynomial growth degree of the group.

For general nilpotent groups, the canonical hypoelliptic operators are
multivariable differential operators with polynomial coefficients.
They reflect the stratification and the step of the corresponding Lie
algebra.  Closed-form expressions for their eigenvalues are not
available in general.  Explicit evaluations are therefore exceptional.

We illustrate this point by comparing two basic examples.  For the
three-dimensional Heisenberg group, the model operator reduces to the
harmonic oscillator.  For the Engel group, which is the next simplest
non-Heisenberg nilpotent example, the same construction naturally
leads to a quartic oscillator.  The latter already requires
resolvent-based integral representations.

In the one-dimensional models considered here, the reduced operator
has the form
\[
\widehat H_M
=
-\frac{d^2}{dq^2}+q^{2M},
\]
with
\[
M=1
\quad\text{for the Heisenberg model},
\qquad
M=2
\quad\text{for the Engel model}.
\]

\subsection{The Heisenberg case}

For \(M=1\), the operator
\[
\widehat H_1
=
-\frac{d^2}{dq^2}+q^2
\]
is the harmonic oscillator.  With the normalization
\[
{\rm Spec}(\widehat H_1)=\{2n+1\mid n=0,1,2,\ldots\},
\]
one obtains
\[
\zeta_{\widehat H_1}(2)
=
\sum_{n=0}^{\infty}\frac{1}{(2n+1)^2}.
\]
Thus
\[
\zeta_{\widehat H_1}(2)
=
\sum_{n=1}^{\infty}\frac{1}{n^2}
-
\sum_{n=1}^{\infty}\frac{1}{(2n)^2}
=
\frac34\zeta(2)
=
\frac{\pi^2}{8}.
\]
In the normalization used in the main text, the harmonic oscillator
contains additional scaling factors coming from
\(\|\omega_1\|_{L^2(M)}\), \(\|\omega_2\|_{L^2(M)}\), and \(2\pi\).
After these factors are included, this gives the coefficient appearing
in the Heisenberg heat-kernel and prime-geodesic formulas.

\subsection{The Engel case and the quartic oscillator}

For the Engel group, the relevant one-dimensional model is the quartic
oscillator
\[
\widehat H_2
=
-\frac{d^2}{dq^2}+q^4 .
\]
Unlike the harmonic oscillator, this operator does not admit a
closed-form description of its individual eigenvalues.  The spectral
zeta function
\[
\zeta_{\widehat H_2}(s)
=
\sum_{k=1}^{\infty}\lambda_k^{-s},
\qquad
{\rm Spec}(\widehat H_2)=\{\lambda_k\}_{k\geq1},
\]
is therefore not evaluated by summing an explicit eigenvalue formula.

In the Engel application, the relevant value is
\[
\zeta_{\widehat H_2}(7/2).
\]
It is more natural to express this value by means of the resolvent of
\(\widehat H_2\).  We recall the construction in a slightly more general
form.  The resolvent and Bessel-function calculations in the next
subsections follow Appendix~C of Voros~\cite{Voros}, with the sign convention
adjusted to \(R(E)=(\widehat H_M-E)^{-1}\).

\subsection[Resolvent kernel for anharmonic oscillators]{Resolvent kernel for \( -d^2/dq^2+q^{2M} \)}

Let
\[
\widehat H_M
=
-\frac{d^2}{dq^2}+q^{2M}
\]
on \(L^2(\mathbb R)\), with \(M\geq1\).  Its spectrum is purely
discrete, and the resolvent
\[
R(E)
=
(\widehat H_M-E)^{-1}
\]
has an integral kernel
\[
R(E;q,q')
=
W(E)^{-1}
\psi_-(E;q_<)\psi_+(E;q_>),
\]
where
\[
q_< = \min(q,q'),
\qquad
q_> = \max(q,q'),
\]
\(\psi_+\) and \(\psi_-\) are solutions of
\[
(\widehat H_M-E)\psi=0
\]
which are recessive at \(+\infty\) and \(-\infty\), respectively.  We use the
Wronskian convention
\[
W(E)=\psi_+(E;q)\psi_-'(E;q)-\psi_+'(E;q)\psi_-(E;q),
\]
which is independent of \(q\) and makes the displayed resolvent-kernel sign
consistent with \(R(E)=(\widehat H_M-E)^{-1}\).

At \(E=0\), the equation
\[
-\psi''(q)+q^{2M}\psi(q)=0
\]
can be reduced to a modified Bessel equation.  For \(q>0\), set
\[
z=\frac{q^{M+1}}{M+1},
\qquad
\mu=\frac{1}{2M+2}.
\]
Then a fundamental system is given by
\[
q^{1/2}I_{\mu}(z),
\qquad
q^{1/2}K_{\mu}(z),
\]
up to normalization constants.  The recessive solution at \(+\infty\)
is proportional to
\[
q^{1/2}K_{\mu}\!\left(\frac{q^{M+1}}{M+1}\right).
\]
The corresponding recessive expression on the negative half-line is
proportional to
\[
|q|^{1/2}K_{\mu}\!\left(\frac{|q|^{M+1}}{M+1}\right),
\]
with the parity determined by the chosen global solution.

For reference, we recall the definitions of the Bessel functions used
above; see Erd\'elyi--Magnus--Oberhettinger--Tricomi~\cite{Erdelyi}.  Bessel's equation is
\[
z^2w''+zw'+(z^2-\mu^2)w=0.
\]
For \(\mu\notin\mathbb Z\), two independent solutions are
\(J_\mu(z)\) and \(J_{-\mu}(z)\), where
\[
J_\mu(z)
=
\sum_{m=0}^{\infty}
\frac{(-1)^m}{m!\Gamma(m+\mu+1)}
\left(\frac z2\right)^{2m+\mu}.
\]
Replacing \(z\) by \(iz\) gives the modified Bessel equation
\[
z^2w''+zw'-(z^2+\mu^2)w=0.
\]
The modified Bessel functions are
\[
I_\mu(z)
=
e^{-\frac12 i\mu\pi}
J_\mu(e^{\frac12 i\pi}z),
\]
and
\[
K_\mu(z)
=
\frac{\pi}{2\sin(\pi\mu)}
\left(
I_{-\mu}(z)-I_\mu(z)
\right).
\]

\subsection{Engel heat kernel from the resolvent}
\label{subsec:engel-heat-resolvent}

For the Engel model one has \(M=2\), hence
\[
\widehat H_2=-\frac{d^2}{dq^2}+q^4 .
\]
Let \(K_2(t;q,q')\) denote its heat kernel.  Since \(\widehat H_2\) is a
positive self-adjoint one-dimensional Schr\"odinger operator with compact
resolvent,
\begin{equation}
K_2(t;q,q')=
\sum_{j=0}^{\infty}e^{-t\lambda_j}\phi_j(q)\overline{\phi_j(q')},
\label{engel-heat-spectral}
\end{equation}
where \(\widehat H_2\phi_j=\lambda_j\phi_j\).  Equivalently, by the
Dunford--Taylor functional calculus,
\begin{equation}
K_2(t;q,q')=
-\frac{1}{2\pi\iu}
\int_{\mathcal C}e^{-tE}R(E;q,q')\,dE,
\label{engel-heat-contour}
\end{equation}
where \(\mathcal C\) is a contour surrounding the spectrum and
\(R(E;q,q')\) is the resolvent kernel described above.

Thus the Engel heat kernel is explicit in the following precise sense: it is
represented by the resolvent kernel built from the two recessive solutions of
\((\widehat H_2-E)\psi=0\).  At \(E=0\) these recessive solutions reduce to
modified Bessel functions.  In this case
\[
  z=\frac{q^3}{3},
  \qquad
  \mu=\frac16,
\]
and the recessive solution on the positive half-line is proportional to
\begin{equation}
q^{1/2}K_{1/6}\!\left(\frac{q^3}{3}\right).
\label{engel-bessel-recessive}
\end{equation}
The reflected solution gives the corresponding expression on the negative
half-line.

This is the Engel analogue of the explicit Heisenberg heat-kernel check in
Section~\ref{comparison-to-explicit-formula}.  The difference is that the
Heisenberg harmonic oscillator has an elementary eigenvalue formula, whereas
the Engel quartic oscillator does not.  Therefore the comparison with the
asymptotic coefficient is made at the level of the spectral invariant
\(\zeta_{\widehat H_2}(7/2)\), or equivalently at the level of the resolvent and
heat-kernel representations \eqref{engel-heat-spectral}--\eqref{engel-heat-contour},
rather than by reducing the coefficient to a single elementary number.

\subsection{Resolvent representation of spectral zeta values}

Let
\[
R=\widehat H_M^{-1}
\]
be the resolvent at zero.  For positive integers \(n\), the spectral
zeta values can be written as
\[
\zeta_{\widehat H_M}(n)
=
{\rm Tr}(\widehat H_M^{-n})
=
{\rm Tr}(R^n).
\]
If \(R(q,q')\) denotes the kernel of \(R\), then
\[
\zeta_{\widehat H_M}(n)
=
\int_{\mathbb R^n}
R(q_1,q_2)R(q_2,q_3)\cdots R(q_n,q_1)
\,dq_1\cdots dq_n .
\]
Thus integer spectral zeta values are represented by cyclic integrals
of the resolvent kernel.

For the Engel group, the value required in the main text is
\[
\zeta_{\widehat H_2}(7/2)
=
{\rm Tr}(\widehat H_2^{-7/2})
=
{\rm Tr}(R^{7/2}).
\]
This can be written as
\[
\zeta_{\widehat H_2}(7/2)
=
\int_{\mathbb R^4}
R(q_1,q_2)R(q_2,q_3)R(q_3,q_4)B(q_4,q_1)
\,dq_1\,dq_2\,dq_3\,dq_4,
\]
where \(B(q,q')\) is the integral kernel of
\[
R^{1/2}=\widehat H_2^{-1/2}.
\]
The operator \(R^{1/2}\) is defined by the spectral functional
calculus.  Equivalently, one may use the heat-kernel representation
\[
\widehat H_2^{-1/2}
=
\frac{1}{\sqrt{\pi}}
\int_0^\infty
t^{-1/2}e^{-t\widehat H_2}\,dt,
\]
or any standard convergent functional-calculus expansion for the
positive operator \(R\).  Equivalently, if \(K_M(t;q,q')\) denotes the heat
kernel of \(\widehat H_M\), then the Dunford--Taylor formula gives
\[
K_M(t;q,q')=
-\frac{1}{2\pi i}\int_{\mathcal C}e^{-tE}R(E;q,q')\,dE,
\]
where \(\mathcal C\) is a contour enclosing the positive spectrum in the usual
sectorial functional calculus.  Thus the Engel heat kernel is represented by
the same resolvent kernel whose zero-energy solutions are expressed in terms of
modified Bessel functions.  This is the Engel analogue of the explicit
Heisenberg heat-kernel check.

Consequently, even for the comparatively simple Engel group, the
constant appearing in the asymptotic formulas does not reduce to a
closed expression in elementary eigenvalue data.  Its natural form is
a resolvent-based integral representation.  This is the analytic
framework used for the Engel computations in the main text.

\section{Harmonic theory of Chen's iterated integrals:
homology connections}
\label{Homologyconnection}

This appendix gives a conceptual framework for the iterated-integral
constructions used in the main text.  The Lie-integral approach is
effective for low-step examples, such as the Heisenberg group.  For
general nilpotent groups, however, Chen's theory of iterated integrals
and homology connections gives a more invariant and systematic
language.

The material in this appendix is not needed for following the main
arguments line by line.  Its purpose is to explain the homological
structure behind the flat formal connections and generalized
Abel--Jacobi maps which appear in Sections~\ref{harmoniclietheory}
and~\ref{asymnilpotent}.  We follow the approach of K.-T.~Chen~\cite{Chen,Chen1} and
K.~Kohno~\cite{Kohno}, together with Sullivan's homological viewpoint~\cite{Sullivan}.

\subsection{Iterated integrals on path spaces}
\label{generaliteratedintegrals}

Let \(M\) be a smooth manifold.  For \(x_0,x_1\in M\), let
\[
P(M;x_0,x_1)
=
\{
\gamma:[0,1]\to M
\mid
\gamma \text{ is piecewise smooth},
\ \gamma(0)=x_0,\ \gamma(1)=x_1
\}
\]
be the space of paths from \(x_0\) to \(x_1\).  We also write
\[
PM
=
\{
\gamma:[0,1]\to M
\mid
\gamma \text{ is piecewise smooth}
\}
\]
for the free path space.

Let
\[
{\rm ev}: [0,1]\times PM\longrightarrow M,
\qquad
{\rm ev}(t,\gamma)=\gamma(t),
\]
be the evaluation map.  If \(\omega\in\mathcal A^p(M)\), then
\({\rm ev}^*\omega\) is a \(p\)-form on \([0,1]\times PM\).  By
fiber integration over \([0,1]\), it defines a \((p-1)\)-form on
\(PM\):
\[
\int\omega
:=
\int_{[0,1]}{\rm ev}^*\omega .
\]
More explicitly, if \(\phi:U\to PM\) is a finite-dimensional chart
and
\[
{\rm ev}_U:[0,1]\times U\longrightarrow M,
\qquad
{\rm ev}_U(t,u)=\phi(u)(t),
\]
then we decompose
\[
{\rm ev}_U^*\omega
=
dt\wedge\alpha+\beta .
\]
The local representative of \(\int\omega\) on \(U\) is
\[
\int_0^1\alpha(t,u)\,dt .
\]
These local forms patch together and give a well-defined form on the
path space.

Now let
\[
\omega_1,\ldots,\omega_k
\]
be differential forms on \(M\), with
\(\deg \omega_j=p_j\).  Let
\[
\Delta_k
=
\{0\leq t_1\leq\cdots\leq t_k\leq1\}
\]
be the standard simplex, and define
\[
{\rm ev}_k:\Delta_k\times PM\longrightarrow M^k
\]
by
\[
{\rm ev}_k(t_1,\ldots,t_k;\gamma)
=
(\gamma(t_1),\ldots,\gamma(t_k)).
\]
The iterated integral is
\[
\int \omega_1\cdots\omega_k
:=
\int_{\Delta_k}
{\rm ev}_k^*
(\omega_1\otimes\cdots\otimes\omega_k).
\]
It is a differential form on \(PM\) of degree
\[
p_1+\cdots+p_k-k.
\]
When all \(\omega_j\) are \(1\)-forms, this is a function on the path
space and coincides with the usual Chen iterated integral.

In particular, if
\[
\omega_j(\gamma'(t))=f_j(t),
\]
then
\[
\int_\gamma \omega_1\cdots\omega_k
=
\int_{0\leq t_1\leq\cdots\leq t_k\leq1}
f_k(t_k)\cdots f_1(t_1)\,
dt_1\cdots dt_k ,
\]
with the ordering convention used in the main text.

\subsection{Formal power series connections}

We next recall Chen's formal power series connections.  Let
\(H_+(M,\mathbb R)\) denote the positive-degree homology
\[
H_+(M,\mathbb R)
=
\bigoplus_{q>0}H_q(M,\mathbb R).
\]
Choose a homogeneous basis
\[
z_1,\ldots,z_m
\]
with \(z_j\in H_{q_j}(M,\mathbb R)\).  Let
\[
X_1,\ldots,X_m
\]
be noncommuting variables dual to this basis, and assign degree
\[
\deg X_j=q_j-1.
\]
Let
\[
\mathbb R\langle X_1,\ldots,X_m\rangle
\]
be the tensor algebra generated by the \(X_j\), and let \(J\) be its
augmentation ideal.  We write
\[
\mathbb R\langle\!\langle X_1,\ldots,X_m\rangle\!\rangle
=
\varprojlim_k
\mathbb R\langle X_1,\ldots,X_m\rangle/J^k
\]
for the completed noncommutative formal power series algebra, and
\(\widehat J\) for its completed augmentation ideal.

Let \(\mathcal A^*(M)\) be the de Rham complex.  A formal power
series connection is an element
\[
\omega
\in
\mathcal A^*(M)
\widehat\otimes
\mathbb R\langle\!\langle X_1,\ldots,X_m\rangle\!\rangle
\]
of the form
\begin{equation}
\omega
=
\sum_i\omega_iX_i
+
\sum_{i,j}\omega_{ij}X_iX_j
+
\sum_{i,j,k}\omega_{ijk}X_iX_jX_k
+\cdots .
\label{powerseriesconnection}
\end{equation}
Here the coefficients \(\omega_{i_1\cdots i_r}\) are differential
forms on \(M\).

We extend the exterior derivative by
\[
d(\eta Z)=(d\eta)Z,
\]
where \(\eta\in\mathcal A^*(M)\) and
\(Z\in\mathbb R\langle\!\langle X_1,\ldots,X_m\rangle\!\rangle\).
The wedge product is extended by
\[
(\eta_1Z_1)\wedge(\eta_2Z_2)
=
(\eta_1\wedge\eta_2)Z_1Z_2.
\]
Likewise, iterated integrals are extended by
\[
\int(\eta_1Z_1)(\eta_2Z_2)
=
\left(\int\eta_1\eta_2\right)Z_1Z_2
\]
and linearity.

Let \(\varepsilon\) denote the sign involution defined on homogeneous
forms by
\[
\varepsilon(\eta Z)
=
(-1)^{\deg\eta}\eta Z.
\]
The curvature of a formal power series connection is
\begin{equation}
\kappa
=
d\omega-\varepsilon(\omega)\wedge\omega .
\label{curvatureformal}
\end{equation}
This is Chen's convention for the curvature of a formal connection.

\subsection{Homology connections and Hodge decomposition}

Let
\[
\delta:
\mathbb R\langle\!\langle X_1,\ldots,X_m\rangle\!\rangle
\longrightarrow
\mathbb R\langle\!\langle X_1,\ldots,X_m\rangle\!\rangle
\]
be a derivation of degree \(-1\).  Thus
\[
\delta(uv)=(\delta u)v+(-1)^{\deg u}u(\delta v),
\]
and
\[
\delta X_j\in\widehat J^2.
\]
We extend \(\delta\) to formal differential forms by
\[
\delta(\eta Z)=\eta\,\delta Z .
\]

A pair \((\omega,\delta)\) is called a homology connection if it
satisfies the flatness equation
\begin{equation}
\delta\omega+\kappa=0,
\label{flatcondition}
\end{equation}
together with the degree condition
\[
\deg \omega_{i_1\cdots i_r}
=
\deg(X_{i_1}\cdots X_{i_r})+1.
\]
In other words,
\[
\deg \omega_{i_1\cdots i_r}
=
q_{i_1}+\cdots+q_{i_r}-r+1.
\]

Now suppose that \(M\) is a compact Riemannian manifold.  The Hodge
decomposition gives
\[
\mathcal A^*(M)
=
\mathcal H\oplus B_d\oplus B_{d^*},
\label{decomposition1}
\]
where \(\mathcal H\) is the space of harmonic forms, and \(B_d\),
\(B_{d^*}\) are the images of \(d\) and \(d^*\), respectively.  Put
\[
\mathcal B:=B_{d^*}.
\]
Then
\[
\mathcal A^*(M)
=
\mathcal H\oplus d\mathcal B\oplus\mathcal B .
\label{decomposition2}
\]

The following result is the homological analogue of choosing harmonic
and coexact representatives in Section~\ref{harmoniclietheory}.

\begin{proposition}[Chen--Kohno]
Assume that a decomposition
\[
\mathcal A^*(M)
=
\mathcal H\oplus d\mathcal B\oplus\mathcal B
\]
satisfies the following conditions:
\begin{description}
\item[\rm{(1)}]
Every element of \(\mathcal H\) is closed, and the inclusion
\[
\mathcal H\hookrightarrow \mathcal A^*(M)
\]
induces an isomorphism
\[
\mathcal H\simeq H_{\rm DR}^*(M).
\]

\item[\rm{(2)}]
The space \(\mathcal B\) is graded, and the only closed form in
\(\mathcal B\) is \(0\).
\end{description}
Then there exists a unique homology connection \((\omega,\delta)\)
such that
\[
\omega_i\in\mathcal H
\quad (1\leq i\leq m),
\]
and
\[
\omega_{i_1\cdots i_r}\in\mathcal B
\quad (r\geq2).
\]
\end{proposition}

This proposition generalizes Condition~\ref{harmoniccoexact}.  In the
Heisenberg case, the first-order components are harmonic, while the
central component is chosen to be coexact.

\subsection{Holonomy representation of the fundamental group}

Let \((\omega,\delta)\) be a homology connection, and let
\(\omega_0\) denote its degree-zero part.  This subsection uses Chen's
standard left-ordered convention, so the ordinary curvature below carries the
plus sign.  Then \(\omega_0\) is a
\(1\)-form on \(M\) with coefficients in the completed
noncommutative formal power series algebra.  Its usual curvature is
\[
\kappa_0
=
d\omega_0+\omega_0\wedge\omega_0.
\]

The formal parallel transport along a loop is given by the iterated
integral series
\[
T
=
1+\sum_{r=1}^{\infty}
\int \underbrace{\omega_0\cdots\omega_0}_{r}.
\]
For the based loop space \(\Omega M\), this defines a holonomy map
\begin{align}
\Theta:
C_*(\Omega M)\otimes\mathbb R
&\longrightarrow
\mathbb R\langle\!\langle X_1,\ldots,X_m\rangle\!\rangle,
\notag\\
\Theta(c)
&=
\langle T,c\rangle .
\label{holonomymap}
\end{align}
Chen's fundamental observation is that \(\Theta\) is a chain map.

The zeroth homology \(H_0(\Omega M;\mathbb R)\) has a product induced
by concatenation of loops, and is naturally identified with the group
algebra
\[
\mathbb R\pi_1(M,x_0).
\]
Let
\[
\mathcal N
=
{\rm Im}\left(
\delta:
\mathbb R\langle\!\langle X_1,\ldots,X_m\rangle\!\rangle_1
\to
\mathbb R\langle\!\langle X_1,\ldots,X_m\rangle\!\rangle_0
\right).
\]
Then \(\mathcal N\) is an ideal, and
\[
H_0\left(
\mathbb R\langle\!\langle X_1,\ldots,X_m\rangle\!\rangle,
\delta
\right)
=
\mathbb R\langle\!\langle X_1,\ldots,X_m\rangle\!\rangle/\mathcal N .
\]
Consequently, \(\Theta\) induces a homomorphism
\[
\Theta_0:
\mathbb R\pi_1(M,x_0)
\longrightarrow
\mathbb R\langle\!\langle X_1,\ldots,X_m\rangle\!\rangle/\mathcal N.
\]

Let \(J\) be the augmentation ideal of
\(\mathbb R\pi_1(M,x_0)\), and let \(\widehat J\) be the completed
augmentation ideal of
\(\mathbb R\langle\!\langle X_1,\ldots,X_m\rangle\!\rangle\).
Then Chen's theorem, in Kohno's formulation, says the following.

\begin{theorem}[Chen--Kohno {\rm \cite[Theorem 5.3.1]{Kohno}}]
\label{chenpi1derham}
For every \(k\geq0\), the holonomy homomorphism \(\Theta_0\) induces
an isomorphism of \(\mathbb R\)-algebras
\[
\mathbb R\pi_1(M,x_0)/J^{k+1}
\cong
\mathbb R\langle\!\langle X_1,\ldots,X_m\rangle\!\rangle/
(\mathcal N+\widehat J^{k+1}).
\]
\end{theorem}

Thus the completed group algebra of the fundamental group is described
by iterated integrals and the differential \(\delta\) of the homology
connection.  This is the formal homological version of the
\(\pi_1\)-de Rham theorem used in the main text.

\subsection{Sketch of the proof via the cobar construction}

We recall the main idea of the proof, since Kohno's exposition is not
widely available in English.  The argument goes back to Chen and is
closely related to Adams' cobar construction.

Let \(C_*(M)_{x_0}\) be the chain complex generated by singular
simplices whose vertices all map to the base point \(x_0\).  Let
\(F(C_*)\) be the free associative unital algebra generated by the
positive-dimensional simplices.  If \(\sigma\) is a simplex, we
write \([\sigma]\) for the corresponding generator and define
\[
\deg[\sigma]=\dim\sigma-1.
\]
The algebra \(F(C_*)\) is filtered by word length:
\[
F(C_*)=F_0(C_*)\supset F_{-1}(C_*)\supset F_{-2}(C_*)\supset\cdots,
\]
where \(F_{-k}(C_*)\) is generated by products of at least \(k\)
positive-dimensional simplices.

There is a natural chain map
\[
\mu:F(C_*)\longrightarrow C_*(\Omega M)
\]
constructed as follows.  For a simplex
\[
\sigma:\Delta_q\to M
\]
with all vertices mapped to \(x_0\), choose a cubical family
\[
\theta_q:I^{q-1}\to \mathcal P(\Delta_q;v_0,v_q)
\]
of paths in the simplex from the first to the last vertex, compatible
with faces.  Then
\[
\mu([\sigma])
=
P(\sigma)\circ\theta_q
\]
for \(q>1\), while for \(q=1\) one subtracts the constant loop at
\(x_0\).  The boundary on \(F(C_*)\) is chosen so that \(\mu\) becomes
a chain map.  This is Adams' cobar model for the based loop space.

The composition
\[
\Theta\circ\mu:
F(C_*)\otimes\mathbb R
\longrightarrow
\mathbb R\langle\!\langle X_1,\ldots,X_m\rangle\!\rangle
\]
preserves the filtrations by word length and by powers of the
augmentation ideal.  The induced map on the associated spectral
sequences is an isomorphism from the \(E^1\)-term onward.  Passing to
\(E^\infty\), one obtains the isomorphism in
Theorem~\ref{chenpi1derham}.

\subsection{Example: the Heisenberg nilmanifold}

Let
\[
G={\rm Heis}_3(\mathbb R),\qquad
\Gamma={\rm Heis}_3(\mathbb Z),\qquad
M=\Gamma\backslash G.
\]
Then
\[
H_q(M,\mathbb Z)=
\begin{cases}
\mathbb Z,&q=0,3,\\
\mathbb Z\oplus\mathbb Z,&q=1,2,\\
0,&\text{otherwise}.
\end{cases}
\]
For Chen's standard left convention in this appendix, the left quotient
\(\Gamma\backslash G\) is used; the main text instead uses the right quotient
\(G/\Gamma\).  Take
\[
\omega_1=dx,\qquad \omega_2=dy,\qquad
\omega_{12}=dz-x\,dy.
\]
Then
\begin{equation}
d\omega_{12}=-\omega_1\wedge\omega_2.
\label{vanishingtwo}
\end{equation}
For the right convention of the main text, the central component is instead
\(\Omega_{12}=dz-y\,dx\), and
\(d\Omega_{12}=\Omega_1\wedge\Omega_2\).  Both formulas are correct in their
respective ordering conventions.  In particular, the class of
\(\omega_1\wedge\omega_2\) vanishes in \(H^2_{\rm DR}(M)\).

The second cohomology is generated by
\[
\omega_1\wedge\omega_{12},
\qquad
\omega_2\wedge\omega_{12}.
\]
Let \(X_1,X_2\) be the variables dual to
\(\omega_1,\omega_2\), and let \(Y_1,Y_2\) be the variables dual to
\(\omega_1\wedge\omega_{12}\) and
\(\omega_2\wedge\omega_{12}\), respectively.

The homology connection has the form
\[
\omega
=
\omega_1X_1
+
\omega_2X_2
+
\omega_1\wedge\omega_{12}Y_1
+
\omega_2\wedge\omega_{12}Y_2
+
\omega_{12}[X_1,X_2].
\]
Indeed, the relation~\eqref{vanishingtwo} gives the quadratic
correction
\[
\omega_{12}[X_1,X_2],
\]
and the flatness equation determines
\[
\delta X_1=\delta X_2=0,
\]
\[
\delta Y_1=[[X_1,X_2],X_1],
\qquad
\delta Y_2=[[X_1,X_2],X_2].
\]
Consequently, the completed group algebra of
\(\pi_1(M)=\Gamma\) is described as
\[
\mathbb R\widehat{\pi}_1(M)
\simeq
\mathbb R\langle\!\langle X_1,X_2\rangle\!\rangle/\mathcal N,
\]
where \(\mathcal N\) is the ideal generated by
\[
[[X_1,X_2],X_1],
\qquad
[[X_1,X_2],X_2].
\]

This example shows explicitly how the higher-order components of a
homology connection encode the nilpotent commutator relations.  For a
general stratified nilpotent group, the \(k\)-th order component of
the homology connection is similarly expressed as a finite sum of
forms multiplied by iterated Lie brackets of order \(k\).

\section{The Harper operator as a model example}
\label{Anotherproofwilkinson}

This appendix explains how the exact finite-dimensionalization used in
the main text appears in the classical analysis of the Harper operator.
The purpose is not to give a new theory of the Hofstadter butterfly,
but to show, in the simplest nonabelian example, how the passage
between finite-dimensional magnetic matrices and the Schr\"odinger
representation can be made before taking any semiclassical limit.
This is the analytic counterpart of the Heisenberg finite-dimensionalization
in Step H-1 and the restriction formula in Step H-2 of the companion
foundation paper.

The example is useful for comparison with the arguments in
Sections~\ref{reductionpathintegral}, \ref{sectionasym}, and
\ref{asymnilpotent}.  In the heat-kernel problem, the exact
decomposition is followed by perturbation theory of twisted
operators.  For the Harper operator, the same mechanism leads to the
Wilkinson expansion.

\subsection{The Harper operator and the Heisenberg group}
\label{Hofstadterbutterfly}

The Harper operator
\[
H_\theta:\ell^2(\mathbb Z^2)\longrightarrow \ell^2(\mathbb Z^2)
\]
is defined by
\begin{equation}
(H_\theta u)(m,n)
=
u(m+1,n)+u(m-1,n)
+
e^{\iu m\theta}u(m,n+1)
+
e^{-\iu m\theta}u(m,n-1).
\label{introharperfirst}
\end{equation}
It is the discrete magnetic Laplacian on the square lattice, up to
the usual shift by \(4I\)~\cite{SunadaMagnetic}, and may be regarded as a lattice model for
the Landau Hamiltonian on the plane.

The spectrum of \(H_\theta\), as a function of the magnetic flux
\(\theta\), gives the Hofstadter butterfly.  For rational flux the
spectrum has a band structure, while for irrational flux it is
related to the Ten Martini Problem, solved by Avila and
Jitomirskaya~\cite{Avila}.

The Hofstadter spectrum is the familiar butterfly-shaped subset of
\([0,2\pi]\times[-4,4]\); no external figure is needed for the arguments below.

The relation with the present work is through the discrete Heisenberg
group
\[
\Gamma={\rm Heis}_3(\mathbb Z).
\]
Let \(u,v\) be the standard generators and let \(w=[u,v]\) be the
central generator.  The Cayley graph of \((\Gamma,\{u,v\})\) is a
\(\Gamma\)-covering of the bouquet of two circles.  The intermediate
abelian covering is the square lattice \(\mathbb Z^2\).

The Cayley graph of \((\mathrm{Heis}_3(\mathbb Z),\{u,v\})\) projects to the
square lattice after quotienting by the center.  This covering relation is the
only graph-theoretic fact used here.

Thus the Harper operator appears as the magnetic operator induced on
the intermediate abelian covering.  For rational flux
\[
\theta=2\pi h,\qquad h=\frac pq,
\]
the corresponding finite-dimensional representation of
\({\rm Heis}_3(\mathbb Z)\) has dimension \(q\).  The operator
\(H_\theta\), or equivalently \(H_\theta-4I\), decomposes into
finite-dimensional magnetic matrices associated with the
representations
\[
\rho_{{\rm fin},x},
\qquad
x=(x_1,x_2,x_3)\in\widehat X.
\]
This is precisely the Heisenberg finite-dimensionalization of Step H-1,
together with the restriction formula of Step H-2, in the companion foundation
paper~\cite{KatsudaFB}.  In the present appendix we use this structure only as
an analytic model for the Harper operator.

We use the following normalization convention in this appendix.  The physical
magnetic phase is \(\theta\), while \(h=\theta/(2\pi)\) is the Heisenberg
central character parameter.  After a harmless unitary dilation of the
Schr\"odinger variable, the two basic unitaries may be written in a symmetric
small-parameter form.  Thus the formulas involving \(\sqrt\theta\) and those
involving \(\sqrt h\) differ only by this fixed dilation and the factor
\(2\pi\).

The twisted graph Laplacian and the Harper operator are related by this
intermediate abelian covering and the rational Heisenberg fiber decomposition.

\subsection{The Wilkinson expansion}

We are interested in the semiclassical regime
\[
\theta\to0.
\]
The classical Wilkinson expansion states that the lower spectral
edges of the Harper operator satisfy
\begin{equation}
E_n
=
-4+(2n+1)\theta+O(\theta^2),
\qquad
n=0,1,2,\ldots .
\label{introeigen1}
\end{equation}
This formula was first obtained by Wilkinson~\cite{Wilkinson} using
WKB methods, and related derivations were given by Rammal and
Bellissard~\cite{Rammal}.  A rigorous semiclassical treatment was
given by Helffer and Sj\"ostrand~\cite{Helffer,Helffer01}.

There is a standard formal derivation.  The operator \(H_\theta\) is
unitarily equivalent, near the bottom of the spectrum, to
\begin{equation}
h_\theta
=
-2\cos\left(\sqrt\theta\,\frac{1}{\iu}\frac{d}{ds}\right)
-
2\cos(\sqrt\theta\,s)
\label{introhtheta}
\end{equation}
on \(L^2(\mathbb R)\).  This expression is closely related to the
Schr\"odinger representation of
\({\rm Heis}_3(\mathbb R)\).

A formal Taylor expansion gives
\begin{equation}
h_\theta
=
-4
+
\theta
\left(
-\frac{d^2}{ds^2}+s^2
\right)
+
O(\theta^2).
\label{introformal}
\end{equation}
The coefficient of \(\theta\) is the harmonic oscillator
\[
\widetilde{\mathcal H}
=
-\frac{d^2}{ds^2}+s^2,
\]
whose eigenvalues are \(2n+1\).  This gives
\eqref{introeigen1} formally.

The point is that \eqref{introformal} is not a harmless operator
Taylor expansion: a bounded operator is being approximated by an
unbounded operator.  Moreover, for rational \(\theta\), the spectrum
of \(H_\theta\) consists of finitely many bands, and one must decide
which point in each band is represented by the expansion
\eqref{introeigen1}.  The edges of the Hofstadter spectrum are not
smooth in \(\theta\) at rational points in general.

The exact finite-dimensionalization gives another way to organize the
rational part of the justification.  The expression~\eqref{introhtheta} comes
from the Schr\"odinger representation \(\rho_h\), while for rational
\(h=p/q\) the corresponding Harper operator is decomposed into
finite-dimensional representations \(\rho_{{\rm fin},x}\).  These rational
finite matrices are exact Bloch fibers of the discrete problem.  The
Schr\"odinger model is then used as the smooth large-denominator normal form in
which the harmonic oscillator calculation is transparent.

This is the same philosophy as in the main text: the decomposition used in the
trace identity is exact on the rational finite-dimensional locus, and
asymptotics enter only afterwards in the perturbative analysis of the fiber
operators.

\subsection{Finite-dimensional magnetic matrices at rational flux}

Let
\[
\theta=2\pi x_1,
\qquad
x_1=\frac1q
\]
for simplicity.  The general coprime case \(x_1=p/q\) is obtained by
a relabeling of the cyclic basis.  For
\[
x=(x_1,x_2,x_3)\in\widehat X,
\]
the corresponding \(q\)-dimensional magnetic matrix is
\[
H_x
=
\rho_{{\rm fin},x}(u)+\rho_{{\rm fin},x}(u)^*
+
\rho_{{\rm fin},x}(v)+\rho_{{\rm fin},x}(v)^* .
\]
Explicitly,
\[
H_x
=
\begin{pmatrix}
c_0 & e^{2\pi\iu x_3/q} & 0 & \cdots & 0 & e^{-2\pi\iu x_3/q}\\
e^{-2\pi\iu x_3/q} & c_1 & e^{2\pi\iu x_3/q} & \cdots & 0 & 0\\
0 & e^{-2\pi\iu x_3/q} & c_2 & \ddots & 0 & 0\\
\vdots & \vdots & \ddots & \ddots & e^{2\pi\iu x_3/q} & \vdots\\
0 & 0 & 0 & e^{-2\pi\iu x_3/q} & c_{q-2} & e^{2\pi\iu x_3/q}\\
e^{2\pi\iu x_3/q} & 0 & 0 & \cdots & e^{-2\pi\iu x_3/q} & c_{q-1}
\end{pmatrix},
\]
where
\[
c_k
=
2\cos\left(2\pi\frac{x_2+k}{q}\right),
\qquad
k=0,\ldots,q-1.
\]
The spectrum of \(H_\theta\) at rational flux is given by
\begin{equation}
\sigma(H_\theta)
=
\bigcup_{x_2,x_3\in[0,1]}
\sigma(H_x).
\label{rational-band-decomp}
\end{equation}
This is the usual band decomposition of the Harper operator, but in
the present setting it is interpreted as the finite-dimensional Bloch
decomposition associated with the Heisenberg lattice.

Let
\[
P_x(E)
=
\det(E{\bf 1}_q-H_x).
\label{characteristic}
\]
With the sign-stable convention in~\eqref{characteristic}, Chambers' identity
has the form
\begin{equation}
P_{(x_1,x_2,x_3)}(E)
=
P_{(x_1,0,0)}(E)+4
-2\cos(2\pi x_2)-2\cos(2\pi x_3).
\label{level}
\end{equation}
Equivalently, \(E\) belongs to the spectrum of a fiber precisely when
\[
P_{(x_1,0,0)}(E)+4
=2\bigl(\cos(2\pi x_2)+\cos(2\pi x_3)\bigr).
\]
If instead one writes \(F_x(E)=\det(H_x-EI_q)\), then
\(F_x=(-1)^qP_x\), and every phase-dependent term acquires the same parity
factor.  Omitting it gives the wrong sign for odd \(q\).

Equation~\eqref{level} describes the motion of the spectral bands as
\((x_2,x_3)\) varies.  The following standard Chambers-formula argument is
included only to indicate where the exact finite-dimensional fibers enter the
non-perturbative band-width analysis; the full exponential estimates belong to
the classical Harper/almost-Mathieu literature~\cite{Bellissard} and are not reproved here.  Once
the leading asymptotic location of the roots \(E_i\) is known from the
Wilkinson expansion, the characteristic polynomial
\[
P_{(x_1,0,0)}(E)
=
(E_1-E)\cdots(E_q-E)
\]
shows that
\[
\left|P'_{(x_1,0,0)}(E_i)\right|
=
\prod_{j\neq i}|E_j-E_i|.
\]
Low-level harmonic-oscillator spacing alone does not control the full product
over all \(q-1\) roots, nor does a derivative at one point give a uniform root
localization.  Exponentially small rational band widths require the uniform
Gauss/Chambers-polynomial estimates from the Harper literature, including
simplicity and derivative control on the relevant interval.  The present
appendix records compatibility with those results; it does not derive them
from finite-dimensionalization.

\subsection{Comparison with the Helffer--Sj\"ostrand method}

We recall, very briefly, how the Wilkinson formula is justified in the
semiclassical approach of Helffer and Sj\"ostrand~\cite{Helffer,Helffer01}.
The formal expansion
\[
h_\theta
=
-4+\theta\left(-\frac{d^2}{ds^2}+s^2\right)+O(\theta^2)
\]
suggests the harmonic oscillator and hence the leading term
\(-4+(2n+1)\theta+O(\theta^2)\).  The difficulty is that a bounded operator is
being compared with an unbounded harmonic oscillator, so the Taylor expansion
is not an operator-norm expansion on \(L^2(\mathbb R)\).

A useful way to understand their argument is the following.  If the common
domain of \(h_\theta\) and the harmonic oscillator were exhausted by an
increasing sequence of finite-dimensional subspaces \(V_k\) invariant under
both operators, then the formal Taylor expansion could be justified by first
restricting to \(V_k\) and then passing to the limit.  Such common invariant
subspaces do not exist.  Helffer and Sj\"ostrand replace them by the spaces
\(\widetilde V_k\) spanned by the harmonic-oscillator eigenfunctions with
eigenvalues up to a fixed cutoff.  These spaces are not invariant under
\(h_\theta\), but semiclassical localization shows that the leakage errors are
\(O(h^\infty)\), with \(h=\theta/(2\pi)\), and therefore do not change the
power-series expansion.

This localization is analogous to the familiar double-well picture.  Low-energy
eigenfunctions are concentrated near the bottoms of the wells; their tails are
not compactly supported, but the interaction outside the relevant well is
exponentially small, or \(O(h^\infty)\), and hence is non-perturbative.  For
the Harper symbol
\[
-2\cos(\sqrt\theta\,\xi)-2\cos(\sqrt\theta\,s),
\]
there are infinitely many wells, and the local model near each bottom is the
harmonic oscillator.  Helffer and Sj\"ostrand make this localization argument
precise.

The representation-theoretic viewpoint used here is different.  It does not
replace the Helffer--Sj\"ostrand proof of the Wilkinson expansion.  Rather, it
explains how, at rational flux, the finite-dimensional magnetic matrices arise
as exact Heisenberg fibers before the semiclassical limit is taken, and how
the rational large-denominator regime has the smooth Schr\"odinger normal form
as its limiting model.  In this picture the two basic unitary operators
\[
e^{\sqrt h\,\frac{1}{\iu}\frac{d}{ds}},
\qquad
e^{2\pi\iu\sqrt h\,s}
\]
can be followed separately before being combined into the harmonic oscillator.
This separation is one of the useful features of the representation-theoretic
method and is the model for the higher nilpotent arguments in the main text.

\subsection{Relation with Landau levels}
\label{meshapproximation}

We finally recall how the Wilkinson expansion approximates the Landau
levels of the magnetic Laplacian on the plane.

Let \(\Delta_B\) be the magnetic Laplacian on \(\mathbb R^2\) with
constant magnetic field of flux \(B\).  For \(B\neq0\), the spectrum
consists of the Landau levels, which coincide with the eigenvalues of
a harmonic oscillator up to normalization.

Approximate \(\mathbb R^2\) by the square lattice
\(\delta\mathbb Z^2\) of mesh length \(\delta\).  The flux through one plaquette
is then
\[
\theta=B\delta^2
\]
(up to the same \(2\pi\)-normalization convention as above).  With our sign
convention \(H_\theta\) is the magnetic adjacency operator, and the discrete
magnetic Laplacian is represented near the bottom of the spectrum by
\[
\Delta_{B,\delta}
\simeq
\delta^{-2}(4I-H_\theta).
\]
The lower-edge Wilkinson formula \eqref{introeigen1} for \(H_\theta\) is
equivalently, by the symmetry of the Harper spectrum, the upper-edge formula
\[
E_n^+
=
4-(2n+1)\theta+O(\theta^2).
\]
Therefore
\[
\lambda_n
=
\frac{4-E_n^+}{\delta^2}
=
(2n+1)B+o(1).
\]
Consequently, once the cited uniform Wilkinson expansion is invoked, the
corresponding low-lying lattice levels converge in this mesh limit to the
Landau levels of the magnetic Laplacian on the plane.

The exponentially small band widths at rational flux are consistent
with the infinite degeneracy of the Landau levels in the continuum
limit.


\address{ 
Research and Education Center for Natural Science \\
Hiyoshi Campus, Keio University \\
4-1-1, Hiyoshi, Kohoku-ku \\
 Yokohama 223-8521, Japan \\
\smallskip \\
Osaka Central Advanced Mathematical Institute (OCAMI) \\
MEXT Joint Usage/Research Center on Mathematics and Theoretical Physics, \\
Osaka Metropolitan University \\
3-3-138 Sugimoto, Sumiyoshi-ku \\
Osaka 558-8585, Japan \\
\smallskip \\
Faculty of Mathematics \\
Kyushu University \\
744 Motooka, Nishi-ku,\\
Fukuoka 819-0395, Japan
}
{katsuda@math.kyushu-u.ac.jp}

\end{document}